\documentclass[11pt,leqno]{article}
\usepackage{amsfonts}
\usepackage{amsmath}
\usepackage{amssymb}
\usepackage{amsthm}
\usepackage{latexsym}
\usepackage{mathrsfs}
\usepackage{mathtools}
\usepackage{bm}
\usepackage{hyperref}

\usepackage{graphicx}
\usepackage{tikz} 
\usetikzlibrary{fit}
\usetikzlibrary{positioning}

\newtheorem{theorem}{\bf Theorem}[section]
\newtheorem{lemma}[theorem]{\bf Lemma}

\newtheorem{proposition}[theorem]{\bf Proposition}
\newtheorem{remark}[theorem]{\bf Remark}

\numberwithin{equation}{section}
\numberwithin{figure}{section}

\begin{document}
\vspace*{0ex}
\begin{center}
{\Large\bf
A mathematical analysis of the Kakinuma model \\ 
for interfacial gravity waves. \\
Part II: Justification as a shallow water approximation \\
}
\end{center}

\begin{center}
Vincent Duch\^ene and Tatsuo Iguchi
\end{center}

\begin{abstract}
We consider the Kakinuma model for the motion of interfacial gravity waves.
The Kakinuma model is a system of Euler--Lagrange equations for an approximate Lagrangian, 
which is obtained by approximating the velocity potentials in the Lagrangian of the full model. 
Structures of the Kakinuma model and the well-posedness of its initial value problem were analyzed in the 
companion paper~\cite{DucheneIguchi2020}. 
In this present paper, we show that the Kakinuma model is a higher order shallow water approximation to 
the full model for interfacial gravity waves with an error of order $O(\delta_1^{4N+2}+\delta_2^{4N+2})$ 
in the sense of consistency, where $\delta_1$ and $\delta_2$ are shallowness parameters, 
which are the ratios of the mean depths of the upper and the lower layers to the typical horizontal wavelength, 
respectively, and $N$ is, roughly speaking, the size of the Kakinuma model and can be taken an arbitrarily large number. 
Moreover, under a hypothesis of the existence of the solution to the full model with a uniform bound, 
a rigorous justification of the Kakinuma model is proved by giving an error estimate between the solution 
to the Kakinuma model and that of the full model. 
An error estimate between the Hamiltonian of the Kakinuma model and that of the full model is also provided. 
\end{abstract}

\section{Introduction}
We will consider the motion of the interfacial gravity waves at the interface between two layers 
of immiscible fluids in $(n+1)$-dimensional Euclidean space. 
Let $t$ be the time, $\bm{x}=(x_1,\ldots,x_n)$ the horizontal spatial coordinates, and $z$ the vertical 
spatial coordinate. 
We assume that the layers are infinite in the horizontal directions, bounded from above by a flat rigid-lid, 
and from below by a time-independent variable topography. 
The interface, the rigid-lid, and the bottom are represented as $z=\zeta(\bm{x},t)$, $z=h_1$, 
and $z=-h_2+b(\bm{x})$, respectively, where $\zeta=\zeta(\bm{x},t)$ is the elevation of the interface, 
$h_1$ and $h_2$ are mean depths of the upper and lower layers, 
and $b=b(\bm{x})$ represents the bottom topography. 
See Figure~\ref{intro:interfacial wave}. 
\begin{figure}[ht]
\setlength{\unitlength}{1pt}
\begin{picture}(0,0)
\put(95,-104){$\bm{x}$}
\put(77,-72){$z$}
\put(170,-70){$\zeta(\bm{x},t)$}
\put(375,-65){$h_1$}
\put(375,-145){$h_2$}
\put(280,-65){$\rho_1$}
\put(280,-145){$\rho_2$}
\end{picture}
\begin{center}
\includegraphics[width=0.7\linewidth]{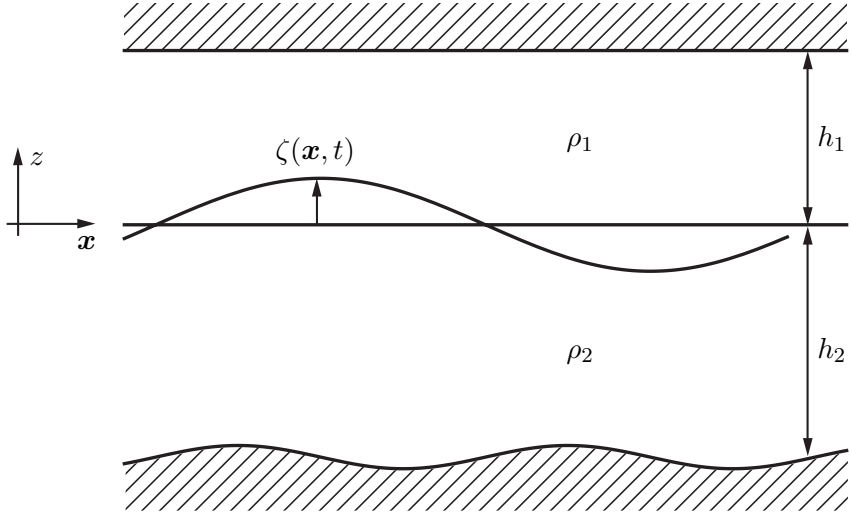}
\end{center}
\caption{Internal gravity waves}
\label{intro:interfacial wave}
\end{figure}
We assume that the fluids in the upper and the lower layers are both incompressible and inviscid fluids 
with constant densities $\rho_1$ and $\rho_2$, respectively, and that the flows are both irrotational. 
Then, the motion of the fluids is described by the velocity potentials $\Phi_1(\bm{x},z,t)$ and $\Phi_2(\bm{x},z,t)$ 
and the pressures $P_1(\bm{x},z,t)$ and $P_2(\bm{x},z,t)$ in the upper and the lower layers. 
We recall the governing equations, referred as the full model for interfacial gravity waves, in Section~\ref{S.equations} below. 
Generalizing the work of J. C. Luke~\cite{Luke1967}, these equations can be obtained as the Euler--Lagrange equations 
associated with the Lagrangian density $\mathscr{L}(\Phi_1,\Phi_2,\zeta)$ given by the vertical integral 
of the pressure in both water regions. 
Building on this variational structure, T. Kakinuma~\cite{Kakinuma2000,Kakinuma2001, Kakinuma2003} proposed 
and studied numerically the model obtained as the Euler--Lagrange equations for an approximated Lagrangian density, 
$\mathscr{L}(\Phi_1^\mathrm{app},\Phi_2^\mathrm{app},\zeta)$, where
\begin{equation}\label{intro:app}
\Phi_\ell^\mathrm{app}(\bm{x},z,t) = \sum_{i=0}^{N_\ell} Z_{\ell,i}(z;\tilde{h}_\ell(\bm{x}))\phi_{\ell,i}(\bm{x},t)
\end{equation}
for $\ell=1,2$, and $\{Z_{1,i}\}$ and $\{Z_{2,i}\}$ are appropriate function systems 
in the vertical coordinate $z$ and may depend on 
$\tilde{h}_1(\bm{x})$ and $\tilde{h}_2(\bm{x})$, 
respectively, which are the depths of the upper and the lower layers in the rest state, 
whereas $\bm{\phi}_\ell=(\phi_{\ell,0},\phi_{\ell,1},\ldots,\phi_{\ell,N_\ell})^\mathrm{T}$, $\ell=1,2$, are unknown variables. 
This yields a coupled system of equations for $\bm{\phi}_1$, $\bm{\phi}_2$, and $\zeta$, 
depending on the function systems $\{Z_{1,i}\}$ and $\{Z_{2,i}\}$, which we named Kakinuma model. 
Note that in our setting of the problem we have $\tilde{h}_1(\bm{x})=h_1$ and $\tilde{h}_2(\bm{x})=h_2-b(\bm{x})$. 
In this work we study the Kakinuma model obtained when the approximate velocity potentials
are defined by
\begin{equation}\label{intro:appk}
	\begin{cases}
		\displaystyle
		\Phi_1^\mathrm{app}(\bm{x},z,t) := \sum_{i=0}^N (-z+h_1)^{2i}\phi_{1,i}(\bm{x},t), \\[2.5ex]
		\displaystyle
		\Phi_2^\mathrm{app}(\bm{x},z,t) := \sum_{i=0}^{N^*} (z+h_2-b(\bm{x}))^{p_i}\phi_{2,i}(\bm{x},t),
	\end{cases}
\end{equation}
where $N, N^*$, and $p_0,p_1,\ldots,p_{N^*}$ are nonnegative integers satisfying ${0=p_0<p_1<\cdots<p_{N^*}}$.
Specifically, we show that the Kakinuma model obtained through the approximated potentials~\eqref{intro:appk} with 
\begin{enumerate}
	\item[(H1)]
	$N^*=N$ and $p_i=2i$ $(i=0,1,\ldots,N)$ in the case of the flat bottom $b(\bm{x})\equiv0$, 
	\item[(H2)]
	$N^*=2N$ and $p_i=i$ $(i=0,1,\ldots,2N)$ in the case with general bottom topographies, 
\end{enumerate}
provides a higher order shallow water approximation to the full model for interfacial gravity waves in the strongly nonlinear regime. 
The choice of the function systems as well as $N, N^*$, and $p_0,p_1,\ldots,p_{N^*}$ is discussed and motivated later on.

\paragraph{Comparison with surface gravity waves}
The Kakinuma model is an extension to interfacial gravity waves of the so-called 
Isobe--Kakinuma model for surface gravity waves, that is, water waves, 
in which Luke's Lagrangian density $\mathscr{L}_\mathrm{Luke}(\Phi,\zeta)$, 
where $\zeta$ is the surface elevation and $\Phi$ is the velocity potential of the water, 
is approximated by a density $\mathscr{L}^\mathrm{app}(\bm{\phi},\zeta)=\mathscr{L}_\mathrm{Luke}(\Phi^\mathrm{app},\zeta)$, where 
\begin{equation}\label{intro:appww}
	\Phi^\mathrm{app}(\bm{x},z,t) = \sum_{i=0}^N Z_i(z;b(\bm{x}))\phi_{i}(\bm{x},t). 
\end{equation}
The Isobe--Kakinuma model was first proposed by M.~Isobe~\cite{Isobe1994,Isobe1994-2} 
and then applied by T.~Kakinuma  to simulate numerically the water waves. 
Recently, this model was analyzed from a mathematical point of view when 
the function system $\{Z_i\}$ is a set of polynomials in $z$:  
$Z_i(z;b(\bm{x}))=(z+h-b(\bm{x}))^{p_i}$ with integers $p_i$ satisfying $0=p_0<p_1<\cdots<p_N$. 
The initial value problem was analyzed by Y.~Murakami and T.~Iguchi~\cite{MurakamiIguchi2015} in a special case 
and by R.~Nemoto and T.~Iguchi~\cite{NemotoIguchi2018} in the general case. 
The hypersurface $t=0$ in the space-time $\mathbf{R}^n\times\mathbf{R}$ is characteristic for 
the Isobe--Kakinuma model in the sense that the operator acting on time derivatives of the unknowns has a non-trivial kernel. 
As a consequence, one needs to impose some compatibility conditions on the initial data 
for the existence of the solution. 
Under these compatibility conditions, the non-cavitation condition, and a Rayleigh--Taylor type condition 
$-\partial_z P^\mathrm{app} \geq c_0>0$ on the water surface, where $P^\mathrm{app}$ 
is an approximate pressure in the Isobe--Kakinuma model calculated from Bernoulli's equation, 
they showed the well-posedness of the initial value problem in Sobolev spaces locally in time. 
Moreover, T.~Iguchi~\cite{Iguchi2018-1,Iguchi2018-2} showed that under the choice of the function system 
\begin{equation}\label{intro:base}
	Z_i(z;b(\bm{x})) = 
	\begin{cases}
		(z+h)^{2i} & \mbox{in the case of the flat bottom}, \\
		(z+h-b(\bm{x}))^i & \mbox{in the case of a variable bottom},
	\end{cases}
\end{equation}
the Isobe--Kakinuma model is a higher order shallow water approximation for the water wave problem 
in the strongly nonlinear regime. 
Furthermore, V.~Duch\^ene and T.~Iguchi~\cite{DucheneIguchi2019-1} showed that the Isobe--Kakinuma model 
also enjoys a Hamiltonian structure analogous to the one exhibited by V.~E.~Zakharov~\cite{Zakharov1968} 
on the full water wave problem and that the Hamiltonian of the Isobe--Kakinuma model is a higher order 
shallow water approximation to the one of the full water wave problem.

Our aim in the present paper and the companion paper~\cite{DucheneIguchi2020} is to extend these results 
on surface gravity waves to the framework of interfacial gravity waves. 
With respect to surface gravity waves, our interfacial gravity waves framework brings two additional difficulties.
The first one is that, due to the rigid-lid assumption, the full system for interfacial gravity waves 
described in Section~\ref{S.equations} features only one evolution equation for the two velocity potentials,
and a constraint associated with the fixed fluid domain. From a physical perspective, the unknown velocity potential 
at the interface may be interpreted as a Lagrange multiplier associated with the constraint.
A second important difference between water waves and interfacial gravity waves is that the latter suffer from
Kelvin--Helmholtz instabilities. As a consequence the initial value problem of the full model for interfacial 
gravity waves is ill-posed in Sobolev spaces; 
see T.~Iguchi, N.~Tanaka, and A.~Tani~\cite{IguchiTanakaTani1997}, V.~Kamotski and G.~Lebeau~\cite{KamotskiLebeau2005}. 
This raises the question of the validity of {\em any} model for interfacial gravity waves. A partial answer is offered
by the work of D. Lannes~\cite{Lannes2013}, which proves the existence and uniqueness of solutions over large time intervals
in the presence of interfacial tension. While interfacial tension effects is not expected to be the relevant regularization mechanism 
for the propagation of waves between, for instance, fresh and salted water, the key observation is that physical systems allow 
the propagation of waves with large amplitude and long wavelengths provided that some mechanism tames Kelvin--Helmholtz instabilities
acting on the high-frequency component of the flow. This description is consistent with the fact that the initial value problem 
of the bi-layer shallow water system for the propagation of interfacial gravity waves in the hydrostatic framework is well-posed 
in Sobolev spaces under some hyperbolicity condition describing the absence of low-frequency Kelvin--Helmholtz instabilities, as
proved by D.~Bresch and M.~Renardy~\cite{BreschRenardy11}. Let us mention however that such a property is not automatic for 
higher order shallow water models. Specifically, we note that the Miyata--Choi--Camassa model derived by M.~Miyata~\cite{Miyata85} and 
W.~Choi and R.~Camassa~\cite{ChoiCamassa99} and which can be regarded as a two-layer generalization of the Green--Naghdi equations for water waves 
turns out to overestimate Kelvin--Helmholtz estimates with respect to the full model; see D.~Lannes and M.~Ming~\cite{LannesMing15}.

In~\cite{DucheneIguchi2020}, we analyzed the initial value problem of the Kakinuma model 
when the approximated velocity potentials are defined by~\eqref{intro:appk}. 
We found that the Kakinuma model has a 
stability regime which can be expressed as 
\begin{equation}\label{intro:stability}
	- \partial_z (P_2^\mathrm{app} - P_1^\mathrm{app} )
	- \frac{\rho_1\rho_2}{\rho_1H_2\alpha_2 + \rho_2H_1\alpha_1} 
	|\nabla\Phi_2^\mathrm{app} - \nabla\Phi_1^\mathrm{app}|^2  \geq c_0 > 0
\end{equation}
on the interface, where 
$H_1:= h_1 - \zeta$ and $H_2:=h_2 + \zeta - b$ are the depths of the upper and the lower layers,
$P_1^\mathrm{app}$ and $P_2^\mathrm{app}$ are approximate pressures of the fluids in the 
upper and the lower layers, $\alpha_1$ and $\alpha_2$ are positive constants depending only on $N$ and 
on $p_0,p_1,\ldots,p_{N^*}$, respectively. 
This is a generalization of the aforementioned Rayleigh--Taylor type condition for the Isobe--Kakinuma model. 
It is worth noticing that, consistently with the expectation that the Kakinuma model is a higher order model
for the full system for interfacial gravity waves and that the latter suffers from Kelvin--Helmholtz instabilities,
 the constants $\alpha_1$ and $\alpha_2$ converge to $0$ as $N$ and $N^*$ go to infinity 
so that the stability condition becomes more and more stringent as $N$ and $N^*$ grow. 
When $N=N^*=0$, the Kakinuma model coincides with the aforementioned bi-layer shallow water system, and the 
stability regime coincides with the hyperbolic domain exhibited in~\cite{BreschRenardy11}. 
Moreover, when the motion of the fluids together with the motion of the interface is in the rest state, 
the above stability condition is reduced to the well-known stable stratification condition 
\begin{equation}\label{Rayleigh's SC}
	(\rho_2-\rho_1)g>0. 
\end{equation}
In~\cite{DucheneIguchi2020}, we showed that under the stability condition~\eqref{intro:stability}, 
the non-cavitation assumptions
\begin{equation}\label{intro:non-cavitation}
	H_1\geq c_0>0, \qquad H_2\geq c_0>0,
\end{equation}
and intrinsic compatibility conditions on the initial data, 
the initial value problem for the Kakinuma model is well-posed in Sobolev spaces locally in time. 
We also showed in~\cite{DucheneIguchi2020} that the Kakinuma model enjoys a Hamiltonian structure analogous to 
the one exhibited by T.~B.~Benjamin and T.~J.~Bridges~\cite{BenjaminBridges1997} on the full model for 
interfacial gravity waves.

\paragraph{Comparison with other higher order models}
The Isobe--Kakinuma and the Kakinuma models belong to higher order models for the water waves and for the full interfacial gravity waves, respectively. 
By this we mean a family of systems of equations parametrized by nonnegative integers describing the order of the system within the family, 
that is $N$ for the Isobe--Kakinuma model,and whose solutions are expected to approach solutions to the full system as the order increases. 
Several such models have been introduced in the literature, mostly in the water waves framework, and we will restrict the discussion to water waves in this paragraph. 

Based on a Taylor expansion of the Dirichlet-to-Neumann operator at stake in the water waves system with respect to the shape of the domain D.~G.~Dommermuth and D.~K.~Yue\cite{DommermuthYue87}, B.~J.~West et al.~\cite{WestBruecknerJandaEtAl87} and W.~Craig and C.~Sulem~\cite{CraigSulem93} have proposed the so-called high order spectral (HOS) models. 
While these models have been successfully employed in efficient numerical schemes (see recent accounts by J.~Wilkening and V.~Vasan~\cite{WilkeningVasan15}, D.~P.~Nicholls~\cite{Nicholls16} and P.~Guyenne~\cite{Guyenne19}), the equations feature Fourier multipliers which prevent their direct use in situations involving non-trivial geometries such as horizontal boundaries. 
Moreover, the rigorous justification of HOS models is challenged by well-posedness issues; see the discussion in D.~M.~Ambrose, J.~L.~Bona, and D.~P.~Nicholls~\cite{AmbroseBonaNicholls14}, and V.~Duchêne and B.~Melinand~\cite{DucheneMelinand}. 

A second class of higher order models originate from formal shallow water expansions put forward by J. Boussinesq~\cite{Boussinesq73} and J. W. S. Rayleigh~\cite{Rayleigh76}. A systematic derivation procedure has been described by K. O. Friedrichs in the appendix to~\cite{Stoker48}. Recently, these higher order shallow water models have been described and discussed by Y. Matsuno in~\cite{Matsuno15,Matsuno16} and W. Choi in~\cite{Choi19,Choi22}. 
The derivation procedure displays formula for approximate velocity potentials under the form~\eqref{intro:appww}--\eqref{intro:base}  (in particular, only even powers appear in the flat bottom case), 
with the important difference that the functions $\phi_{i}$ ($i=0,\dots,N$) are prescribed through explicit recursion relations. An important consequence of this derivation is that the resulting systems of equations involve only standard differential operators. However the order of the differential operators at stake augments with the order of the system, which renders such models impractical for numerical simulations.

By contrast, the Isobe--Kakinuma model features only differential operators of order at most two acting on the variables $\phi_{i}$ ($i=0,\dots,N$) which are unknowns of the system. Notice that the size of the system augments with its order, $N$. However the degrees of freedom do not augment with the order since, as mentioned above, some compatibility conditions must be satisfied. In fact all quantities are uniquely determined by two scalar functions which represent the canonical variables in the Hamiltonian formulation of the water waves system.
Let us mention that function systems different from~\eqref{intro:base} have been considered by G.~A.~Athanassoulis and K.~A.~Belibassakis~\cite{AthanassoulisBelibassakis99}, G.~Klopman, B.~van Groesen, and M.~W.~Dingemans~\cite{KlopmanGroesenDingemans10} and C.~E.~Papoutsellis and G.~A.~Athanassoulis~\cite{PapoutsellisAthanassoulis17} (see also references therein). 
While the systems obtained in these works have a similar nature, they are all different. 
We let the reader refer to V.~Duchêne~\cite[Chapter~D]{MM4WW} for an extended discussion and comparison of these models. 

The choice of the function systems in~\eqref{intro:base} is motivated by the aforementioned Friedrichs expansion and is essential in the analysis of 
T.~Iguchi~\cite{Iguchi2018-1,Iguchi2018-2} proving that the Isobe--Kakinuma model is a higher order shallow water approximation for the water wave problem 
in the strongly nonlinear regime. 
We note that one may modify~\eqref{intro:base} by putting all odd and even terms $(z+h)^{i}$ for $i=0,1,\ldots$ in the case of the flat bottom. 
However, in that case one needs to use the terms up to order $2N$ to keep the same precision of the approximation. 
Therefore, such a choice increases the number of unkonwns and equations by $N$ so that it is undesirable for practical application. 
In other words, one can save memories in numerical simulations by using only even terms in the case of the flat bottom. 
On the contrary, if we put only odd terms $(z+h-b(\bm{x}))^{2i}$ for $i=0,1,2,\ldots$ in the case of a non-flat bottom, 
then the corresponding Isobe--Kakinuma model does not give any good approximation even if we take $N$ a sufficiently large number, 
because the corresponding approximate velocity potential $\Phi^\mathrm{app}$ cannot approximate the boundary condition on the bottom so well 
due to the lack of odd order terms $(z+h-b(\bm{x}))^{2i+1}$ for $i=0,1,2,\ldots$.

Following this discussion, the choice of the function systems~\eqref{intro:appk} with (H1) or (H2) in our interfacial waves framework is very natural. 
In particular, the rigid-lid is assumed to be flat so that we do not need to use odd order terms $(-z+h_1)^{2i+1}$ for $i=0,1,2,\ldots$, 
in the approximate velocity potential $\Phi_1^\mathrm{app}$ to obtain a good approximation, 
because $\Phi_1^\mathrm{app}$ can approximate the boundary condition on the rigid-lid without such terms.

\paragraph{Description of the results}

In the present paper we show that the Kakinuma model obtained through the 
approximated potentials~\eqref{intro:appk} with 
\begin{enumerate}
\item[(H1)]
$N^*=N$ and $p_i=2i$ $(i=0,1,\ldots,N)$ in the case of the flat bottom $b(\bm{x})\equiv0$, 
\item[(H2)]
$N^*=2N$ and $p_i=i$ $(i=0,1,\ldots,2N)$ in the case with general bottom topographies, 
\end{enumerate}
provides a higher order shallow water approximation to the full model for interfacial gravity waves in the strongly nonlinear regime. 
Our results apply to the dimensionless Kakinuma model obtained after suitable rescaling. The system of equations then depend on the 
positive dimensionless parameters  $\delta_1$ and $\delta_2$ which are shallowness parameters related to 
the upper and the lower layers, respectively, that is, $\delta_\ell=\frac{h_\ell}{\lambda}$ $(\ell=1,2)$ 
with the typical horizontal wavelength $\lambda$. The shallow water regime is described through the smallness of the parameters $\delta_1$ and $\delta_2$.
What is more, our results are uniform with respect to parameters satisfying either $\rho_2\lesssim \rho_1<\rho_2$, or $\rho_1\ll\rho_2$ and $h_2\lesssim h_1$.
We notice that the rigid-lid framework is expected to be invalid in the regime $\rho_1\ll\rho_2$ and $h_1\ll h_2$ which is excluded in this paper; see V.~Duchêne~\cite{Duchene12}.

Our first result extends the result of~\cite{DucheneIguchi2020} on the well-posedness of the initial value problem
by showing that solutions to the dimensionless Kakinuma model
are defined on a time interval which does not vanish for arbitrarily small values of $\delta_1$ and $\delta_2$.

\begin{theorem}[Long-time well-posedness]\label{T.uniform}
	Under the (dimensionless) stability condition~\eqref{intro:stability}, 
	the (dimensionless) non-cavitation assumptions~\eqref{intro:non-cavitation}, and intrinsic compatibility conditions on the initial data, 
	the initial value problem for the Kakinuma model is well-posed in Sobolev spaces on a time interval which is 
	independent of $\delta_1\in(0,1]$ and $\delta_2\in(0,1]$. 
\end{theorem}

While the non-cavitation assumption and the stability condition are automatically satisfied for small initial data 
and small bottom topography $b$, an arrangement of nontrivial initial data satisfying the compatibility 
conditions with suitable bounds is a non-trivial issue, and demands a specific analysis.

\begin{proposition}\label{P.compatibility}
	Initial data satisfying the compatibility conditions and necessary bounds in Theorem~\ref{T.uniform} are uniquely 
	determined (up to an additive constant) by sufficiently regular initial data for the canonical variables of the Hamiltonian structure.
\end{proposition}
Then, we show that under the special choice of the indices $p_0,p_1,\ldots,p_{N^*}$ as in (H1) or (H2), 
the dimensionless Kakinuma model is consistent with the full model for interfacial gravity waves with 
an error of order $O(\delta_1^{4N+2}+\delta_2^{4N+2})$.
\begin{theorem}[Consistency]\label{T.consistency}
	Assume (H1) or (H2).
	The solutions to the dimensionless Kakinuma model constructed in Theorem~\ref{T.uniform} produce functions that satisfy approximately 
	the dimensionless full interfacial gravity waves system up to error terms of size $O(\delta_1^{4N+2}+\delta_2^{4N+2})$. 
	
	Conversely, solutions to the dimensionless full interfacial gravity waves system satisfying suitable uniform bounds produce 
	through Proposition~\ref{P.compatibility} functions that satisfy approximately the dimensionless Kakinuma model 
	up to error terms of size $O(\delta_1^{4N+2}+\delta_2^{4N+2})$. 
\end{theorem}
In the last result we assume the existence of a solution to the full model with a uniform bound since 
for general initial data in Sobolev spaces, one cannot expect to construct a solution to the initial value problem,
due to the ill-posedness of the problem discussed previously. The same issue arises for the full justification of the Kakinuma model.
\begin{theorem}[Full justification]\label{T.justification}
	Assuming the existence of a solution to the dimensionless full interfacial gravity waves system with a uniform bound and satisfying initially
	the (dimensionless) stability condition~\eqref{intro:stability} and (dimensionless) non-cavitation assumptions~\eqref{intro:non-cavitation}, 
	then the Kakinuma model with (H1) or (H2) and appropriate initial data produces an approximate solution with the error estimate 
	\[
	|\zeta^{\mbox{\rm\tiny K}}(\bm{x},t)-\zeta^{\mbox{\rm\tiny IW}}(\bm{x},t)| \lesssim \delta_1^{4N+2}+\delta_2^{4N+2}
	\]
	on some time interval independent of $\delta_1\in(0,1]$ and $\delta_2\in(0,1]$, where $\zeta^{\mbox{\rm\tiny K}}$ and 
	$\zeta^{\mbox{\rm\tiny IW}}$ are solutions to the dimensionless Kakinuma model and to the full model, respectively. 
\end{theorem}
In our last main result, we show that the Hamiltonian structure of the Kakinuma model is a shallow water approximation of
the Hamiltonian structure of the full interfacial gravity waves model. 
\begin{theorem}[Hamiltonians]\label{T.Hamiltonian}
	Assume (H1) or (H2). Under appropriate assumptions on the canonical variables $(\zeta,\phi)$, we have
	\[
	|\mathscr{H}^{\mbox{\rm\tiny K}}(\zeta,\phi)-\mathscr{H}^{\mbox{\rm\tiny IW}}(\zeta,\phi)|
	\lesssim \delta_1^{4N+2}+\delta_2^{4N+2}, 
	\]
	where $\mathscr{H}^{\mbox{\rm\tiny K}}$ and $\mathscr{H}^{\mbox{\rm\tiny IW}}$ are the Hamiltonians of the dimensionless 
	Kakinuma model and of the dimensionless full interfacial gravity waves model, respectively.  
\end{theorem}

\begin{remark}
	The precise statements of our main results are displayed in Section~\ref{S.main-results}. Specifically, Theorem~\ref{T.uniform} corresponds to Theorem~\ref{theorem-uniform}, Proposition~\ref{P.compatibility} corresponds to Proposition~\ref{preparation-ini}, Theorem~\ref{T.consistency} corresponds to Theorem~\ref{theorem-consistency1} and Theorem~\ref{theorem-consistency2} (see also Remark~\ref{R.app-sol}), Theorem~\ref{T.justification} corresponds to Theorem~\ref{theorem-justification}, and Theorem~\ref{T.Hamiltonian} corresponds to Theorem~\ref{theorem-Hamiltonian}. 
\end{remark}

\paragraph{Structures of the Kakinuma model}

In order to obtain our main results, we exploit several structures of the Kakinuma model.
The Kakinuma model can be written compactly as 
\begin{equation}\label{intro:Kakinuma-compact} 
	\begin{cases}
		\displaystyle
		{\bm l}_1(H_1)\partial_t\zeta + L_1(H_1){\bm \phi}_1 = {\bm 0}, \\
		\displaystyle
		{\bm l}_2(H_2)\partial_t\zeta - L_2(H_2,b){\bm \phi}_2 = {\bm 0}, \\
		{\rho}_1\bigl\{ {\bm l}_1(H_1) \cdot \partial_t{\bm \phi}_1 
		+ \frac12\bigl( |{\bm u}_1|^2 +  w_1^2 \bigr) \bigr\} \\
		\quad
		- {\rho}_2\bigl\{ {\bm l}_2(H_2) \cdot \partial_t{\bm \phi}_2 
		+ \frac12\bigl( |{\bm u}_2|^2 + w_2^2 \bigr) \bigr\} 
		+ (\rho_1-\rho_2)g\zeta = 0,
	\end{cases}
\end{equation}
where we denote ${\bm \phi}_1 := (\phi_{1,0},\phi_{1,1},\ldots,\phi_{1,N})^\mathrm{T}$, 
${\bm \phi}_2 := (\phi_{2,0},\phi_{2,1},\ldots,\phi_{2,N^*})^\mathrm{T}$, 
put ${\bm l}_1(H_1) := (1,H_1^2,H_1^4,\ldots,H_1^{2N})^\mathrm{T}$, ${\bm l}_2(H_2) := (1,H_2^{p_1},H_2^{p_2},\ldots,H_2^{p_{N^*}})^\mathrm{T}$,
and the linear operators $L_\ell$, and functions ${\bm u}_\ell$ and $w_\ell$ for $\ell=1,2$ are defined (after non-dimensionalization) in Section~\ref{S.main-results}. Here we recognize the fact that the hypersurface $t=0$ in the space-time $\mathbf{R}^n\times\mathbf{R}$ is characteristic for 
the Kakinuma model, since the system of evolution equations is overdetermined for the variable $\zeta$, and underdetermined for the variables ${\bm \phi}_1$ and ${\bm \phi}_2$. As a consequence, solutions to the Kakinuma model must satisfy some compatibility conditions. Introducing linear operators $\mathcal{L}_{1,i}$ $(i=0,\ldots,N)$
acting on ${\bm \varphi}_1 = (\varphi_{1,0},\ldots,\varphi_{1,N})^\mathrm{T}$ and 
$\mathcal{L}_{2,i} $ $(i=0,\ldots,N^*)$ 
acting on ${\bm \varphi}_2 = (\varphi_{2,0},\ldots,\varphi_{2,N^*})^\mathrm{T}$ by 
\[
	\begin{cases}
		\displaystyle
		\mathcal{L}_{1,0}(H_1) {\bm \varphi}_1 
		:= \sum_{j=0}^{N} L_{1,0j}(H_1)\varphi_{1,j}, \\
		\displaystyle
		\mathcal{L}_{1,i}(H_1) {\bm \varphi}_1
		:= \sum_{j=0}^N ( L_{1,ij}(H_1)\varphi_{1,j} - H_1^{2i}L_{1,0j}(H_1)\varphi_{1,j} )
		\quad\mbox{for}\quad i=1,2,\ldots,N, \\
		\displaystyle
		\mathcal{L}_{2,0}(H_2,b) {\bm \varphi}_2 
		:= \sum_{j=0}^{N^*} L_{2,0j}(H_2,b)\varphi_{2,j}, \\
		\displaystyle
		\mathcal{L}_{2,i}(H_2,b) {\bm \varphi}_2
		:= \sum_{j=0}^{N^*} ( L_{2,ij}(H_2,b)\varphi_{2,j} - H_2^{p_i}L_{2,0j}(H_2,b)\varphi_{2,j} )
		\quad\mbox{for}\quad i=1,2,\ldots,N^*, 
	\end{cases}
\]
the necessary conditions can be written simply as 
\begin{equation}\label{intro:necessary}
	\begin{cases}
		\mathcal{L}_{1,i}(H_1) {\bm \phi}_1 = 0 \quad\mbox{for}\quad i=1,2,\ldots,N, \\
		\mathcal{L}_{2,i}(H_2,b) {\bm \phi}_2 = 0 \quad\mbox{for}\quad i=1,2,\ldots,N^*, \\
		\mathcal{L}_{1,0}(H_1) {\bm \phi}_1
		+ \mathcal{L}_{2,0}(H_2,b) {\bm \phi}_2  = 0.
	\end{cases}
\end{equation}
The first two vectorial identities are analogous to the compatibility conditions of the Isobe--Kakinuma model for water waves, 
while the last identity is specific to the bi-layer framework and is related to the continuity of the normal component of the velocity at the interface. 

A first key ingredient of the analysis is the fact that for sufficiently regular functions $\zeta$, $b$ and $\phi_1$ (respectively $\phi_2$), 
there exists a unique solution ${\bm \phi}_1$ (respectively ${\bm \phi}_2$) to the problems
\begin{equation}\label{intro:elliptic1}
	\begin{cases}
		\bm{l}_1(H_1)\cdot\bm{\phi}_1=\phi_1, \quad \mathcal{L}_{1,i}(H_1)\bm{\phi}_1=0
		\quad\mbox{for}\quad i=1,2,\ldots,N, \\
		\bm{l}_2(H_2)\cdot\bm{\phi}_2=\phi_2, \quad \mathcal{L}_{2,i}(H_2,b)\bm{\phi}_2=0
		\quad\mbox{for}\quad i=1,2,\ldots,N^*
	\end{cases}
\end{equation}
satisfying suitable elliptic estimates. What is more, the well-defined linear operators
\begin{align*}
	& \Lambda_1^{(N)}(\zeta) \colon \phi_1 \mapsto \mathcal{L}_{1,0}(H_1)\bm{\phi}_1, \\
	& \Lambda_2^{(N^*)}(\zeta,b) \colon \phi_2 \mapsto \mathcal{L}_{2,0}(H_2,b)\bm{\phi}_2, 
\end{align*}
are found to approximate the corresponding  Dirichlet-to-Neumann maps $\Lambda_1(\zeta)$ and $\Lambda_2(\zeta,b)$ 
defined by 
\begin{align*}
	&\Lambda_1(\zeta)\phi_1
	:= \bigl( -\partial_z\Phi_1+\nabla\Phi_1 \cdot \nabla\zeta \bigr)\bigr\vert_{z=\zeta(\bm{x},t)}, \\
	&\Lambda_2(\zeta,b)\phi_2
	:= \bigl( \partial_z\Phi_2-\nabla\Phi_2 \cdot \nabla\zeta \bigr)\bigr\vert_{z=\zeta(\bm{x},t)},
\end{align*}
where $\Phi_1$ and $\Phi_2$ are the unique solutions to  Laplace's equations
\[
\begin{cases}
	\Delta\Phi_1 + \partial_z^2\Phi_1 = 0 & \mbox{in}\quad \Omega_1(t), \\
	\Phi_1=\phi_1& \mbox{on}\quad \Gamma(t), \\
	\partial_z\Phi_1 = 0 & \mbox{on}\quad \Sigma_1, 
\end{cases}
\text{ and } \quad
\begin{cases}
	\Delta\Phi_2 + \partial_z^2\Phi_2 = 0 & \mbox{in}\quad \Omega_2(t), \\
	\Phi_2=\phi_2& \mbox{on}\quad \Gamma(t), \\
	\nabla\Phi_2 \cdot \nabla b - \partial_z\Phi_2 = 0
	& \mbox{on}\quad \Sigma_2,
\end{cases}
\]
where we denote the upper layer, the lower layer, the interface, the rigid-lid, 
and the bottom at time $t$ by $\Omega_1(t)$, $\Omega_2(t)$, $\Gamma(t)$, $\Sigma_1$, and $\Sigma_2$, respectively. 
Specifically it is proved that, under the special choice of the 
indices $p_0,p_1,\ldots,p_{N^*}$ in (H1) or (H2) and after suitable rescaling, 
the difference between the dimensionless operators is of size $O(\delta_1^{4N+2}+\delta_2^{4N+2})$.
This analysis, which follows directly from the corresponding analysis for surface waves developed in~\cite{Iguchi2018-2} and scaling arguments, 
provides the key argument in the proof of the consistency result described in Theorem~\ref{T.consistency}.

In order to study the Kakinuma model, we also need to analyze the full elliptic system
\begin{equation}\label{intro:elliptic2}
	\begin{cases}
		\mathcal{L}_{1,i}(H_1) {\bm \phi}_1 = f_{1,i} \quad\mbox{for}\quad i=1,2,\ldots,N, \\
		\mathcal{L}_{2,i}(H_2,b) {\bm \phi}_2 = f_{2,i} \quad\mbox{for}\quad i=1,2,\ldots,N^*, \\
		\mathcal{L}_{1,0}(H_1) {\bm \phi}_1
		+ \mathcal{L}_{2,0}(H_2,b) {\bm \phi}_2  = \nabla\cdot {\bm f}_3,\\
		-\bm{l}_1(H_1)\cdot\bm{\phi}_1+\bm{l}_2(H_2)\cdot\bm{\phi}_2=f_4,
	\end{cases}
\end{equation}
for sufficiently regular functions $\zeta$, $b$ and ${\bm f}_1=(f_{1,1},\ldots f_{1,N})^\mathrm{T},{\bm f}_2=(f_{2,1},\ldots,f_{2,N^*})^\mathrm{T},{\bm f}_3,f_4$. The ellipticity of the problem relies on the coercivity of the corresponding operators $L_1(H_1)$ and $L_2(H_2)$. The solvability of~\eqref{intro:elliptic2} is essential in several directions. Firstly, it provides an alternative consistency result, where solutions to the full interfacial gravity waves system produce approximate solutions to the Kakinuma model but satisfying exactly and not approximately the necessary conditions~\eqref{intro:necessary}. 
In turn, this provides a crucial ingredient to the full justification of the Kakinuma model described in Theorem~\ref{T.justification}.
Furthermore, the arrangement of initial data 
satisfying the compatibility conditions as stated in Proposition~\ref{P.compatibility} amounts to solving~\eqref{intro:elliptic2} 
with ${\bm f}_1={\bm 0}$, ${\bm f}_2={\bm 0}$, ${\bm f}_3={\bm 0}$ and $f_4=\phi$.
Similarly, our result on the Hamiltonians $\mathscr{H}^{\mbox{\rm\tiny K}}$ and $\mathscr{H}^{\mbox{\rm\tiny IW}}$ described in Theorem~\ref{T.Hamiltonian} relies on a comparison of solutions to~\eqref{intro:elliptic2} with ${\bm f}_1={\bm 0}$, ${\bm f}_2={\bm 0}$, ${\bm f}_3={\bm 0}$ and $f_4=\phi$ and solutions to
\[
\begin{cases}
	\Delta\Phi_1 + \partial_z^2\Phi_1 = 0 & \mbox{in}\quad \Omega_1(t), \\
	\Delta\Phi_2 + \partial_z^2\Phi_2 = 0 & \mbox{in}\quad \Omega_2(t), \\
	\partial_z\Phi_1 = 0 & \mbox{on}\quad \Sigma_1, \\
	\nabla\Phi_2 \cdot \nabla b - \partial_z\Phi_2 = 0
	& \mbox{on}\quad \Sigma_2,\\
	(\nabla\Phi_1 \cdot \nabla \zeta - \partial_z\Phi_1)  -(\nabla\Phi_2 \cdot \nabla \zeta - \partial_z\Phi_2)=0 & \mbox{on}\quad \Gamma(t), \\
	\rho_2\Phi_2-\rho_1\Phi_1 = \phi& \mbox{on}\quad \Gamma(t), \\
\end{cases}
\]
thus extending to the interfacial gravity waves framework the analysis in~\cite{DucheneIguchi2019-1}.
Finally, the solvability of~\eqref{intro:elliptic2} allows to determine and control time derivatives 
$\partial_t{\bm \phi}_1 $ and $\partial_t{\bm \phi}_2 $ of sufficiently regular solutions to the Kakinuma model~\eqref{intro:Kakinuma-compact} by using the equations obtained when differentiating with respect to time the compatibility conditions~\eqref{intro:necessary} combined with the last equation of~\eqref{intro:Kakinuma-compact}. This is a crucial ingredient for the analysis of the initial value problem. 

Another crucial ingredient for the analysis of the initial value problem concerns uniform energy estimates on the linearized Kakinuma system. To this end, we write the linearized system under the form
\begin{equation}\label{intro:linearized}
	\mathscr{A}_1(\partial_t+\bm{u}\cdot\nabla)\dot{\bm{U}} + \mathscr{A}_0^\mathrm{mod}\dot{\bm{U}} = \dot{\bm{F}},
\end{equation}
where $\dot{\bm{U}}:=(\dot{\zeta}, \dot{\bm{\phi}}_1, \dot{\bm{\phi}}_2)^\mathrm{T}$ is the deviation from the reference state $\bm{U}:=({\zeta}, {\bm{\phi}}_1, {\bm{\phi}}_2)^\mathrm{T}$, $\bm{u}$ is a suitable velocity which is a convex combination of  $\bm{u}_1$ and $\bm{u}_2$ whose weights depend on $\rho_\ell$, $H_\ell$ as well as $\alpha_\ell$ ($\ell=1,2$) the positive constants mentioned previously,  $\dot{\bm{F}}$ represent lower order terms and 
 $\mathscr{A}_1:=\mathscr{A}_1(\bm{U})$ is a skew-symmetric matrix and $\mathscr{A}_0^\mathrm{mod}:=\mathscr{A}_0^\mathrm{mod}(\bm{U})$ is a linear operator symmetric in $L^2$. 
The energy function associated to~\eqref{intro:linearized} is given by 
$(\mathscr{A}_0^\mathrm{mod}\dot{\bm{U}},\dot{\bm{U}})_{L^2}$, and we prove that
\[(\mathscr{A}_0^\mathrm{mod}\dot{U},\dot{U})_{L^2} \simeq \mathscr{E}(\dot{\bm{U}})
	:= \|\dot{\zeta}\|_{L^2}^2
	+ \sum_{\ell=1,2}\rho_\ell( \|\nabla\dot{\bm{\phi}}_\ell\|_{L^2}^2
	+ \|\dot{\bm{\phi}}_\ell'\|_{L^2}^2 )
\]
under the non-cavitation assumption~\eqref{intro:non-cavitation} and the stability condition~\eqref{intro:stability}. 
Because the structure of~\eqref{intro:linearized} is not standard, the control of the energy function is obtained 
by testing~\eqref{intro:linearized} with the time derivatives, $\partial_t\dot{\bm{U}}$.
This, together with suitable product and commutator estimates in Sobolev spaces, provides the {\em a priori} control of 
the energy function for solutions to the Kakinuma model and their derivatives, and we show that this control is uniform
in the shallow water regime after suitable rescaling. 
Since the construction and uniqueness of a solution was obtained in the companion paper~\cite{DucheneIguchi2020},
the uniform estimates provide the proof of the long-time well-posedness of the initial value problem for the Kakinuma model result 
stated in Theorem~\ref{T.uniform}. 
Furthermore, using the aforementioned consistency result, we prove that the difference between solutions to the full interfacial gravity waves system
and corresponding solutions to the Kakinuma model satisfy an identity analogous to~\eqref{intro:linearized}, and hence
infer a control of the energy function of the difference and its derivatives, which yields the full justification of the Kakinuma model stated in Theorem~\ref{T.justification}.

\paragraph{Outline}
The contents of this paper are as follows. 
In Section~\ref{S.equations} we first recall the basic equations governing the interfacial gravity waves 
and write down the Kakinuma model that we are going to analyze in this paper, 
and then rewrite them in a nondimensional form by introducing several nondimensional parameters. 
Hamiltonians of the full model and of the Kakinuma model in the nondimensional variables are also provided. 
In Section~\ref{S.main-results} we first introduce some differential operators, which enable us to write 
the Kakinuma model simply in the form~\eqref{intro:Kakinuma-compact}, and then we present the precise statements of our main results in this paper. 
In Section~\ref{S.consistency} we first recall results in the framework of surface waves related to the consistency 
of the Isobe--Kakinuma model, and then prove Theorems~\ref{theorem-consistency1} and~\ref{theorem-consistency2} 
concerning the consistency of the Kakinuma model by a simple scaling argument. 
In Section~\ref{S.elliptic} we first derive an elliptic estimate related to the compatibility conditions for the Kakinuma model, 
which explains how to prepare the initial data, as stated in Proposition~\ref{preparation-ini}.
Then we give uniform a priori bounds on regular solutions to the Kakinuma model, especially, a priori bounds of time derivatives. 
In Section~\ref{S.hyperbolic} we provide uniform energy estimates for the solution to the Kakinuma model and prove 
Theorem~\ref{theorem-uniform}, which ensures the existence of the solution to the initial value problem for the Kakinuma model 
on a time interval independent of parameters, especially, $\delta_1$ and $\delta_2$, under the stability condition, 
the non-cavitation assumptions, and intrinsic compatibility conditions on the initial data, together with a uniform bound of the solution. 
In Section~\ref{S.convergence} we first give a supplementary estimate on an approximation of the Dirichlet-to-Neumann map, 
and then revisit the consistency of the Kakinuma model. 
We prove Proposition~\ref{prop-consistency}, which is another version of the consistency given in 
Theorem~\ref{theorem-consistency2}, where we adopt a different construction of an approximate solution to the Kakinuma model 
from the solution to the full model. 
Then, by making use of the well-posedness of the initial value problem for the Kakinuma model we prove 
Theorem~\ref{theorem-justification} which provides a conditional rigorous justification of the Kakinuma model, that is, 
assuming the existence of a solution to the full model with a uniform bound we derive an error estimate between 
a corresponding solution to the Kakinuma model and that of the full model.
Finally, in Section~\ref{S.Hamiltonian} we prove Theorem~\ref{theorem-Hamiltonian} which gives an error estimate between the Hamiltonian of the Kakinuma model and that of the full model. For the convenience of the reader, the structure of the paper and proofs dependencies are sketched in Figure~\ref{F.paper-structure}. 

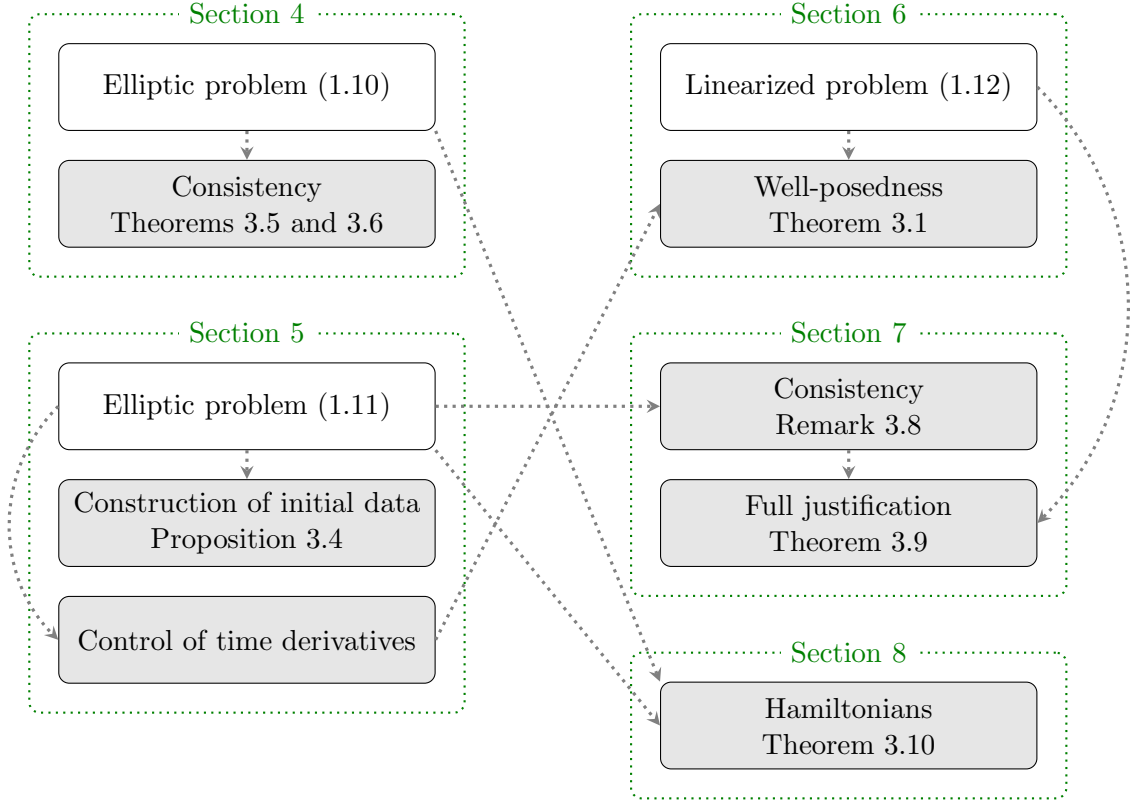
\begin{figure}[ht]
	\begin{center}
		\begin{tikzpicture}[every text node part/.style={align=center}]

			\tikzstyle{analysis}=[rectangle,draw,rounded corners=4pt,fill=white,minimum width=13em,minimum height=3em,node distance=1em]
			\tikzstyle{result}=[rectangle,draw,rounded corners=4pt,fill=gray!20,minimum width=13em,minimum height=3em,node distance=1em]
			\tikzstyle{section}=[draw,thick,dotted,color=black!50!green,rounded corners=4pt,minimum width=13em,minimum height=3em,node distance=4em,inner xsep=1em, inner ysep=1em]
			\tikzstyle{sectiontitle}=[color=black!50!green,fill=white]
			\tikzstyle{arrow}=[->,dotted,very thick,>=stealth,color=gray,rounded corners=5pt ]

			\node[analysis]  (Elliptic1) at (0,0) {Elliptic problem~\eqref{intro:elliptic1}};
			\node[result]  (Consistency1) [below = of Elliptic1] {Consistency \\ Theorems~\ref{theorem-consistency1} and~\ref{theorem-consistency2}};
			\node[section, fit=(Elliptic1) (Consistency1)] (S4) {};
			\node[sectiontitle] at (S4.north) {Section~\ref{S.consistency}};

			\node[analysis]  (Elliptic2) [below = 8em of Elliptic1] {Elliptic problem~\eqref{intro:elliptic2}};
			\node[result]  (Compatibility) [below = of Elliptic2] {Construction of initial data \\ Proposition~\ref{preparation-ini}};
			\node[result]  (timederivatives) [below = of Compatibility] {Control of time derivatives};
			\node[section, fit=(Elliptic2) (Compatibility) (timederivatives)] (S5) {};
			\node[sectiontitle] at (S5.north) {Section~\ref{S.elliptic}};
			
			\node[analysis]  (Linearized) at (8,0) {Linearized problem~\eqref{intro:linearized}};
			\node[result]  (Uniform) [below = of Linearized] {Well-posedness \\ Theorem~\ref{theorem-uniform}};
			\node[section, fit=(Uniform) (Linearized)] (S6) {};
			\node[sectiontitle] at (S6.north) {Section~\ref{S.hyperbolic}};

			\node[result]  (Consistency2) [below = 8em of Linearized] {Consistency \\ Remark~\ref{R.app-sol}};
			\node[result]  (Justification) [below = of Consistency2] {Full justification \\ Theorem~\ref{theorem-justification}};
			\node[section, fit=(Consistency2) (Justification)] (S7) {};
			\node[sectiontitle] at (S7.north) {Section~\ref{S.convergence}};
			
			\node[result]  (Hamiltonian) [below = 8em of Consistency2] {Hamiltonians \\ Theorem~\ref{theorem-Hamiltonian}};
			\node[section, fit=(Hamiltonian)] (S8) {};
			\node[sectiontitle] at (S8.north) {Section~\ref{S.Hamiltonian}};
			
			\draw[arrow] (Elliptic1) to (Consistency1){};
			\draw[arrow] (Elliptic2) to (Compatibility){};
			\draw[arrow] (Elliptic2.west) to[out = 225,in=135] (timederivatives.west){};
			\draw[arrow] (Linearized) to (Uniform){};
			\draw[arrow] (timederivatives.east) to (Uniform.west){};
			\draw[arrow] (Elliptic2) to (Consistency2){};
			\draw[arrow] (Consistency2) to (Justification){};
			\draw[arrow] (Linearized.east) to[out = -45,in=45] (Justification.east){};
			\draw[arrow] (Elliptic1.south east) to (Hamiltonian.north west){};
			\draw[arrow] (Elliptic2.south east) to (Hamiltonian.west){};
					
		\end{tikzpicture}
	\end{center}
	
	\caption{Articulation of the proofs.}
	\label{F.paper-structure}
\end{figure}

\paragraph{Notation} 
We denote by $W^{m,p}$ the $L^p$ Sobolev space of order $m$ on $\mathbf{R}^n$ and  $H^m=W^{m,2}$. 
We put $\mathring{H}^m=\{ \phi \,;\, \nabla\phi \in H^{m-1}\}$. 
The norm of a Banach space $B$ is denoted by $\|\cdot\|_B$.
The $L^2$-inner product is denoted by $(\cdot,\cdot)_{L^2}$. 
We put $\partial_t=\frac{\partial}{\partial t}$, $\partial_j=\partial_{x_j}=\frac{\partial}{\partial x_j}$, 
and $\partial_z=\frac{\partial}{\partial z}$. 
$[P,Q]=PQ-QP$ denotes the commutator and $[P;u,v]=P(uv)-(Pu)v-u(Pv)$ denotes the symmetric commutator. 
For a matrix $A$ we denote by $A^\mathrm{T}$ the transpose of $A$. 
$O$ denotes a zero matrix. 
For a vector $\mbox{\boldmath$\phi$}=(\phi_0,\phi_1,\ldots,\phi_N)^\mathrm{T}$ we denote the last $N$ 
components by $\mbox{\boldmath$\phi$}'=(\phi_1,\ldots,\phi_N)^\mathrm{T}$. 
$f \lesssim g$ means that there exists a non-essential positive constant $C$ such that 
$f \leq Cg$ holds. 
$f \simeq g$ means that $f \lesssim g$ and $g \lesssim f$ hold.

\paragraph{Acknowledgement} 
T. I. was partially supported by JSPS KAKENHI Grant Number JP17K18742 and JP22H01133. 
V. D. thanks the Centre Henri Lebesgue ANR-11-LABX-0020-01 for creating an attractive mathematical environment.

\section{The basic equations and the Kakinuma model}\label{S.equations}
%
\subsection{Equations with physical variables}
We first recall the equations governing potential flows for two layers of immiscible, incompressible, 
homogeneous, and inviscid fluids, and then write down the Kakinuma model at stake in this work. 
In the following, we denote the upper layer, the lower layer, the interface, the rigid-lid, 
and the bottom at time t by $\Omega_1(t)$, $\Omega_2(t)$, $\Gamma(t)$, $\Sigma_1$, and $\Sigma_2$, respectively. 
The velocity potentials $\Phi_1(\bm{x},z,t)$ and $\Phi_2(\bm{x},z,t)$ in the upper and lower layers, respectively, 
satisfy Laplace's equations 
\begin{align}
&\label{LaplaceUpper}
 \Delta\Phi_1 +\partial_z^2 \Phi_1= 0 \quad\mbox{in}\quad \Omega_1(t), \\
&\label{LaplaceLower}
 \Delta\Phi_2 +\partial_z^2 \Phi_2= 0 \quad\mbox{in}\quad \Omega_2(t),
\end{align}
where $\Delta=\partial_1^2+\cdots+\partial_n^2$ is the Laplacian with respect to the horizontal space variables 
$\bm{x}=(x_1,\ldots,x_n)$. 
Bernoulli's laws of each layers have the form 
\begin{align}
&\label{BernoulliUpper}
 \rho_1\left( \partial_t\Phi_1 + \frac12(|\nabla\Phi_1|^2 + (\partial_z\Phi_1)^2) + gz \right) + P_1 = 0
 \quad\mbox{in}\quad \Omega_1(t), \\
&\label{BernoulliLower}
 \rho_2\left( \partial_t\Phi_2 + \frac12(|\nabla\Phi_2|^2 + (\partial_z\Phi_2)^2) + gz \right) + P_2 = 0
 \quad\mbox{in}\quad \Omega_2(t), 
\end{align}
where $\nabla=(\partial_1,\ldots,\partial_n)$, the positive constant $g$ is the acceleration due to gravity, 
and $P_1(\bm{x},z,t)$ and $P_2(\bm{x},z,t)$ are pressures in the upper and lower layers, respectively. 
The dynamical boundary condition on the interface is given by 
\begin{equation}\label{Dynamical}
P_1 = P_2 \quad\mbox{on}\quad \Gamma(t).
\end{equation}
The kinematic boundary conditions on the interface, the rigid-lid, and the bottom are given by 
\begin{align}
&\label{KinematicInterface1}
 \partial_t\zeta + \nabla\Phi_1\cdot\nabla\zeta - \partial_z\Phi_1 = 0
 \quad\mbox{on}\quad \Gamma(t), \\
&\label{KinematicInterface2}
 \partial_t\zeta + \nabla\Phi_2\cdot\nabla\zeta - \partial_z\Phi_2 = 0
 \quad\mbox{on}\quad \Gamma(t), \\
&\label{KinematicLid}
 \partial_z\Phi_1 = 0 
 \makebox[18.3ex]{} \quad\mbox{on}\quad \Sigma_1, \\
&\label{KinematicBottom}
 \nabla\Phi_2\cdot\nabla b - \partial_z\Phi_2 = 0
 \makebox[6.3ex]{} \quad\mbox{on}\quad \Sigma_2.
\end{align}
These are the basic equations for interfacial gravity waves. 
It follows from Bernoulli's laws~\eqref{BernoulliUpper}--\eqref{BernoulliLower} and the dynamical 
boundary condition~\eqref{Dynamical} that 
\begin{align}\label{DynamicalBC}
& \rho_1\left( \partial_t\Phi_1 + \frac12(|\nabla\Phi_1|^2 + (\partial_z\Phi_1)^2) \right) \\
& - \rho_2\left( \partial_t\Phi_2 + \frac12(|\nabla\Phi_2|^2 + (\partial_z\Phi_2)^2) \right) 
 = (\rho_2-\rho_1)g\zeta
 \quad\mbox{on}\quad \Gamma(t). \nonumber
\end{align}
We will always assume the stable stratification condition $(\rho_2-\rho_1)g > 0$. 
As in the case of surface water waves, the basic equations have a variational structure and 
the corresponding Luke's Lagrangian is given, up to terms which do not contribute to the variation of the Lagrangian, 
by the vertical integral of the pressure in the water regions. 
After using Bernoulli's laws~\eqref{BernoulliUpper}--\eqref{BernoulliLower} we can find the Lagrangian density 
\begin{align}\label{Lagrangian}
\mathscr{L}(\Phi_1,\Phi_2,\zeta)
&= -\rho_1\int_{\zeta}^{h_1} \left( \partial_t\Phi_1 + \frac12(|\nabla\Phi_1|^2 + (\partial_z\Phi_1)^2) \right) {\rm d}z \\
&\quad\;
 - \rho_2\int_{-h_2+b}^{\zeta} \left( \partial_t\Phi_2 + \frac12(|\nabla\Phi_2|^2 + (\partial_z\Phi_2)^2) \right) {\rm d}z 
 - \frac12(\rho_2-\rho_1)g\zeta^2. \nonumber
\end{align}
In fact, one checks readily that~\eqref{LaplaceUpper}--\eqref{LaplaceLower} 
and~\eqref{KinematicInterface1}--\eqref{DynamicalBC} are Euler--Lagrange equations associated with the action function 
\[
\mathscr{J}(\Phi_1,\Phi_2,\zeta)
 := \int_{t_0}^{t_1}\!\!\!\int_{\mathbf{R}^n}\mathscr{L}(\Phi_1,\Phi_2,\zeta)\,\mathrm{d}\bm{x}\,\mathrm{d}t. 
\]

We proceed to the Kakinuma model. 
Let $N$ and $N^*$ be nonnegative integers. 
In view of the analysis for the Isobe--Kakinuma model for surface water waves, 
we approximate the velocity potentials $\Phi_1$ and $\Phi_2$ in the Lagrangian by 
\begin{equation}\label{Approximation}
\begin{cases}
 \displaystyle
  \Phi_1^\mathrm{app}(\bm{x},z,t) = \sum_{i=0}^N (-z+h_1)^{2i}\phi_{1,i}(\bm{x},t), \\[2.5ex]
 \displaystyle
  \Phi_2^\mathrm{app}(\bm{x},z,t) = \sum_{i=0}^{N^*} (z+h_2-b(\bm{x}))^{p_i}\phi_{2,i}(\bm{x},t),
\end{cases}
\end{equation}
where $p_0,p_1,\ldots,p_{N^*}$ are nonnegative integers satisfying $0=p_0<p_1<\cdots<p_{N^*}$. 
Plugging~\eqref{Approximation} into the Lagrangian density~\eqref{Lagrangian}, 
we obtain an approximate Lagrangian density 
\[
\mathscr{L}^\mathrm{app}(\bm{\phi}_1,\bm{\phi}_2,\zeta)
:=\mathscr{L}(\Phi_1^\mathrm{app},\Phi_2^\mathrm{app},\zeta),
\]
where $\bm{\phi}_1:=(\phi_{1,0},\phi_{1,1},\ldots,\phi_{1,N})^\mathrm{T}$ and 
$\bm{\phi}_2:=(\phi_{2,0},\phi_{2,1},\ldots,\phi_{2,N^*})^\mathrm{T}$. 
The corresponding Euler--Lagrange equation is the Kakinuma model, which has the form 
\begin{equation}\label{KakinumaModel}
\begin{cases}
\displaystyle
 H_1^{2i}\partial_t\zeta - \sum_{j=0}^N\biggl\{ \nabla\cdot\biggr(
  \frac{1}{2(i+j)+1}H_1^{2(i+j)+1}\nabla\phi_{1,j} \biggr)
  - \frac{4ij}{2(i+j)-1}H_1^{2(i+j)-1}\phi_{1,j} \biggr\}=0 \\
\makebox[27em]{}\mbox{for}\quad i=0,1,\ldots,N, \\
\displaystyle
 H_2^{p_i}\partial_t\zeta + \sum_{j=0}^{N^*} \biggl\{ \nabla\cdot\biggl(
   \frac{1}{p_i+p_j+1}H_2^{p_i+p_j+1}\nabla\phi_{2,j}
   -\frac{p_j}{p_i+p_j}H_2^{p_i+p_j}\phi_{2,j}\nabla b\biggr) \\
  \displaystyle\phantom{ H_2^{p_i}\partial_t\zeta + \sum_{j=0}^N }
   +\frac{p_i}{p_i+p_j}H_2^{p_i+p_j}\nabla b\cdot\nabla\phi_{2,j}
   -\frac{p_ip_j}{p_i+p_j-1}H_2^{p_i+p_j-1}(1 + |\nabla b|^2)\phi_{2,j}\biggr\}=0 \\
\makebox[27em]{}\mbox{for}\quad i=0,1,\ldots,N^*, \\
\displaystyle
 \rho_1\biggl\{ \sum_{j=0}^NH_1^{2j}\partial_t\phi_{1,j} + g\zeta + \frac12\biggl(
  \biggl|\sum_{j=0}^NH_1^{2j}\nabla\phi_{1,j}\biggr|^2
  + \biggl(\sum_{j=0}^N2jH_1^{2j-1}\phi_{1,j}\biggr)^2 \biggr) \biggr\} \\
\displaystyle\quad
 - \rho_2\biggl\{ \sum_{j=0}^{N^*} H_2^{p_j} \partial_t \phi_{2,j} + g\zeta \\
\displaystyle\qquad\quad
   + \frac12 \biggl( \biggl| \sum_{j=0}^{N^*} ( H_2^{p_j}\nabla\phi_{2,j} - p_j H_2^{p_j-1}\phi_{2,j}\nabla b ) \biggr|^2 
   + \biggl( \sum_{j=0}^{N^*} p_j H_2^{p_j-1} \phi_{2,j} \biggr)^2 \biggr) \biggr\} = 0,
\end{cases}
\end{equation}
where $H_1$ and $H_2$ are depths of the upper and the lower layers, that is, 
\[
H_1(t,\bm{x}) := h_1 - \zeta(\bm{x},t), \qquad H_2(\bm{x},t) := h_2 + \zeta(\bm{x},t) - b(\bm{x}).
\]
In~\eqref{KakinumaModel}, we used the notational convention $0/0 = 0$. 
More precisely, this convention was used so as to dictate $p_0/(p_0+p_0)=0$ and $p_0p_1/(p_0+p_1-1)=0$ in the case $p_1=1$. 
We recall also that $p_0=0$ is always assumed.

\subsection{The dimensionless equations}
In order to rigorously validate the Kakinuma model~\eqref{KakinumaModel} as a higher order shallow water approximation 
of the full model for interfacial gravity waves~\eqref{LaplaceUpper}--\eqref{KinematicBottom}, 
we first introduce nondimensional parameters and then non-dimensionalize the equations, through a convenient rescaling of variables. 
Let $\lambda$ be a typical horizontal wavelength. 
Following D. Lannes~\cite{Lannes2013}, we introduce a nondimensional parameter $\delta$ by 
\[
\delta := \frac{h}{\lambda} \qquad\mbox{with}\quad h:=\frac{h_1h_2}{\underline{\rho}_1h_2+\underline{\rho}_2h_1},
\]
where $\underline{\rho}_1$ and $\underline{\rho}_2$ are relative densities. 
We also need to use relative depths $\underline{h}_1$ and $\underline{h}_2$ of the layers. 
These nondimensional parameters are defined by 
\[
\underline{\rho}_\ell:=\frac{\rho_\ell}{\rho_1+\rho_2}, \qquad
\underline{h}_\ell:=\frac{h_\ell}{h} \qquad (\ell=1,2),
\]
which satisfy the relations 
\begin{equation}\label{parameters}
\underline{\rho}_1+\underline{\rho}_2=1, \qquad 
\frac{\underline{\rho}_1}{\underline{h}_1}+\frac{\underline{\rho}_2}{\underline{h}_2}=1.
\end{equation}
Note also that $\min\{h_1,h_2\} \leq h \leq \max\{h_1,h_2\}$. 
It follows from the second relation in~\eqref{parameters} that 
\begin{equation}\label{parameter-relation}
1 < \min\biggl\{ \frac{\underline{h}_1}{\underline{\rho}_1}, \frac{\underline{h}_2}{\underline{\rho}_2} \biggr\} \leq 2.
\end{equation}
Here, we note that the standard shallowness parameters $\delta_1:=\frac{h_1}{\lambda}$ and $\delta_2:=\frac{h_2}{\lambda}$ 
relative to the upper and the lower layers, respectively, are related to the above parameters by 
$\delta_\ell = \underline{h}_\ell\delta$ for $\ell=1,2$. 
In many results of this paper, we restrict our consideration to the parameter regime 
\begin{equation}\label{parameter-regime}
\underline{h}_1^{-1} + \underline{h}_2^{-1} \lesssim 1. 
\end{equation}
To understand this restriction, it is convenient to use nondimensional parameters 
$\gamma:=\frac{\rho_1}{\rho_2}$ and $\theta:=\frac{h_1}{h_2}$. 
In terms of these parameters, $\underline{h}_\ell^{-1}$ $(\ell=1,2)$ can be represented as 
\[
\underline{h}_1^{-1} = \frac{\gamma+1}{\gamma+\theta}, \qquad
\underline{h}_2^{-1} = \frac{\gamma^{-1}+1}{\gamma^{-1}+\theta^{-1}}.
\]
Therefore, the only cases that~\eqref{parameter-regime} excludes are the case $\gamma,\theta \ll 1$ and the case $\gamma,\theta \gg1$. 
Since we shall also assume the stable stratification condition $(\rho_2-\rho_1)g>0$, 
we can describe the two regimes considered in this paper as 
\begin{enumerate}
 \item[(i)]
 $\gamma\simeq1$, i.e., $\rho_1\simeq\rho_2$, 
 \item[(ii)]
 $\gamma \ll 1$ and $\theta\gtrsim 1$, i.e., $\rho_1\ll\rho_2$ and $h_2\lesssim h_1$. 
\end{enumerate}

Introducing $c_{\mbox{\rm\tiny SW}} :=\sqrt{ (\underline{\rho}_2-\underline{\rho}_1)gh }$
the speed of infinitely long and small interfacial gravity waves, we rescale the independent and the dependent variables by 
\[
\bm{x} = \lambda\tilde{\bm{x}}, \quad z = h\tilde{z}, \quad t = \frac{\lambda}{c_{\mbox{\rm\tiny SW}}}\tilde{t}, \quad
 \zeta = h\tilde{\zeta}, \quad b = h\tilde{b}, \quad  \Phi_\ell = \lambda c_{\mbox{\rm\tiny SW}}\tilde{\Phi}_\ell
 \quad (\ell=1,2).
\]
Plugging these into the full model~\eqref{LaplaceUpper}--\eqref{LaplaceLower} and 
\eqref{KinematicInterface1}--\eqref{DynamicalBC} and dropping the tilde sign in the notation we obtain 
\[
\begin{cases}
\Delta\Phi_1 + \delta^{-2}\partial_z^2\Phi_1 = 0 & \mbox{in}\quad \Omega_1(t), \\
\Delta\Phi_2 + \delta^{-2}\partial_z^2\Phi_2 = 0 & \mbox{in}\quad \Omega_2(t), \\
\partial_t\zeta + \nabla\Phi_1 \cdot \nabla\zeta - \delta^{-2}\partial_z\Phi_1 = 0
 & \mbox{on}\quad \Gamma(t), \\
\partial_t\zeta + \nabla\Phi_2 \cdot \nabla\zeta - \delta^{-2}\partial_z\Phi_2 = 0
 & \mbox{on}\quad \Gamma(t), \\
\partial_z\Phi_1 = 0 & \mbox{on}\quad \Sigma_1, \\
\nabla\Phi_2 \cdot \nabla b - \delta^{-2}\partial_z\Phi_2 = 0
 & \mbox{on}\quad \Sigma_2, \\
\underline{\rho}_1\Bigl( \partial_t\Phi_1 + \frac12|\nabla\Phi_1|^2 + \frac12\delta^{-2}(\partial_z\Phi_1)^2 \Bigr) \\
\quad
 - \underline{\rho}_2\Bigl( \partial_t\Phi_2 + \frac12|\nabla\Phi_2|^2 + \frac12\delta^{-2}(\partial_z\Phi_2)^2 \Bigr)
 - \zeta = 0
 & \mbox{on}\quad \Gamma(t),
\end{cases}
\]
where in this scaling the upper layer $\Omega_1(t)$, the lower layer $\Omega_2(t)$, the interface $\Gamma(t)$, 
the rigid-lid $\Sigma_1$, and the bottom $\Sigma_2$ are written as 
\begin{align*}
& \Omega_1(t) = \{ X=(\bm{x},z)\in\mathbf{R}^{n+1} \,;\, \zeta(\bm{x},t) < z < \underline{h}_1 \}, \\
& \Omega_2(t) = \{ X=(\bm{x},z)\in\mathbf{R}^{n+1} \,;\, -\underline{h}_2+b(\bm{x}) < z < \zeta(\bm{x},t) \}, \\
& \Gamma(t) = \{ X=(\bm{x},z)\in\mathbf{R}^{n+1} \,;\, z = \zeta(\bm{x},t) \}, \\
& \Sigma_1 = \{ X=(\bm{x},z)\in\mathbf{R}^{n+1} \,;\, z = \underline{h}_1 \}, \\
& \Sigma_2 = \{ X=(\bm{x},z)\in\mathbf{R}^{n+1} \,;\,  z = -\underline{h}_2+b(\bm{x}) \}.
\end{align*}
Denoting
\[
\phi_\ell(\bm{x},t):=\Phi_\ell(\bm{x},\zeta(\bm{x},t),t) \qquad (\ell=1,2)
\]
and using the chain rule, the above system can be written in a more compact and closed form as 
\begin{equation}\label{full-model-evolution}
\begin{cases}
\partial_t\zeta + \Lambda_1(\zeta,\delta,\underline{h}_1)\phi_1=0,\\
\partial_t\zeta - \Lambda_2(\zeta,b,\delta,\underline{h}_2)\phi_2 =0,\\
\displaystyle
\underline{\rho}_1\biggl( \partial_t\phi_1 + \frac12|\nabla\phi_1|^2
 - \frac12\delta^2 \frac{(\Lambda_1(\zeta,\delta,\underline{h}_1)\phi_1
  - \nabla\zeta \cdot \nabla\phi_1 )^2}{1+\delta^2|\nabla\zeta|^2} \biggr) \\
\displaystyle\quad
 - \underline{\rho}_2\biggl( \partial_t\phi_2 + \frac12|\nabla\phi_2|^2
  - \frac12\delta^2 \frac{(\Lambda_2(\zeta,b,\delta,\underline{h}_2)\phi_2
   + \nabla\zeta \cdot \nabla\phi_2 )^2}{1+\delta^2|\nabla\zeta|^2} \biggr)
 - \zeta = 0,
\end{cases}
\end{equation}
where $\Lambda_1(\zeta,\delta,\underline{h}_1)$ and $\Lambda_2(\zeta,b,\delta,\underline{h}_2)$ 
are the Dirichlet-to-Neumann maps for Laplace's equations. 
More precisely, these are defined by 
\begin{align*}
&\Lambda_1(\zeta,\delta,\underline{h}_1)\phi_1
 := \bigl( -\delta^{-2}\partial_z\Phi_1+\nabla\Phi_1 \cdot \nabla\zeta \bigr)\bigr\vert_{z=\zeta(\bm{x},t)}, \\
&\Lambda_2(\zeta,b,\delta,\underline{h}_2)\phi_2
 := \bigl( \delta^{-2}\partial_z\Phi_2-\nabla\Phi_2 \cdot \nabla\zeta \bigr)\bigr\vert_{z=\zeta(\bm{x},t)},
\end{align*}
where $\Phi_1$ and $\Phi_2$ are unique solutions to the boundary value problems 
\[
\begin{cases}
 \Delta\Phi_1 + \delta^{-2}\partial_z^2\Phi_1 = 0 & \mbox{in}\quad \Omega_1(t), \\
 \Phi_1=\phi_1& \mbox{on}\quad \Gamma(t), \\
 \partial_z\Phi_1 = 0 & \mbox{on}\quad \Sigma_1, 
\end{cases}
\text{ and } \quad
\begin{cases}
\Delta\Phi_2 + \delta^{-2}\partial_z^2\Phi_2 = 0 & \mbox{in}\quad \Omega_2(t), \\
\Phi_2=\phi_2& \mbox{on}\quad \Gamma(t), \\
\nabla\Phi_2 \cdot \nabla b - \delta^{-2}\partial_z\Phi_2 = 0
 & \mbox{on}\quad \Sigma_2.
\end{cases}
\]

As for the Kakinuma model, we introduce additionally the rescaled variables 
\[
\phi_{1,i} := \frac{\lambda c_{\mbox{\rm\tiny SW}}}{h_1^{2i}} \tilde\phi_{1,i}, 
 \qquad \phi_{2,i} := \frac{\lambda c_{\mbox{\rm\tiny SW}}}{h_2^{p_i}}\tilde{\phi}_{2,i}, 
\]
where we recall that  $p_0,p_1,\ldots,p_{N^*}$ are nonnegative integers satisfying ${0=p_0<p_1<\cdots<p_{N^*}}$ 
appearing in the approximation~\eqref{Approximation}.
Plugging these and the previous scaling into the Kakinuma model~\eqref{KakinumaModel} 
and dropping the tilde sign in the notation we obtain the Kakinuma model in the nondimensional form, 
which is written as 
\begin{equation}\label{Kakinuma-dimensionless}
\begin{cases}
\displaystyle
 H_1^{2i}\partial_t\zeta - \underline{h}_1 \sum_{j=0}^N \biggl\{ \nabla\cdot\biggr(
  \frac{1}{2(i+j)+1}H_1^{2(i+j)+1}\nabla\phi_{1,j} \biggr) 
  - \frac{4ij}{2(i+j)-1}H_1^{2(i+j)-1}(\underline{h}_1\delta)^{-2}\phi_{1,j} \biggr\}=0 \\
\makebox[29em]{}\mbox{for}\quad i=0,1,\ldots,N, \\
\displaystyle
  H_2^{p_i}\partial_t\zeta + \underline{h}_2 \sum_{j=0}^{N^*} \biggl\{ \nabla \cdot \biggl(
   \frac{1}{p_i+p_j+1} H_2^{p_i+p_j+1} \nabla\phi_{2,j}
   - \frac{p_j}{p_i+p_j} H_2^{p_i+p_j} \phi_{2,j} \underline{h}_2^{-1}\nabla b \biggr) \\
\displaystyle\qquad
   + \frac{p_i}{p_i+p_j} H_2^{p_i+p_j} \underline{h}_2^{-1}\nabla b \cdot \nabla\phi_{2,j}
   - \frac{p_ip_j}{p_i+p_j-1} H_2^{p_i+p_j-1} ((\underline{h}_2\delta)^{-2} + \underline{h}_2^{-2}|\nabla b|^2) \phi_{2,j} \biggr\} = 0 \\
\makebox[29em]{}\mbox{for}\quad i=0,1,\ldots,N^*, \\
\displaystyle
 \underline{\rho}_1\biggl\{ \sum_{j=0}^N H_1^{2j}\partial_t\phi_{1,j} + \frac12\biggl(
  \biggl|\sum_{j=0}^N H_1^{2j}\nabla\phi_{1,j} \biggr|^2
  + (\underline{h}_1\delta)^{-2}\biggl(\sum_{j=0}^N2jH_1^{2j-1} \phi_{1,j} \biggr)^2 \biggr) \biggr\} \\
\displaystyle
\qquad
 - \underline{\rho}_2\biggl\{ \sum_{j=0}^{N^*} H_2^{2j}\partial_t\phi_{2,j} + \frac12\biggl(
  \biggl|\sum_{j=0}^{N^*} ( H_2^{p_j}\nabla\phi_{2,j} - p_j H_2^{p_j-1}\phi_{2,j} \underline{h}_2^{-1}\nabla b )\biggr|^2 \\
\displaystyle\makebox[17em]{}
  + (\underline{h}_2\delta)^{-2}\biggl( \sum_{j=0}^{N^*} p_j H_2^{p_j-1} \phi_{2,j} \biggr)^2 \biggr)
   \biggr\}
 - \zeta = 0,
\end{cases}
\end{equation}
where we used the notational convention $0/0 = 0$, and
\begin{equation}\label{thicknesses}
H_1(\bm{x},t) := 1 - \underline{h}_1^{-1}\zeta(\bm{x},t), \qquad 
H_2(\bm{x},t) := 1 + \underline{h}_2^{-1}\zeta(\bm{x},t) - \underline{h}_2^{-1}b(\bm{x}).
\end{equation}
We impose the initial conditions to the Kakinuma model of the form 
\begin{equation}\label{Kaki:IC}
(\zeta,\bm{\phi}_1,\bm{\phi}_2)=(\zeta_{(0)},\bm{\phi}_{1(0)},\bm{\phi}_{2(0)})
 \makebox[3em]{at} t=0. 
\end{equation}

\subsection{Hamiltonian structures}\label{S.def-hamiltonian}
T. B. Benjamin and T. J. Bridges~\cite{BenjaminBridges1997} found that the full model for interfacial gravity waves 
can be written in Hamilton's canonical form 
\[
\partial_t\zeta = \frac{\delta\mathscr{H}^{\mbox{\rm\tiny IW}}}{\delta\phi}, \qquad
\partial_t\phi = -\frac{\delta\mathscr{H}^{\mbox{\rm\tiny IW}}}{\delta\zeta},
\]
where the canonical variable $\phi$ is defined by 
\begin{equation}\label{intro:cv}
	\phi = \underline{\rho}_2\phi_2 - \underline{\rho}_1\phi_1
\end{equation}
and the Hamiltonian $\mathscr{H}^{\mbox{\rm\tiny IW}}$ is the total 
energy $\mathscr{E}$ written in terms of the canonical variables $(\zeta,\phi)$. 
Specifically, $\mathscr{E}$ is the sum of the kinetic energies of the fluids in the upper 
and the lower layers and the potential energy due to the gravity defined as
\begin{align*}
\mathscr{E}
&:= \sum_{\ell=1,2}\iint_{\Omega_\ell(t)}\frac12\underline{\rho}_\ell
 \bigl( |\nabla\Phi_\ell(\bm{x},z,t)|^2 + \delta^{-2}(\partial_z\Phi_\ell(\bm{x},z,t))^2 \bigr) \mathrm{d}\bm{x}\mathrm{d}z 
 + \int_{\mathbf{R}^n}\frac12\zeta(\bm{x},t)^2\mathrm{d}\bm{x} \\
&= \sum_{\ell=1,2}\frac12\underline{\rho}_\ell(\Lambda_\ell(\zeta)\phi_\ell(t),\phi_\ell(t))_{L^2}
 + \frac12\|\zeta(t)\|_{L^2}^2.
\end{align*}
Here and in what follows, we denote simply $\Lambda_1(\zeta)=\Lambda_1(\zeta,\delta,\underline{h}_1)$ 
and  $\Lambda_2(\zeta)=\Lambda_2(\zeta,b,\delta,\underline{h}_2)$. 
It follows from the kinematic boundary conditions on the interface that 
$\Lambda_1(\zeta)\phi_1+\Lambda_2(\zeta)\phi_2=0$, so that $\phi_1$ and $\phi_2$ can be written in terms of 
the canonical variables $(\zeta,\phi)$ as 
\[
\begin{cases}
\phi_1 = -(\underline{\rho}_1\Lambda_2(\zeta)+\underline{\rho}_2\Lambda_1(\zeta))^{-1}\Lambda_2(\zeta)\phi, \\
\phi_2 = (\underline{\rho}_1\Lambda_2(\zeta)+\underline{\rho}_2\Lambda_1(\zeta))^{-1}\Lambda_1(\zeta)\phi.
\end{cases}
\]
Therefore, the Hamiltonian $\mathscr{H}^{\mbox{\rm\tiny IW}}(\zeta,\phi)$ of the full model for interfacial gravity waves 
is given explicitly by 
\begin{equation}\label{Hamiltonian-full-model}
\mathscr{H}^{\mbox{\rm\tiny IW}}(\zeta,\phi)
= \frac12((\underline{\rho}_1\Lambda_2(\zeta)+\underline{\rho}_2\Lambda_1(\zeta))^{-1}\Lambda_1(\zeta)\phi, 
  \Lambda_2(\zeta)\phi)_{L^2} + \frac12\|\zeta\|_{L^2}^2.
\end{equation}

As was shown in the companion paper~\cite{DucheneIguchi2020}, the Kakinuma model~\eqref{Kakinuma-dimensionless} 
also enjoys a Hamiltonian structure analogous to that of the full model for interfacial gravity waves. 
The canonical variables are the elevation of the interface $\zeta$ and $\phi$ defined by 
\begin{align}\label{canonical-variable}
\phi(\bm{x},t) 
&:= \underline{\rho}_2\Phi_2^\mathrm{app}(\bm{x},\zeta(\bm{x},t),t)
  - \underline{\rho}_1\Phi_1^\mathrm{app}(\bm{x},\zeta(\bm{x},t),t) \\
&= \underline{\rho}_2\sum_{i=0}^{N^*} H_2(\bm{x},t)^{p_i}\phi_{2,i}(\bm{x},t)
  - \underline{\rho}_1\sum_{i=0}^N H_1(\bm{x},t)^{2i}\phi_{1,i}(\bm{x},t),
  \nonumber
\end{align}
where $\Phi_\ell^\mathrm{app}$ $(\ell=1,2)$ are nondimensional versions of the approximate velocity potentials, 
which are defined by 
\begin{equation}\label{Approximation-nondim} 
\begin{cases}
 \displaystyle
  \Phi_1^\mathrm{app}(\bm{x},z,t) := \sum_{i=0}^N (1-\underline{h}_1^{-1}z)^{2i}\phi_{1,i}(\bm{x},t), \\[2.5ex]
 \displaystyle
  \Phi_2^\mathrm{app}(\bm{x},z,t) := \sum_{i=0}^{N^*} (1+\underline{h}_2^{-1}(z-b(\bm{x})))^{p_i}\phi_{2,i}(\bm{x},t),
\end{cases}
\end{equation}
and $H_\ell$ $(\ell=1,2)$ are depths of the upper and lower layers defined by~\eqref{thicknesses}. 
We note that if the canonical variables $(\zeta,\phi)$ are given, then the Kakinuma model 
\eqref{Kakinuma-dimensionless} determines $\bm{\phi}_1=(\phi_{1,0},\phi_{1,1},\ldots,\phi_{1,N})^\mathrm{T}$ 
and $\bm{\phi}_2=(\phi_{2,0},\phi_{2,1},\ldots,\phi_{2,N^*})^\mathrm{T}$, which are unique up to an additive 
constant of the form $(\mathcal{C}\underline{\rho}_1,\mathcal{C}\underline{\rho}_2)$ to $(\phi_{1,0},\phi_{2,0})$. 
For details, we refer to~\cite[Subsection 8.1]{DucheneIguchi2020} and Lemma~\ref{L.elliptic} in Section~\ref{S.elliptic}.
Then, the Hamiltonian $\mathscr{H}^{\mbox{\rm\tiny K}}(\zeta,\phi)$ of the Kakinuma model is given by 
\begin{equation}\label{Hamiltonian-Kakinuma}
\mathscr{H}^{\mbox{\rm\tiny K}}(\zeta,\phi)
:= \sum_{\ell=1,2}\iint_{\Omega_\ell}\frac12\underline{\rho}_\ell
 \bigl( |\nabla\Phi_\ell^\mathrm{app}(\bm{x},z,t)|^2
  + \delta^{-2}(\partial_z\Phi_\ell^\mathrm{app}(\bm{x},z,t))^2 \bigr) \mathrm{d}\bm{x}\mathrm{d}z 
 + \int_{\mathbf{R}^n}\frac12\zeta(\bm{x},t)^2\mathrm{d}\bm{x}. 
\end{equation}

\section{Statements of the main results}\label{S.main-results}
Before stating the main results in this paper, let us introduce some notations
which allow in particular to rewrite~\eqref{Kakinuma-dimensionless} in a compact form.
We introduce second order differential operators $L_{1,ij} = L_{1,ij}(H_1,\delta,\underline{h}_1)$ $(i,j=0,1,\ldots,N)$ 
and $L_{2,ij} = L_{2,ij}(H_2,b,\delta,\underline{h}_2)$ $(i,j=0,1,\ldots,N^*)$ by 
\begin{align}\label{operatorLij-delta}
L_{1,ij}\varphi_{1,j}
&:= - \nabla\cdot\biggl( \frac{1}{2(i+j)+1}H_1^{2(i+j)+1}\nabla\varphi_{1,j} \biggr) 
   + \frac{4ij}{2(i+j)-1}H_1^{2(i+j)-1}(\underline{h}_1\delta)^{-2}\varphi_{1,j}, \\
L_{2,ij}\varphi_{2,j} \label{operatorLij2-delta}
&:= - \nabla\cdot\biggl(
   \frac{1}{p_i+p_j+1}H_2^{p_i+p_j+1}\nabla\varphi_{2,j}
   - \frac{p_j}{p_i+p_j}H_2^{p_i+p_j}\varphi_{2,j}\underline{h}_2^{-1}\nabla b\biggr) \\[0.5ex]
&\quad\,
  - \frac{p_i}{p_i+p_j}H_2^{p_i+p_j}\underline{h}_2^{-1}\nabla b\cdot\nabla\varphi_{2,j} \nonumber \\[0.5ex] 
&\quad\,
   + \frac{p_ip_j}{p_i+p_j-1}H_2^{p_i+p_j-1}((\underline{h}_2\delta)^{-2}+ \underline{h}_2^{-2}|\nabla b|^2)\varphi_{2,j},
  \nonumber
\end{align}
where we use the notational convention $0/0 = 0$.
Notice that we have $(L_{\ell,ij})^*=L_{\ell,ji}$ for $\ell=1,2$, where $(L_{\ell,ij})^*$ is the adjoint operator of 
$L_{\ell,ij}$ in $L^2(\mathbf{R}^n)$.  
We put ${\bm \phi}_1 := (\phi_{1,0},\phi_{1,1},\ldots,\phi_{1,N})^\mathrm{T}$, 
${\bm \phi}_2 := (\phi_{2,0},\phi_{2,1},\ldots,\phi_{2,N^*})^\mathrm{T}$, and 
\begin{equation}\label{def-l1l2}
\begin{cases}
{\bm l}_1(H_1) := (1,H_1^2,H_1^4,\ldots,H_1^{2N})^\mathrm{T},\\
{\bm l}_1'(H_1) := (0,2H_1,\ldots,2N H_1^{2N-1})^\mathrm{T}, \\
 {\bm l}_1''(H_1) := (0,2,\ldots,2N(2N-1) H_1^{2N-2})^\mathrm{T} , \\
 {\bm l}_2(H_2) := (1,H_2^{p_1},H_2^{p_2},\ldots,H_2^{p_{N^*}})^\mathrm{T}, \\
{\bm l}_2'(H_2) := (0,p_1 H_2^{p_1-1},\ldots,p_{N^*}H_2^{p_{N^*}})^\mathrm{T}, \\
{\bm l}_2''(H_2) := (0,p_1(p_1-1) H_2^{p_1-2},\ldots,p_{N^*}(p_{N^*}-1)H_2^{p_{N^*}})^\mathrm{T}, 
\end{cases}
\end{equation}
and define ${\bm u}_\ell$ and $w_\ell$ for $\ell=1,2$, which represent approximately the horizontal and the vertical components 
of the velocity field on the interface from the water region $\Omega_\ell(t)$, by 
\begin{equation}\label{def-uw}
 \left\{\begin{array}{ll}
 \displaystyle {\bm u}_1 := ({\bm l}_1(H_1) \otimes \nabla)^\mathrm{T}{\bm \phi}_{1}, \qquad
  & \displaystyle w_1 := - {\bm l}_1^\prime(H_1) \cdot {\bm\phi}_{1}, \\
 \displaystyle {\bm u}_2 := ({\bm l}_2(H_2) \otimes \nabla)^\mathrm{T}{\bm\phi}_{2}
  - ( {\bm l}_2^\prime(H_2) \cdot {\bm\phi}_{2})\underline{h}_2^{-1}\nabla b, \qquad 
  & \displaystyle w_2 := {\bm l}_2^\prime(H_2) \cdot {\bm\phi}_{2}.
 \end{array}
 \right.
\end{equation}
Then, denoting $L_1 := (L_{1,ij})_{0\leq i,j\leq N}$ and $L_2 := (L_{2,ij})_{0\leq i,j\leq N^*}$ 
we can write the Kakinuma model~\eqref{Kakinuma-dimensionless} more compactly as 
\begin{equation}\label{Kakinuma-dimensionless-compact} 
\begin{cases}
 \displaystyle
 {\bm l}_1(H_1)\partial_t\zeta + \underline{h}_1 L_1(H_1,\delta,\underline{h}_1){\bm \phi}_1 = {\bm 0}, \\
 \displaystyle
 {\bm l}_2(H_2)\partial_t\zeta - \underline{h}_2 L_2(H_2,b,\delta,\underline{h}_2){\bm \phi}_2 = {\bm 0}, \\
 \underline{\rho}_1\bigl\{ {\bm l}_1(H_1) \cdot \partial_t{\bm \phi}_1 
   + \frac12\bigl( |{\bm u}_1|^2 + (\underline{h}_1\delta)^{-2} w_1^2 \bigr) \bigr\} \\
 \quad
 - \underline{\rho}_2\bigl\{ {\bm l}_2(H_2) \cdot \partial_t{\bm \phi}_2 
  + \frac12\bigl( |{\bm u}_2|^2 +  (\underline{h}_2\delta)^{-2} w_2^2 \bigr) \bigr\} 
 - \zeta = 0. 
\end{cases}
\end{equation}
By eliminating $\partial_t\zeta$ from the first two vectorial identities in~\eqref{Kakinuma-dimensionless-compact}, we obtain $N+N^*+1$ scalar relations 
which are necessary conditions for the existence of solutions to the Kakinuma model, as stated below. 
Introducing linear operators $\mathcal{L}_{1,i} := \mathcal{L}_{1,i}(H_1,\delta,\underline{h}_1)$ $(i=0,\ldots,N)$ 
acting on ${\bm \varphi}_1 = (\varphi_{1,0},\ldots,\varphi_{1,N})^\mathrm{T}$ and 
$\mathcal{L}_{2,i} := \mathcal{L}_{2,i}(H_2,b,\delta,\underline{h}_2)$ $(i=0,\ldots,N^*)$ 
acting on ${\bm \varphi}_2 = (\varphi_{2,0},\ldots,\varphi_{2,N^*})^\mathrm{T}$ by 
\begin{equation}\label{def-calL}
\begin{cases}
 \displaystyle
  \mathcal{L}_{1,0} {\bm \varphi}_1 
   := \sum_{j=0}^{N} L_{1,0j}\varphi_{1,j}, \\
 \displaystyle
  \mathcal{L}_{1,i} {\bm \varphi}_1
   := \sum_{j=0}^N ( L_{1,ij}\varphi_{1,j} - H_1^{2i}L_{1,0j}\varphi_{1,j} )
   \quad\mbox{for}\quad i=1,2,\ldots,N, \\
 \displaystyle
  \mathcal{L}_{2,0} {\bm \varphi}_2 
   := \sum_{j=0}^{N^*} L_{2,0j}\varphi_{2,j}, \\
 \displaystyle
  \mathcal{L}_{2,i} {\bm \varphi}_2
  := \sum_{j=0}^{N^*} ( L_{2,ij}\varphi_{2,j} - H_2^{p_i}L_{2,0j}\varphi_{2,j} )
   \quad\mbox{for}\quad i=1,2,\ldots,N^*, 
\end{cases}
\end{equation}
the necessary conditions can be written simply as 
\begin{equation}\label{necessary-delta}
\begin{cases}
 \mathcal{L}_{1,i}(H_1,\delta,\underline{h}_1) {\bm \phi}_1 = 0 \quad\mbox{for}\quad i=1,2,\ldots,N, \\
 \mathcal{L}_{2,i}(H_2,b,\delta,\underline{h}_2) {\bm \phi}_2 = 0 \quad\mbox{for}\quad i=1,2,\ldots,N^*, \\
 \underline{h}_1\mathcal{L}_{1,0}(H_1,\delta,\underline{h}_1) {\bm \phi}_1
  + \underline{h}_2\mathcal{L}_{2,0}(H_2,b,\delta,\underline{h}_2) {\bm \phi}_2  = 0.
\end{cases}
\end{equation}
Hereafter, these necessary conditions will be referred to as the compatibility conditions. 
Notice that under these compatibility conditions we have for $\ell=1,2$
\begin{equation}\label{expression-L_k}
L_\ell\bm{\phi}_\ell = \bm{l}_\ell\mathcal{L}_{\ell,0}\bm{\phi}_\ell,
\end{equation}
where $\bm{l}_\ell=\bm{l}_\ell(H_\ell)$ and similar simplifications of notations will be used in the following 
without any comments. 
In connection with the stability condition~\eqref{intro:stability}, we introduce a function 
\begin{align}\label{def-a}
a := 1 &+ \underline{\rho}_1\underline{h}_1^{-1}\{ \bm{l}_1'(H_1)\cdot(\partial_t+\bm{u}_1\cdot\nabla)\bm{\phi}_1
  - (\underline{h}_1\delta)^{-2}w_1\bm{l}_1''(H_1)\cdot\bm{\phi}_1 \} \\
&
 + \underline{\rho}_2\underline{h}_2^{-1}\{ \bm{l}_2'(H_2)\cdot(\partial_t+\bm{u}_2\cdot\nabla)\bm{\phi}_2
  + \bigl( (\underline{h}_2\delta)^{-2}w_2 - \underline{h}_2^{-1}\nabla b\cdot\bm{u}_2\bigr)\bm{l}_2''(H_2)\cdot\bm{\phi}_2 \},
 \nonumber
\end{align}
which corresponds to $- (\partial_z (P_2^\mathrm{app} - P_1^\mathrm{app} ))|_{\Gamma(t)}$ in the stability condition.

Our first main result in this paper is the existence of the solution to the initial value problem 
\eqref{Kakinuma-dimensionless}--\eqref{Kaki:IC} for the Kakinuma model on a time interval independent of 
parameters, especially, the shallowness parameters $\delta_1=\underline{h}_1\delta$ and $\delta_2=\underline{h}_2\delta$ 
together with a uniform bound of the solution. 
For simplicity, we denote $H_{\ell(0)}:=H_\ell|_{t=0}$, $\bm{u}_{\ell(0)}:=\bm{u}_\ell|_{t=0}$ for $\ell=1,2$, 
and $a_{(0)}:=a|_{t=0}$, which can be written in terms of the initial data according to the initial condition~\eqref{Kaki:IC}. 
Although the function $a$ includes the terms $(\partial_t\bm{\phi}_\ell')|_{t=0}$ for $\ell=1,2$, where 
$\bm{\phi}_1'=(\phi_{1,1},\ldots,\phi_{1,N})^\mathrm{T}$ and  $\bm{\phi}_2'=(\phi_{2,1},\ldots,\phi_{2,N^*})^\mathrm{T}$, 
and the hypersurface $t=0$ is characteristic for the Kakinuma model, 
we can uniquely determine them in terms of the initial data. 
For details, we refer to Remark~\ref{remark-ID}.

\begin{theorem}\label{theorem-uniform}
Let $c_0, M_0, \underline{h}_\mathrm{min}$ be positive constants and $m$ an integer such that ${m>\frac{n}{2}+1}$. 
There exist a time $T>0$ and a constant $M>0$ such that for any positive parameters 
$\underline{\rho}_1, \underline{\rho}_2, \underline{h}_1, \underline{h}_2, \delta$ satisfying the natural 
restrictions~\eqref{parameters}, $\underline{h}_1\delta, \underline{h}_2\delta \leq 1$, as well as the condition 
${\underline{h}_\mathrm{min} \leq \underline{h}_1, \underline{h}_2}$, 
if the initial data $(\zeta_{(0)},\bm{\phi}_{1(0)},\bm{\phi}_{2(0)})$ 
and the bottom topography $b$ satisfy 
\begin{equation}\label{uniform-ini}
\begin{cases}
\displaystyle
\|\zeta_{(0)}\|_{H^m}^2 + \sum_{\ell=1,2}\underline{\rho}_\ell\underline{h}_\ell\bigl(
 \|\nabla\bm{\phi}_{\ell(0)}\|_{H^m}^2 + (\underline{h}_\ell\delta)^{-2}\|\bm{\phi}_{\ell(0)}'\|_{H^m}^2\bigr) \leq M_0, \\
\underline{h}_2^{-1}\bigl(\|b\|_{W^{m+1,\infty}} + (\underline{h}_2\delta)\|b\|_{W^{m+2,\infty}} \bigr) \leq M_0,
\end{cases}
\end{equation}
the non-cavitation assumption 
\begin{equation}\label{NonCavitation}
H_{1(0)}(\bm{x}) \geq c_0, \quad H_{2(0)}(\bm{x}) \geq c_0 \quad\mbox{for}\quad \bm{x}\in\mathbf{R}^n,
\end{equation}
the stability condition
\begin{equation}\label{Stability}
 a_{(0)}(\bm{x})- \frac{\underline{\rho}_1\underline{\rho}_2}{
  \underline{\rho}_1\underline{h}_2H_{2(0)}(\bm{x})\alpha_2 + \underline{\rho}_2\underline{h}_1H_{1(0)}(\bm{x})\alpha_1}
  |{\bm u}_{1(0)}(\bm{x})-{\bm u}_{2(0)}(\bm{x})|^2 \geq c_0 \quad\mbox{for}\quad \bm{x}\in\mathbf{R}^n,
\end{equation}
with positive constants $\alpha_1$ and $\alpha_2$ defined by~\eqref{def-alpha-intro}, and the compatibility conditions 
\begin{equation}\label{Compatibility}
\begin{cases}
 \mathcal{L}_{1,i}(H_{1(0)},\delta,\underline{h}_1) {\bm \phi}_{1(0)} = 0 \quad\mbox{for}\quad i=1,2,\ldots,N, \\
 \mathcal{L}_{2,i}(H_{2(0)},b,\delta,\underline{h}_2) {\bm \phi}_{2(0)} = 0 \quad\mbox{for}\quad i=1,2,\ldots,N^*, \\
 \underline{h}_1\mathcal{L}_{1,0}(H_{1(0)},\delta,\underline{h}_1)  {\bm \phi}_{1(0)}
  +  \underline{h}_2\mathcal{L}_{2,0}(H_{2(0)},b,\delta,\underline{h}_2)  {\bm \phi}_{2(0)} = 0,
\end{cases}
\end{equation}
then the initial value problem~\eqref{Kakinuma-dimensionless}--\eqref{Kaki:IC} has a unique solution 
$(\zeta,\bm{\phi}_1,\bm{\phi}_2)$ on the time interval $[0,T]$ satisfying 
\[
\begin{cases}
\zeta,\nabla\phi_{1,0},\nabla\phi_{2,0} \in C([0,T];H^m)\cap  C^1([0,T];H^{m-1}), \\
{\bm \phi}_1^{\prime},{\bm \phi}_2^{\prime} \in C([0,T];H^{m+1})\cap  C^1([0,T];H^{m}),
\end{cases}
\]
where we recall the notation 
${\bm \phi}_1^{\prime} = (\phi_{1,1},\phi_{1,2},\ldots,\phi_{1,N})^\mathrm{T}$ and 
${\bm \phi}_2^{\prime} = (\phi_{2,1},\phi_{2,2},\ldots,\phi_{2,N^*})^\mathrm{T}$. 
Moreover, the solution satisfies the uniform bound 
\begin{equation}\label{uniform-sol}
\|\zeta(t)\|_{H^m}^2 + \sum_{\ell=1,2}\underline{\rho}_\ell\underline{h}_\ell\bigl(
 \|\nabla\bm{\phi}_\ell(t)\|_{H^m}^2 + (\underline{h}_\ell\delta)^{-2}\|\bm{\phi}_\ell'(t)\|_{H^m}^2\bigr) \leq M
\end{equation}
for $t\in[0,T]$ together with 
\begin{equation}\label{uniform-below}
\begin{cases}
\displaystyle
 a(\bm{x},t)- \frac{\underline{\rho}_1\underline{\rho}_2}{
  \underline{\rho}_1\underline{h}_2H_2(\bm{x},t)\alpha_2 + \underline{\rho}_2\underline{h}_1H_1(\bm{x},t)\alpha_1}
  |{\bm u}_1(\bm{x},t)-{\bm u}_2(\bm{x},t)|^2 \geq c_0/2, \\
\displaystyle
 H_1(\bm{x},t) \geq c_0/2, \quad H_2(\bm{x},t) \geq c_0/2
 \qquad \mbox{for}\quad \bm{x}\in\mathbf{R}^n, t\in[0,T].
\end{cases}
\end{equation}
\end{theorem}
\begin{remark}\label{remark-alpha}
The constants $\alpha_1$ and $\alpha_2$ are defined by
\begin{equation}\label{def-alpha-intro}
	\alpha_\ell := \frac{\det A_{\ell,0}}{\det \tilde{A}_{\ell,0}}, \qquad
	\tilde{A}_{\ell,0} := 
	\begin{pmatrix}
		0 & \bm{1}^\mathrm{T} \\
		-\bm{1} & A_{\ell,0}
	\end{pmatrix},
\end{equation}
for $\ell=1,2$, where $\bm{1}:=(1,\ldots,1)^\mathrm{T}$ and the matrices $A_{1,0}$ and $A_{2,0}$ are defined by
\[
	\begin{cases}
		\displaystyle
		A_{1,0} := \left( \frac{1}{2(i+j)+1} \right)_{0\leq i,j\leq N}, \\
		\displaystyle
		A_{2,0} := \left( \frac{1}{p_i+p_j+1}\right)_{0\leq i,j\leq N^*}.
	\end{cases}
\]
Hence, $\alpha_1$ and $\alpha_2$ are positive constants depending only on $N$ and the nonnegative integers $0=p_0<p_1<\ldots<p_{N^*}$, respectively, 
and go to $0$ as $N,N^*\to\infty$. 
\end{remark}
\begin{remark}\label{remark-ini}
{\rm 
It is easy to check that the non-cavitation assumption~\eqref{NonCavitation} and the stability condition 
\eqref{Stability} are automatically satisfied for small initial data $(\zeta_{(0)},\bm{\phi}_{1(0)},\bm{\phi}_{2(0)})$ 
and small bottom topography $b$, whereas an arrangement of nontrivial initial data satisfying the compatibility 
conditions~\eqref{Compatibility} together with the uniform bound~\eqref{uniform-ini} is a non-trivial issue. 
To this end, we use the canonical variable $\phi$ defined by~\eqref{canonical-variable}, 
which can be written as 
\begin{equation}\label{canonical-variable2}
\phi = \underline{\rho}_2\bm{l}_2(H_2)\cdot\bm{\phi}_2 - \underline{\rho}_1\bm{l}_1(H_1)\cdot\bm{\phi}_1. 
\end{equation}
Given the initial data $(\zeta_{(0)},\phi_{(0)})$ for the canonical variables $(\zeta,\phi)$, 
and the bottom topography $b$, the necessary conditions~\eqref{necessary-delta} and the above relation 
\eqref{canonical-variable2} determine the initial data $(\bm{\phi}_{1(0)},\bm{\phi}_{2(0)})$ for the Kakinuma model~\eqref{Kakinuma-dimensionless}--\eqref{Kaki:IC}
satisfying the compatibility conditions~\eqref{Compatibility} and the uniform bound~\eqref{uniform-ini}, 
which is unique up to an additive constant of the form $(\mathcal{C}\underline{\rho}_2,\mathcal{C}\underline{\rho}_1)$ 
to $(\phi_{1,0(0)},\phi_{2,0(0)})$. 
In fact, we have the following proposition, which is a simple corollary of Lemma~\ref{L.elliptic} given in 
Section~\ref{S.elliptic}. 
}
\end{remark}

\begin{proposition}\label{preparation-ini}
Let $c_0, M_0$ be positive constants and $m$ an integer such that $m>\frac{n}{2}+1$. 
There exists a positive constant $C$ such that for any positive parameters 
$\underline{\rho}_1, \underline{\rho}_2, \underline{h}_1, \underline{h}_2, \delta$ satisfying the natural 
restrictions~\eqref{parameters} and $\underline{h}_1\delta, \underline{h}_2\delta \leq 1$, 
if the initial data $(\zeta_{(0)},\phi_{(0)})\in H^m\times\mathring{H}^m$ of the canonical variables, 
the bottom topography $b\in W^{m,\infty}$, and initial depths $H_{1(0)} := 1 - \underline{h}_1^{-1}\zeta_{(0)}$ and 
$H_{2(0)} := 1 + \underline{h}_2^{-1}\zeta_{(0)} - \underline{h}_2^{-1}b$ satisfy 
\[
\begin{cases}
 \underline{h}_1^{-1}\|\zeta_{(0)}\|_{H^m} + \underline{h}_2^{-1}\|\zeta_{(0)}\|_{H^m}
  + \underline{h}_2^{-1}\|b\|_{W^{m,\infty}} \leq M_0, \\
 H_{1(0)}(\bm{x}) \geq c_0, \quad H_{2(0)}(\bm{x}) \geq c_0
  \quad\mbox{for}\quad \bm{x}\in\mathbf{R}^n,
\end{cases}
\]
then there exist initial data $(\bm{\phi}_{1(0)},\bm{\phi}_{2(0)})$ satisfying the compatibility 
conditions~\eqref{Compatibility} as well as 
$\phi_{(0)} = \underline{\rho}_2\bm{l}_2(H_{2(0)})\cdot\bm{\phi}_{2(0)}
 - \underline{\rho}_1\bm{l}_1(H_{1(0)})\cdot\bm{\phi}_{1(0)}$. 
Moreover, we have 
\[
\sum_{\ell=1,2}\underline{\rho}_\ell\underline{h}_\ell\bigl(
 \|\nabla\bm{\phi}_{\ell(0)}\|_{H^{m-1}}^2 + (\underline{h}_\ell\delta)^{-2}\|\bm{\phi}_{\ell(0)}'\|_{H^{m-1}}^2\bigr)
\leq C\|\nabla\phi_{(0)}\|_{H^{m-1}}^2.
\]
\end{proposition}

The next theorem shows that the Kakinuma model~\eqref{Kakinuma-dimensionless} is consistent with the full model for interfacial gravity waves 
\eqref{full-model-evolution} at order $O((\underline{h}_1\delta)^{4N+2}+(\underline{h}_2\delta)^{4N+2})$ under the special choice of the 
indices $p_0,p_1,\ldots,p_{N^*}$ as 
\begin{enumerate}
\item[(H1)]
$N^*=N$ and $p_i=2i$ $(i=0,1,\ldots,N)$ in the case of the flat bottom $b(\bm{x})\equiv0$, 
\item[(H2)]
$N^*=2N$ and $p_i=i$ $(i=0,1,\ldots,2N)$ in the case with general bottom topographies. 
\end{enumerate}

\begin{theorem}\label{theorem-consistency1}
Let $c, M$ be positive constants and $m$ an integer such that $m \geq 4(N+1)$ and $m>\frac{n}{2}+1$. 
We assume {\rm (H1)} or {\rm (H2)}. 
There exists a positive constant $C$ such that for any positive parameters $\underline{\rho}_1, \underline{\rho}_2, \underline{h}_1, \underline{h}_2, \delta$ satisfying 
$\underline{h}_1\delta, \underline{h}_2\delta \leq 1$ and for any solution $(\zeta,\bm{\phi}_1,\bm{\phi}_2)$ to the Kakinuma model~\eqref{Kakinuma-dimensionless} 
on a time interval $[0,T]$ with a bottom topography $b\in W^{m+1,\infty}$ satisfying 
\begin{equation}\label{cond.consisitency}
\begin{cases}
 \underline{h}_1^{-1}\|\zeta(t)\|_{H^m} + \underline{h}_2^{-1}\|\zeta(t)\|_{H^m}
  + \underline{h}_2^{-1}\|b\|_{W^{m+1,\infty}} \leq M, \\
 H_1(\bm{x},t)\geq c, \quad H_2(\bm{x},t)\geq c \quad\mbox{for}\quad \bm{x}\in\mathbf{R}^n, t\in[0,T],
\end{cases}
\end{equation}
if we define $\phi_\ell:=\bm{l}_\ell(H_\ell)\cdot\bm{\phi}_\ell$ for $\ell=1,2$, then $(\zeta,\phi_1,\phi_2)$ 
satisfy approximately the full model for interfacial gravity waves as 
\[
\begin{cases}
\partial_t\zeta + \Lambda_1(\zeta,\delta,\underline{h}_1)\phi_1=\mathfrak{r}_1,\\
\partial_t\zeta - \Lambda_2(\zeta,b,\delta,\underline{h}_2)\phi_2 =\mathfrak{r}_2,\\
\displaystyle
\underline{\rho}_1\biggl( \partial_t\phi_1 + \frac12|\nabla\phi_1|^2
 - \frac12\delta^2 \frac{(\Lambda_1(\zeta,\delta,\underline{h}_1)\phi_1
  - \nabla\zeta \cdot \nabla\phi_1 )^2}{1+\delta^2|\nabla\zeta|^2} \biggr) \\
\displaystyle\quad
 - \underline{\rho}_2\biggl( \partial_t\phi_2 + \frac12|\nabla\phi_2|^2
  - \frac12\delta^2 \frac{(\Lambda_2(\zeta,b,\delta,\underline{h}_2)\phi_2
   + \nabla\zeta \cdot \nabla\phi_2 )^2}{1+\delta^2|\nabla\zeta|^2} \biggr)
 - \zeta = \mathfrak{r}_0.
\end{cases}
\]
Here, the errors $(\mathfrak{r}_1,\mathfrak{r}_2,\mathfrak{r}_0)$ satisfy 
\[
\begin{cases}
 \|\mathfrak{r}_\ell(t)\|_{H^{m-4(N+1)}}
  \leq C\underline{h}_\ell(\underline{h}_\ell\delta)^{4N+2}\|\nabla\phi_\ell(t)\|_{H^{m-1}}
  \quad (\ell=1,2), \\
 \displaystyle
 \|\mathfrak{r}_0(t)\|_{H^{m-4(N+1)}} \leq 
  C\sum_{\ell=1,2}\underline{\rho}_\ell(\underline{h}_\ell\delta)^{4N+2}\|\nabla\phi_\ell(t)\|_{H^{m-1}}^2
\end{cases}
\]
for $t\in[0,T]$. 
\end{theorem}

Particularly, we see that under the special choice of indices (H1) or (H2) the solutions to the Kakinuma model~\eqref{Kakinuma-dimensionless}--\eqref{Kaki:IC}
constructed in Theorem~\ref{theorem-uniform} satisfy approximately the full model for interfacial gravity 
waves~\eqref{full-model-evolution} with the choice $\phi_\ell=\bm{l}_\ell(H_\ell)\cdot\bm{\phi}_\ell$ $(\ell=1,2)$ and that 
the error is of order $O((\underline{h}_1\delta)^{4N+2}+(\underline{h}_2\delta)^{4N+2})$. 

Conversely, the next theorem shows that the full model for interfacial gravity waves is consistent with the 
Kakinuma model at order $O((\underline{h}_1\delta)^{4N+2}+(\underline{h}_2\delta)^{4N+2})$ 
under the special choice of indices (H1) or (H2).

\begin{theorem}\label{theorem-consistency2}
Let $c, M$ be positive constants and $m$ an integer such that $m \geq 4(N+1)$ and $m>\frac{n}{2}+1$. 
We assume {\rm (H1)} or {\rm (H2)}. 
There exists a positive constant $C$ such that for any positive parameters 
$\underline{\rho}_1, \underline{\rho}_2, \underline{h}_1, \underline{h}_2, \delta$ 
satisfying $\underline{h}_1\delta, \underline{h}_2\delta \leq 1$ and for any 
solution $(\zeta,\phi_1,\phi_2)$ to the full model for interfacial gravity waves~\eqref{full-model-evolution} 
on a time interval $[0,T]$ with a bottom topography $b\in W^{m+1,\infty}$ satisfying~\eqref{cond.consisitency}, 
if we define $H_1$ and $H_2$ as in~\eqref{thicknesses} and $\bm{\phi}_1$ and $\bm{\phi}_2$ as the unique solutions to the problems 
\begin{equation}\label{Conditions}
\begin{cases}
 \bm{l}_1(H_1)\cdot\bm{\phi}_1=\phi_1, \quad \mathcal{L}_{1,i}(H_1,\delta,\underline{h}_1)\bm{\phi}_1=0
  \quad\mbox{for}\quad i=1,2,\ldots,N, \\
 \bm{l}_2(H_2)\cdot\bm{\phi}_2=\phi_2, \quad \mathcal{L}_{2,i}(H_2,b,\delta,\underline{h}_2)\bm{\phi}_2=0
  \quad\mbox{for}\quad i=1,2,\ldots,N^*,
\end{cases}
\end{equation}
then $(\zeta,\bm{\phi}_1,\bm{\phi}_2)$ satisfy approximately the Kakinuma model as
\[
\begin{cases}
 \displaystyle
 {\bm l}_1(H_1)\underline{h}_1^{-1}
 \partial_t\zeta + L_1(H_1,\delta,\underline{h}_1){\bm \phi}_1
  = \tilde{\bm{\mathfrak{r}}}_1, \\
 \displaystyle
 {\bm l}_2(H_2)\underline{h}_2^{-1}
 \partial_t\zeta - L_2(H_2,b,\delta,\underline{h}_2){\bm \phi}_2
  = \tilde{\bm{\mathfrak{r}}}_2, \\
 \underline{\rho}_1\bigl\{ {\bm l}_1(H_1) \cdot \partial_t{\bm \phi}_1 
   + \frac12\bigl( |{\bm u}_1|^2 + (\underline{h}_1\delta)^{-2} w_1^2 \bigr) \bigr\} \\
 \quad
 - \underline{\rho}_2\bigl\{ {\bm l}_2(H_2) \cdot \partial_t{\bm \phi}_2 
  + \frac12\bigl( |{\bm u}_2|^2 +  (\underline{h}_2\delta)^{-2} w_2^2 \bigr) \bigr\} 
 - \zeta = \tilde{\mathfrak{r}}_0. 
\end{cases}
\]
Here, the errors $(\tilde{\bm{\mathfrak{r}}}_1,\tilde{\bm{\mathfrak{r}}}_2,\tilde{\mathfrak{r}}_0)$ satisfy 
\begin{equation}\label{error-estimate}
\begin{cases}
 \|\tilde{\bm{\mathfrak{r}}}_\ell(t)\|_{H^{m-4(N+1)}}
  \leq C (\underline{h}_\ell\delta)^{4N+2}\|\nabla\phi_\ell(t)\|_{H^{m-1}}
  \quad (\ell=1,2), \\
 \displaystyle
 \|\tilde{\mathfrak{r}}_0(t)\|_{H^{m-4(N+1)}} \leq 
  C\sum_{\ell=1,2}\underline{\rho}_\ell(\underline{h}_\ell\delta)^{4N+2}\|\nabla\phi_\ell(t)\|_{H^{m-1}}^2
\end{cases}
\end{equation}
for $t\in[0,T]$. 
\end{theorem}

\begin{remark}
{\rm
The unique existence of the solutions $\bm{\phi}_1$ and $\bm{\phi}_2$ to the problems~\eqref{Conditions} 
is guaranteed by Lemma~\ref{L.Lambda} below under an additional assumption $\phi_1(\cdot,t),\phi_2(\cdot,t) 
\in \mathring{H}^m$. 
Lemma~\ref{L.Lambda} is essentially a simple corollary of~\cite[Lemma 3.4]{Iguchi2018-2}. 
}
\end{remark}

\begin{remark}\label{R.app-sol}
{\rm
In order to define the approximate solution $(\bm{\phi}_1,\bm{\phi}_2)$ to the Kakinuma model~\eqref{Kakinuma-dimensionless}
from the solution $(\zeta,\phi_1,\phi_2)$ to the full model, we can use, in place of~\eqref{Conditions}, 
the following system of equations 
\begin{equation}\label{Conditions-v2}
\begin{cases}
 \mathcal{L}_{1,i}(H_1,\delta,\underline{h}_1)\bm{\phi}_1=0 \quad\mbox{for}\quad i=1,2,\ldots,N, \\
 \mathcal{L}_{2,i}(H_1,b,\delta,\underline{h}_2)\bm{\phi}_2=0 \quad\mbox{for}\quad i=1,2,\ldots,N^*, \\
 \underline{h}_1\mathcal{L}_{1,0}(H_1,\delta,\underline{h}_1) {\bm \phi}_1
  + \underline{h}_2\mathcal{L}_{2,0}(H_2,b,\delta,\underline{h}_2) {\bm \phi}_2  = 0, \\
 \underline{\rho}_2\bm{l}_2(H_2)\cdot\bm{\phi}_2-\underline{\rho}_1\bm{l}_1(H_1)\cdot\bm{\phi}_1 = \phi,
\end{cases}
\end{equation}
where $\phi=\underline{\rho}_2\phi_2-\underline{\rho}_1\phi_1$ is the canonical variable for the full model 
for interfacial gravity waves. 
The above system is nothing but the compatibility conditions~\eqref{necessary-delta} together with 
the definition~\eqref{canonical-variable2} of the canonical variable for the Kakinuma model. 
The existence of the approximate solution $(\bm{\phi}_1,\bm{\phi}_2)$ is guaranteed by 
Lemma~\ref{L.elliptic} given in Section~\ref{S.elliptic}. 
Then, we have similar error estimates to~\eqref{error-estimate}. 
For details, we refer to Proposition~\ref{prop-consistency}. 
}
\end{remark}

The above Theorems~\ref{theorem-consistency1} and~\ref{theorem-consistency2} concern essentially the approximation of the equations. 
To give a rigorous justification of the Kakinuma model~\eqref{Kakinuma-dimensionless}
as a higher order shallow water approximation to the full model for interfacial gravity waves~\eqref{full-model-evolution}, one needs to give an error 
estimate between solutions to the Kakinuma model and that to the full model. 
However, we cannot expect to construct general solutions to the initial value problem for the full model for interfacial gravity waves 
because the initial value problem is ill-posed. 
Nevertheless, if we assume the existence of a solution to the full model with a uniform bound
with respect to the shallowness parameters $\delta_1=\underline{h}_1\delta$ and $\delta_2=\underline{h}_2\delta$, 
then we can give an error estimate with respect to a solution to the Kakinuma model by making use of the well-posedness of 
the initial value problem for the Kakinuma model as we can see in the following theorem.

\begin{theorem}\label{theorem-justification}
Let $c, M, \underline{h}_\mathrm{min}$ be positive constants and $m$ an integer such that ${m>\frac{n}{2}+4(N+1)}$. 
We assume {\rm (H1)} or {\rm (H2)}. 
Then, there exist a time $T>0$ and a constant $C>0$ such that the following holds true. 
Let $\underline{\rho}_1, \underline{\rho}_2, \underline{h}_1, \underline{h}_2, \delta$ be positive parameters satisfying 
the natural restrictions~\eqref{parameters}, $\underline{h}_1\delta, \underline{h}_2\delta \leq 1$, and the condition 
${\underline{h}_\mathrm{min} \leq \underline{h}_1, \underline{h}_2}$, and let $b\in W^{m+2,\infty}$ such that 
$\underline{h}_2^{-1}\|b\|_{W^{m+2,\infty}}\leq M$. 
Suppose that the full model for interfacial gravity waves~\eqref{full-model-evolution} possesses a solution 
$(\zeta^{\mbox{\tiny\rm IW}},\phi_1^{\mbox{\tiny\rm IW}},\phi_2^{\mbox{\tiny\rm IW}}) \in 
 C([0,T^{\mbox{\tiny\rm IW}}];H^{m+1}\times\mathring{H}^{m+1}\times\mathring{H}^{m+1})$ satisfying a uniform bound 
\[
\begin{cases}
 \displaystyle
 \|\zeta^{\mbox{\tiny\rm IW}}(t)\|_{H^{m+1}}^2 + \sum_{\ell=1,2}\underline{\rho}_\ell\underline{h}_\ell
  \|\nabla\phi_\ell^{\mbox{\tiny\rm IW}}(t)\|_{H^m}^2 \leq M, \\
 H_1^{\mbox{\tiny\rm IW}}(\bm{x},t) \geq c, \quad H_2^{\mbox{\tiny\rm IW}}(\bm{x},t) \geq c
  \quad\mbox{for}\quad \bm{x}\in\mathbf{R}^n, t\in[0,T^{\mbox{\tiny\rm IW}}],
\end{cases}
\]
where we denote $H_1^{\mbox{\tiny\rm IW}}:=1-\underline{h}_1^{-1}\zeta^{\mbox{\tiny\rm IW}}$ and 
$H_2^{\mbox{\tiny\rm IW}}:=1+\underline{h}_2^{-1}\zeta^{\mbox{\tiny\rm IW}}-\underline{h}_2^{-1}b$. 
Let $\zeta_{(0)}:=\zeta^{\mbox{\tiny\rm IW}}|_{t=0}$ and 
$\phi_{(0)}:=(\underline{\rho}_2\phi_2^{\mbox{\tiny\rm IW}}-\underline{\rho}_1\phi_1^{\mbox{\tiny\rm IW}})|_{t=0}$ 
be the initial data for the canonical variables, and let $(\bm{\phi}_{1(0)},\bm{\phi}_{2(0)})$ be the initial data 
to the Kakinuma model constructed from $(\zeta_{(0)},\phi_{(0)})$ by Proposition {\rm~\ref{preparation-ini}}. 
Assume moreover that the initial data $(\zeta_{(0)},\bm{\phi}_{1(0)},\bm{\phi}_{2(0)})$ satisfy the stability 
condition~\eqref{Stability}, 
let $(\zeta^{\mbox{\tiny\rm K}}, \bm{\phi}_1^{\mbox{\tiny\rm K}},\bm{\phi}_2^{\mbox{\tiny\rm K}})$ be the 
solution to the initial value problem for the Kakinuma model~\eqref{Kakinuma-dimensionless}--\eqref{Kaki:IC} 
on the time interval $[0,T]$ whose unique existence is guaranteed by Theorem {\rm~\ref{theorem-uniform}}, 
and put $\phi_\ell^{\mbox{\tiny\rm K}}=\bm{l}_\ell(H_\ell)\cdot\bm{\phi}_\ell^{\mbox{\tiny\rm K}}$ for $\ell=1,2$. 
Then, we have the error bound 
\begin{align*}
 \|\zeta^{\mbox{\tiny\rm K}}(t)-\zeta^{\mbox{\tiny\rm IW}}(t)\|_{H^{m-4(N+1)}}
 + \sum_{\ell=1,2} \sqrt{ \underline{\rho}_\ell\underline{h}_\ell }
  \|\nabla\phi_\ell^{\mbox{\tiny\rm K}}(t)-\nabla\phi_\ell^{\mbox{\tiny\rm IW}}(t)\|_{H^{m-(4N+5)}} \\
\leq C((\underline{h}_1\delta)^{4N+2}+(\underline{h}_2\delta)^{4N+2})
\end{align*}
for $0\leq t\leq \min\{T,T^{\mbox{\tiny\rm IW}}\}$. 
\end{theorem}

The next theorem is the final main result in this paper and states the consistency of the Hamiltonian 
$\mathscr{H}^{\mbox{\rm\tiny K}}(\zeta,\phi)$ of the Kakinuma model 
with respect to the Hamiltonian $\mathscr{H}^{\mbox{\rm\tiny IW}}(\zeta,\phi)$ of the full model for interfacial gravity waves. 
We recall that these Hamiltonians are defined in~\eqref{Hamiltonian-Kakinuma} and~\eqref{Hamiltonian-full-model}, respectively.

\begin{theorem}\label{theorem-Hamiltonian}
Let $c, M, \underline{h}_\mathrm{min}$ be positive constants and $m$ an integer such that $m>\frac{n}{2}+1$ and $m \geq 4(N+1)$. 
We assume {\rm (H1)} or {\rm (H2)}. 
There exists a positive constant $C$ such that for any positive parameters 
$\underline{\rho}_1, \underline{\rho}_2, \underline{h}_1, \underline{h}_2, \delta$ satisfying the natural 
restrictions~\eqref{parameters}, $\underline{h}_1\delta, \underline{h}_2\delta \leq 1$, and the condition 
${\underline{h}_\mathrm{min} \leq \underline{h}_1, \underline{h}_2}$, and for any 
$(\zeta,\phi)\in H^m\times\mathring{H}^{4(N+1)}$ and $b\in W^{m+1,\infty}$ satisfying 
\[
\begin{cases}
 \underline{h}_1^{-1}\|\zeta\|_{H^m} + \underline{h}_2^{-1}\|\zeta\|_{H^m}
  + \underline{h}_2^{-1}\|b\|_{W^{m+1,\infty}} \leq M, \\
 H_1(\bm{x})\geq c, \quad H_2(\bm{x})\geq c \quad\mbox{for}\quad \bm{x}\in\mathbf{R}^n,
\end{cases}
\]
with $H_1$ and $H_2$ defined by~\eqref{thicknesses}, we have 
\[
|\mathscr{H}^{\mbox{\rm\tiny K}}(\zeta,\phi)-\mathscr{H}^{\mbox{\rm\tiny IW}}(\zeta,\phi)|
\leq C\|\nabla\phi\|_{H^{4N+3}}\|\nabla\phi\|_{L^2}
 ((\underline{h}_1\delta)^{4N+2}+(\underline{h}_2\delta)^{4N+2}).
\]
\end{theorem}

\section{Consistency of the Kakinuma model; 
 proof of Theorems~\ref{theorem-consistency1} and~\ref{theorem-consistency2}}\label{S.consistency}
In this section we show that under the special choice of the indices $p_0,p_1,\ldots,p_{N^*}$ as 
\begin{enumerate}
\item[(H1)]
$N^*=N$ and $p_i=2i$ $(i=0,1,\ldots,N)$ in the case of the flat bottom $b(\bm{x})\equiv0$, 
\item[(H2)]
$N^*=2N$ and $p_i=i$ $(i=0,1,\ldots,2N)$ in the case with general bottom topographies,
\end{enumerate}
the Kakinuma model~\eqref{Kakinuma-dimensionless}is a higher order model to the full model for interfacial gravity waves~\eqref{full-model-evolution} 
in the limit ${\delta_1=\underline{h}_1\delta\to0}$, ${\delta_2=\underline{h}_2\delta\to 0}$, in the sense of consistency. 
Specifically, we prove Theorems~\ref{theorem-consistency1} and~\ref{theorem-consistency2}. 
Our proof relies essentially on results obtained in the framework of surface waves in~\cite{Iguchi2018-2}, which are recalled in Subsection~\ref{S.cons:one-layer}. 
The extension to the framework of interfacial waves and the completion of the proof are provided in Subsection~\ref{S.cons:two-layers}.

\subsection{Results in the framework of surface waves}\label{S.cons:one-layer}
In this subsection, we consider the case of surface waves where the water surface and the bottom of the 
water are represented as $z=\zeta(\bm{x})$ and $z=-1+b(\bm{x})$, respectively. 
Here, the time $t$ is fixed arbitrarily, so that we omit the dependence of $t$ in notations. 
Let $H(\bm{x})=1+\zeta(\bm{x})-b(\bm{x})$ be the water depth. 
For a nonnegative integer $N$, let $N^*$ and $p_0,p_1,\ldots,p_{N^*}$ be nonnegative integers 
satisfying the condition (H1) or (H2). 
Put 
\begin{equation}\label{def-l}
{\bm l}(H) := (1,H^{p_1},\dots,H^{p_{N^*}})^\mathrm{T}
\end{equation}
and define $L_{ij} = L_{ij}(H,b,\delta)$ $(i,j=0,1,\ldots,N^*)$ by 
\begin{align}\label{def-Lij}
L_{ij}\varphi_{j}
:=& - \nabla\cdot\biggl(
   \frac{1}{p_i+p_j+1}H^{p_i+p_j+1}\nabla\varphi_{j}
   - \frac{p_j}{p_i+p_j}H^{p_i+p_j}\varphi_{j}\nabla b\biggr) \\
 &- \frac{p_i}{p_i+p_j}H^{p_i+p_j}\nabla b\cdot\nabla\varphi_{j}
   + \frac{p_ip_j}{p_i+p_j-1}H^{p_i+p_j-1}(\delta^{-2}+ |\nabla b|^2)\varphi_{j}, \nonumber
\end{align}
where we use the notational convention $0/0 = 0$.
Introduce linear operators $\mathcal{L}_{i} = \mathcal{L}_{i}(H,b,\delta)$ $(i=0,1,\ldots,N^*)$ 
acting on ${\bm \varphi} = (\varphi_{0},\ldots,\varphi_{N^*})^\mathrm{T}$ by 
\begin{equation}\label{def-L}
\begin{cases}
 \displaystyle
  \mathcal{L}_{0} {\bm \varphi} := \sum_{j=0}^{N^*} L_{0j} \varphi_{j}, \\
 \displaystyle
  \mathcal{L}_{i} {\bm \varphi} := \sum_{j=0}^{N^*} ( L_{ij}\varphi_{j} - H^{p_i}L_{0j}\varphi_{j} )
   \quad\mbox{for}\quad i=1,2,\ldots,N^*.
\end{cases}
\end{equation}
The following Lemma has been proved in~\cite[Lemmas 3.2 and 3.4]{Iguchi2018-2}.

\begin{lemma}\label{L.L-invertible}
Let $c, M$ be positive constants and $m$ an integer such that $m>\frac{n}{2}+1$. 
There exists a positive constant $C$ such that if $\zeta\in H^m$, $b\in W^{m,\infty}$, and $H=1+\zeta-b$ satisfy 
\begin{equation}\label{eq-Hyp0}
\begin{cases}
 \|\zeta\|_{H^m}+\|b\|_{W^{m,\infty}}\leq M, \\
 H(\bm{x})\geq c \quad\mbox{for}\quad \bm{x}\in\mathbf{R}^n,
\end{cases}
\end{equation}
then for any $k=\pm 0,\dots, \pm(m-1)$, any $\delta\in(0,1]$, and any $\phi\in\mathring{H}^{k+1}$ 
there exists a unique solution $\bm{\phi}=(\phi_0,\phi_1,\ldots,\phi_{N^*})=(\phi_0,\bm{\phi}') 
\in \mathring{H}^{k+1}\times(H^{k+1})^{N^*}$ 
to the problem 
\begin{equation}\label{condition}
\begin{cases}
 \mathcal{L}_i(H,b,\delta){\bm \phi} = 0 \quad \mbox{for}\quad i=1,2,\ldots,N^*, \\
 {\bm l}(H)\cdot {\bm \phi} = \phi. 
\end{cases}
\end{equation}
Moreover, the solution satisfies 
$\|\nabla\bm{\phi}\|_{H^k} + \delta^{-1}\|\bm{\phi}'\|_{H^k} \leq C\|\nabla\phi\|_{H^k}$. 
\end{lemma}

As a corollary of this lemma, under the assumptions of Lemma~\ref{L.L-invertible} 
\[
\Lambda^{(N)}(\zeta,b,\delta) \colon \phi \mapsto \mathcal{L}_0(H,b,\delta)\bm{\phi}, 
\]
where $\bm{\phi}$ is the unique solution to~\eqref{condition}, is defined as a bounded linear operator 
from $\mathring{H}^{k+1}$ to $H^{k-1}$ for any $k=\pm 0,\dots, \pm(m-1)$. 
A key result is that the operator $\Lambda^{(N)}(\zeta,b,\delta)$ provide good approximations in the shallow water 
regime $\delta\ll1$ to the corresponding Dirichlet-to-Neumann map $\Lambda(\zeta,b,\delta)$, which is defined by 
\begin{equation}\label{def-DNmap}
\Lambda(\zeta,b,\delta)\phi := \big(\delta^{-2}\partial_z\Phi-\nabla\zeta\cdot\nabla\Phi\big)\big\vert_{z=\zeta},
\end{equation}
where $\Phi$ is the unique solution to the boundary value problem 
\begin{equation}\label{BVP}
\begin{cases}
\Delta\Phi + \delta^{-2}\partial_z^2\Phi = 0 & \mbox{in}\quad -1+b(\bm{x})<z<\zeta(\bm{x}), \\
\Phi=\phi& \mbox{on}\quad z=\zeta(\bm{x}), \\
\nabla b\cdot\nabla\Phi - \delta^{-2}\partial_z\Phi = 0 & \mbox{on}\quad z=-1+b(\bm{x}).
\end{cases}
\end{equation}
More precisely, we have the following Lemma.

\begin{lemma}\label{L.L-estimate-D2N}
Let $c, M$ be positive constants and $m, j$ integers such that $m>\frac{n}{2}+1$, $m\geq2(j+1)$, and $1\leq j\leq 2N+1$. 
We assume {\rm (H1)} or {\rm (H2)}. 
There exists a positive constant $C$ such that if $\zeta\in H^m$, $b\in W^{m+1,\infty}$, and $H=1+\zeta-b$ satisfy 
\begin{equation}\label{eq-Hyp}
\begin{cases}
 \|\zeta\|_{H^m}+\|b\|_{W^{m+1,\infty}}\leq M, \\
 H(\bm{x})\geq c \quad\mbox{for}\quad \bm{x}\in\mathbf{R}^n,
\end{cases}
\end{equation}
then for any $\phi\in\mathring{H}^{k+2(j+1)}$ with $0\leq k\leq m-2(j+1)$ and any $\delta\in(0,1]$ we have 
\[
\|\Lambda^{(N)}(\zeta,b,\delta)\phi-\Lambda(\zeta,b,\delta)\phi\|_{H^k}
 \leq C \delta^{2j} \|\nabla \phi\|_{H^{k+2j+1}}.
\]
\end{lemma}

\begin{proof}
We observe that the bound on $\mathfrak{r}_1:=\Lambda^{(N)}(\zeta,b,\delta)\phi-\Lambda(\zeta,b,\delta)\phi$ in the case 
$j=2N+1$ and $k=m-4(N+1)$ is given in~\cite[Theorem 2.2]{Iguchi2018-2} and proved in~\cite[Sections 8.1 and 8.2]{Iguchi2018-2}. 
The proof is also valid in the case $1\leq j\leq 2N+1$ and $0\leq k\leq m-2(j+1)$. 
\end{proof}

The above estimate allows us to obtain the desired consistency result on the equations describing the conservation of mass. 
We need a similar estimate for the contributions of Bernoulli's equation. 
To this end, we denote 
\begin{equation}\label{def-B}
B(\phi;\zeta,b,\delta)
 := \frac12|\nabla\phi|^2 - \frac12\delta^{2}
  \frac{(\Lambda(\zeta,b,\delta)\phi+\nabla\zeta\cdot\nabla\phi)^2}{1+\delta^2|\nabla\zeta|^2}
\end{equation}
and 
\begin{equation}\label{def-BN}
B^{(N)}(\phi;\zeta,b,\delta)
 := \frac12\bigl( |{\bm u}|^2 + \delta^{-2} w^2 \bigr) - w\Lambda^{(N)}(\zeta,b,\delta)\phi
\end{equation}
with 
\[
\begin{cases}
\displaystyle
\bm{u} :=  ({\bm l}(H) \otimes \nabla)^\mathrm{T}{\bm \phi} - ( {\bm l}^\prime(H) \cdot {\bm\phi})\nabla b, \\
\displaystyle
w :={\bm l}^\prime(H) \cdot {\bm\phi},
\end{cases}
\]
where ${\bm l}^\prime(H) := (0,p_1 H^{p_1-1},\ldots,p_{N^*}H^{p_{N^*}-1})^\mathrm{T}$ 
and ${\bm \phi} := (\phi_0,\phi_1,\ldots,\phi_{N^*})^\mathrm{T}$ is the solution to~\eqref{condition}, 
whose unique existence is guaranteed by Lemma~\ref{L.L-invertible}. 
Then, the following lemma shows that $B^{(N)}(\phi;\zeta,b,\delta)$ is a higher order approximation 
to $B(\phi;\zeta,b,\delta)$ in the shallow water regime $\delta\ll1$.

\begin{lemma}\label{L.L-estimate-Bernoulli}
Let $c, M$ be positive constants and $m$ an integer such that $m\geq4(N+1)$ and $m>\frac{n}{2}+1$. 
We assume {\rm (H1)} or {\rm (H2)}. 
There exists a positive constant $C$ such that if $\zeta\in H^m$, $b\in W^{m+1,\infty}$, and $H=1+\zeta-b$ satisfy~\eqref{eq-Hyp}, 
then for any $\phi\in\mathring{H}^m$ and any $\delta\in(0,1]$ we have 
\[
\|B^{(N)}(\phi;\zeta,b,\delta)-B(\phi;\zeta,b,\delta)\|_{H^{m-4(N+1)}}
 \leq C \delta^{4N+2} \|\nabla \phi\|_{H^{m-1}}^2.
\]
\end{lemma}

\begin{proof}
Notice first that differentiating $\phi=\bm{l}(H)\cdot\bm{\phi}$ we have 
$\nabla \phi={\bm u}+ w\nabla\zeta$, so that 
\begin{align*}
B^{(N)}(\phi;\zeta,b,\delta)
=& \frac12\bigl( |\nabla\phi|^2 + \delta^{-2} w^2(1+\delta^2|\nabla\zeta|^2) \bigr)
  - w\bigl(\nabla\zeta\cdot\nabla\phi + \Lambda^{(N)}(\zeta,b,\delta)\phi\bigr) \\
=& \frac12\bigl( |\nabla\phi|^2 + \delta^{-2} w^2(1+\delta^2|\nabla\zeta|^2) \bigr)
  - w\bigl(\Lambda(\zeta,b,\delta)\phi+\nabla\zeta\cdot\nabla\phi\bigr) \\
 & + w\bigl( \Lambda(\zeta,b,\delta)\phi-\Lambda^{(N)}(\zeta,b,\delta)\phi \bigr).
\end{align*}
If we introduce a residual $\mathfrak{r}$ by 
\[
\mathfrak{r} = (\delta^{-2}\partial_z\Phi^\mathrm{app} - \nabla\zeta\cdot\nabla \Phi^\mathrm{app})|_{z=\zeta}
 - (\delta^{-2}\partial_z\Phi - \nabla\zeta\cdot\nabla \Phi)|_{z=\zeta},
\]
where $\Phi$ is the solution to the boundary value problem~\eqref{BVP} and $\Phi^\mathrm{app}$ 
is an approximate velocity potential defined by 
\begin{equation*}\label{AVP}
\Phi^\mathrm{app}(\bm{x},z) = \sum_{i=0}^{N^*}(z+1-b(\bm{x}))^{p_i}\phi_i(\bm{x}),
\end{equation*}	
then we have 
$\mathfrak{r}=\delta^{-2}w-\nabla\zeta\cdot \bm{u}-\Lambda(\zeta,b,\delta)\phi
=\delta^{-2} w(1+\delta^2|\nabla\zeta|^2)-\nabla\zeta\cdot\nabla\phi-\Lambda(\zeta,b,\delta)\phi$. 
Therefore, we obtain 
\[
B^{(N)}(\phi;\zeta,b,\delta) - B(\phi;\zeta,b,\delta)
= \frac12\delta^2\frac{\mathfrak{r}^2}{1+\delta^2|\nabla\zeta|^2}
 + w\bigl( \Lambda(\zeta,b,\delta)\phi-\Lambda^{(N)}(\zeta,b,\delta)\phi \bigr).
\]
The desired estimate for the second term readily follows from Lemma~\ref{L.L-invertible} and Lemma~\ref{L.L-estimate-D2N}. 
As for the first term, in view of $m>\frac{n}{2}$ we can use a calculus inequality 
$\|\mathfrak{r}^2\|_{H^k} \lesssim \|\mathfrak{r}\|_{H^{(m+k)/2}}^2$ for $k\in\{0,1,\dots,m\}$. 
Particularly, we have $\|\mathfrak{r}^2\|_{H^{m-4(N+1)}} \lesssim \|\mathfrak{r}\|_{H^{m-2(N+1)}}^2$. 
The last term can be evaluated by estimates in~\cite[Sections 8.1 and 8.2]{Iguchi2018-2}. 
\end{proof}

\subsection{Results in the framework of interfacial waves}\label{S.cons:two-layers}
In this section, we prove Theorems~\ref{theorem-consistency1} and~\ref{theorem-consistency2}. 
To this end, we first rewrite the Kakinuma model~\eqref{Kakinuma-dimensionless} using a formulation 
which allows a direct comparison with the full model for interfacial gravity waves~\eqref{full-model-evolution}, 
thanks to the following Lemma.

\begin{lemma}\label{L.Lambda}
Let $c, M$ be positive constants and $m$ an integer such that $m>\frac{n}{2}+1$. 
There exists a positive constant $C$ such that for any positive parameters 
$\underline{h}_1, \underline{h}_2, \delta$ satisfying $\underline{h}_1\delta, \underline{h}_2\delta \leq 1$, 
if $\zeta\in H^m$, $b\in W^{m,\infty}$, $H_1 = 1 - \underline{h}_1^{-1}\zeta$, and 
$H_2 = 1 + \underline{h}_2^{-1}\zeta - \underline{h}_2^{-1}b$ satisfy 
\begin{equation}\label{eq-Hyp20}
\begin{cases}
 \underline{h}_1^{-1}\|\zeta\|_{H^m} + \underline{h}_2^{-1}\|\zeta\|_{H^m}
  + \underline{h}_2^{-1}\|b\|_{W^{m,\infty}} \leq M, \\
 H_1(\bm{x})\geq c, \quad H_2(\bm{x})\geq c \quad\mbox{for}\quad \bm{x}\in\mathbf{R}^n,
\end{cases}
\end{equation}
then for any $k=0,\pm1,\ldots,\pm(m-1)$ and any $\phi_1,\phi_2 \in \mathring{H}^{k+1}$ there exists a 
unique solution $\bm{\phi}_1=(\phi_{1,0},\bm{\phi}_1')\in \mathring{H}^{k+1}\times(H^{k+1})^N$, 
$\bm{\phi}_2=(\phi_{2,0},\bm{\phi}_2')\in \mathring{H}^{k+1}\times(H^{k+1})^{N^*}$ to the problem 
\begin{equation}\label{Conditions2}
\begin{cases}
 \bm{l}_1(H_1)\cdot\bm{\phi}_1=\phi_1, \quad \mathcal{L}_{1,i}(H_1,\delta,\underline{h}_1)\bm{\phi}_1=0
  \quad\mbox{for}\quad i=1,2,\ldots,N, \\
 \bm{l}_2(H_2)\cdot\bm{\phi}_2=\phi_2, \quad \mathcal{L}_{2,i}(H_2,b,\delta,\underline{h}_2)\bm{\phi}_2=0
  \quad\mbox{for}\quad i=1,2,\ldots,N^*.
\end{cases}
\end{equation}
Moreover, the solution satisfies 
$\|\nabla\bm{\phi}_\ell\|_{H^k} + (\underline{h}_\ell\delta)^{-1}\|\bm{\phi}_\ell'\|_{H^k} \leq C\|\nabla\phi_\ell\|_{H^k}$ for $\ell=1,2$. 
\end{lemma}

\begin{proof}
Notice that we have identities 
\[
L_{1,ij}(H_1,\delta,\underline{h}_1) = L_{ij}(H_1,0,\underline{h}_1\delta), \quad 
L_{2,ij}(H_2,b,\delta,\underline{h}_2) = L_{ij}(H_2,\underline{h}_2^{-1}b,\underline{h}_2\delta)
\]
with suitable choices of indices $\{p_i\}$. 
Hence, Lemma~\ref{L.L-invertible} gives the desired result. 
\end{proof}

As a corollary of this lemma, under the assumptions of Lemma~\ref{L.Lambda} 
\begin{align*}
& \Lambda_1^{(N)}(\zeta,\delta,\underline{h}_1) \colon \phi_1 \mapsto \mathcal{L}_{1,0}(H_1,\underline{h}_1,\delta)\bm{\phi}_1, \\
& \Lambda_2^{(N)}(\zeta,b,\delta,\underline{h}_2) \colon \phi_2 \mapsto \mathcal{L}_{2,0}(H_2,b,\underline{h}_2,\delta)\bm{\phi}_2, 
\end{align*}
where $(\bm{\phi}_1,\bm{\phi}_2)$ is the unique solution to~\eqref{Conditions2}, are defined as bounded linear operators 
from $\mathring{H}^{k+1}$ to $H^{k-1}$ for any $k=\pm 0,\dots, \pm(m-1)$. 
Using these definitions and noting the relations~\eqref{expression-L_k} and 
$\bm{l}_\ell(H_\ell)\cdot\partial_t\bm{\phi}_\ell
 = \partial_t(\bm{l}_\ell(H_\ell)\cdot\bm{\phi}_\ell)-w_\ell\underline{h}_\ell^{-1}\partial_t\zeta$, 
we can transform the Kakinuma model~\eqref{Kakinuma-dimensionless} equivalently as 
\begin{equation}\label{Kakinuma-evolution}
\begin{cases}
 \partial_t\zeta + \underline{h}_1\Lambda_1^{(N)}(\zeta,\delta,\underline{h}_1)\phi_1 = 0, \\
 \partial_t\zeta - \underline{h}_2\Lambda_1^{(N)}(\zeta,b,\delta,\underline{h}_2)\phi_2 = 0, \\
 \underline{\rho}_1\bigl\{ \partial_t\phi_1 
   + \frac12\bigl( |{\bm u}_1|^2 + (\underline{h}_1\delta)^{-2} w_1^2 \bigr)
   + w_1\Lambda_1^{(N)}(\zeta,\delta,\underline{h}_1)\phi_1 \bigr\} \\
 \quad
 - \underline{\rho}_2\bigl\{  \partial_t\phi_2 
  + \frac12\bigl( |{\bm u}_2|^2 +  (\underline{h}_2\delta)^{-2} w_2^2 \bigr)
  - w_2\Lambda_2^{(N)}(\zeta,b,\delta,\underline{h}_2)\phi_2 \bigr\} 
 - \zeta = 0, 
\end{cases}
\end{equation}
where we recall that ${\bm u}_1$, ${\bm u}_2$, $w_1$, and $w_2$ are uniquely determined from $\phi_1$ 
and $\phi_2$ by~\eqref{def-uw}, wherein ${\bm \phi}_1$ and ${\bm \phi}_2$ are defined as the solutions to~\eqref{Conditions2}.

We further introduce notations, which are contributions of Bernoulli's equation and interfacial versions of 
$B$ and $B^{(N)}$ defined by~\eqref{def-B} and~\eqref{def-BN}. 
We denote 
\[
\begin{cases}
 \displaystyle
 B_1(\phi_1;\zeta,\delta,\underline{h}_1)
 := \frac12|\nabla\phi_1|^2 - \frac12\delta^{2}
  \frac{(\Lambda_1(\zeta,\delta,\underline{h}_1)\phi_1-\nabla\zeta\cdot\nabla\phi_1)^2}{1+\delta^2|\nabla\zeta|^2}, \\
 \displaystyle
 B_2(\phi_2;\zeta,b,\delta,\underline{h}_2)
 := \frac12|\nabla\phi_2|^2 - \frac12\delta^{2}
  \frac{(\Lambda_2(\zeta,b,\delta,\underline{h}_2)\phi_2+\nabla\zeta\cdot\nabla\phi_2)^2}{1+\delta^2|\nabla\zeta|^2},
\end{cases}
\]
and 
\[
\begin{cases}
 \displaystyle
 B_1^{(N)}(\phi_1;\zeta,\delta,\underline{h}_1)
 := \frac12\bigl( |\bm{u}_1|^2 + (\underline{h}_1\delta)^{-2} w_1^2 \bigr)
  + w_1\Lambda_1^{(N)}(\zeta,\delta,\underline{h}_1)\phi_1, \\
 \displaystyle
 B_2^{(N)}(\phi_2;\zeta,b,\delta,\underline{h}_2)
 := \frac12\bigl( |\bm{u}_2|^2 + (\underline{h}_2\delta)^{-2} w_2^2 \bigr)
  - w_2\Lambda_2^{(N)}(\zeta,b,\delta,\underline{h}_2)\phi_2. 
\end{cases}
\]
Then, the full model for interfacial gravity waves~\eqref{full-model-evolution} and the Kakinuma model
\eqref{Kakinuma-evolution} can be written simply as 
\[
\begin{cases}
 \partial_t\zeta + \Lambda_1(\zeta,\delta,\underline{h}_1)\phi_1 =0, \\
 \partial_t\zeta - \Lambda_2(\zeta,b,\delta,\underline{h}_2)\phi_2 =0, \\
 \displaystyle
 \underline{\rho}_1\bigl( \partial_t\phi_1 + B_1(\phi_1;\zeta,\delta,\underline{h}_1) \bigr) 
 - \underline{\rho}_2\bigl( \partial_t\phi_2 + B_2(\phi_2;\zeta,b,\delta,\underline{h}_2) \bigr)
 - \zeta = 0,
\end{cases}
\]
and 
\[
\begin{cases}
 \partial_t\zeta + \underline{h}_1\Lambda_1^{(N)}(\zeta,\delta,\underline{h}_1)\phi_1 = 0, \\
 \partial_t\zeta - \underline{h}_2\Lambda_1^{(N)}(\zeta,b,\delta,\underline{h}_2)\phi_2 = 0, \\
 \underline{\rho}_1\bigl( \partial_t\phi_1 + B_1^{(N)}(\phi_1;\zeta,\delta,\underline{h}_1) \bigr) 
 - \underline{\rho}_2\bigl( \partial_t\phi_2 + B_2^{(N)}(\phi_2;\zeta,b,\delta,\underline{h}_2) \bigr)
 - \zeta = 0,
\end{cases}
\]
respectively. 
The following lemmas show that $\underline{h}_1\Lambda_1^{(N)}$, $\underline{h}_2\Lambda_2^{(N)}$, $B_1^{(N)}$, 
and $B_2^{(N)}$ are higher order approximations in the shallow water regime $\delta_1=\underline{h}_1\delta\ll1$ 
and $\delta_2=\underline{h}_2\delta\ll1$ to $\Lambda_1$, $\Lambda_2$, $B_1$, and $B_2$, respectively.

\begin{lemma}\label{L.res1}
Let $c, M$ be positive constants and $m, j$ integers such that $m>\frac{n}{2}+1$, $m\geq 2(j+1)$, and $1\leq j\leq 2N+1$. 
We assume {\rm (H1)} or {\rm (H2)}. 
There exists a positive constant $C$ such that for any positive parameters 
$\underline{h}_1, \underline{h}_2, \delta$ satisfying $\underline{h}_1\delta, \underline{h}_2\delta \leq 1$, 
if $\zeta\in H^m$, $b\in W^{m+1,\infty}$, $H_1 = 1 - \underline{h}_1^{-1}\zeta$, 
and $H_2 = 1 + \underline{h}_2^{-1}\zeta - \underline{h}_2^{-1}b$ satisfy 
\begin{equation}\label{eq-Hyp2}
\begin{cases}
 \underline{h}_1^{-1}\|\zeta\|_{H^m} + \underline{h}_2^{-1}\|\zeta\|_{H^m}
  + \underline{h}_2^{-1}\|b\|_{W^{m+1,\infty}} \leq M, \\
 H_1(\bm{x})\geq c, \quad H_2(\bm{x})\geq c \quad\mbox{for}\quad \bm{x}\in\mathbf{R}^n,
\end{cases}
\end{equation}
then for any $\phi_1,\phi_2 \in \mathring{H}^{k+2(j+1)}$ with $0\leq k\leq m-2(j+1)$ we have 
\[
\begin{cases}
 \|\underline{h}_1\Lambda_1^{(N)}(\zeta,\delta,\underline{h}_1)\phi_1
  - \Lambda_1(\zeta,\delta,\underline{h}_1)\phi_1\|_{H^k}
  \leq C \underline{h}_1(\underline{h}_1\delta)^{2j} \|\nabla \phi_1\|_{H^{k+2j+1}}, \\
 \|\underline{h}_2\Lambda_2^{(N)}(\zeta,b,\delta,\underline{h}_2)\phi_2
  - \Lambda_2(\zeta,b,\delta,\underline{h}_2)\phi_2\|_{H^k}
  \leq C \underline{h}_2(\underline{h}_2\delta)^{2j} \|\nabla \phi_2\|_{H^{k+2j+1}}.
\end{cases}
\]
\end{lemma}

\begin{proof}
By simple scaling arguments, we have 
\begin{equation}\label{relations-DN-B}
\begin{cases}
 \Lambda_1(\zeta,\delta,\underline{h}_1)
  = \underline{h}_1\Lambda(-\underline{h}_1^{-1}\zeta,0,\underline{h}_1\delta), \\
 \Lambda_2(\zeta,b,\delta,\underline{h}_2)
  = \underline{h}_2\Lambda(\underline{h}_2^{-1}\zeta,\underline{h}_2^{-1}b,\underline{h}_2\delta), \\
 \Lambda_1^{(N)}(\zeta,\delta,\underline{h}_1)
  = \Lambda^{(N)}(-\underline{h}_1^{-1}\zeta,0,\underline{h}_1\delta), \\
 \Lambda_2^{(N)}(\zeta,b,\delta,\underline{h}_2)
  = \Lambda^{(N)}(\underline{h}_2^{-1}\zeta,\underline{h}_2^{-1}b,\underline{h}_2\delta).
\end{cases}
\end{equation}
Therefore, the results follow from Lemma~\ref{L.L-estimate-D2N}. 
\end{proof}

\begin{lemma}\label{L.res2}
Let $c, M$ be positive constants and $m$ an integer such that $m\geq 4(N+1)$ and $m>\frac{n}{2}+1$. 
We assume {\rm (H1)} or {\rm (H2)}. 
There exists a positive constant $C$ such that for any positive parameters 
$\underline{h}_1, \underline{h}_2, \delta$ satisfying $\underline{h}_1\delta, \underline{h}_2\delta \leq 1$, 
if $\zeta\in H^m$, $b\in W^{m+1,\infty}$, $H_1 = 1 - \underline{h}_1^{-1}\zeta$, 
and $H_2 = 1 + \underline{h}_2^{-1}\zeta - \underline{h}_2^{-1}b$ satisfy~\eqref{eq-Hyp2}, 
then for any $\phi_1,\phi_2 \in \mathring{H}^m$ we have 
\[
\begin{cases}
 \|B_1^{(N)}(\phi_1;\zeta,\delta,\underline{h}_1) - B_1(\phi_1;\zeta,\delta,\underline{h}_1)\|_{H^{m-4(N+1)}}
 \leq C \|\nabla \phi_1\|_{H^{m-1}}^2(\underline{h}_1\delta)^{4N+2}, \\
 \|B_2^{(N)}(\phi_2;\zeta,b,\delta,\underline{h}_2) - B_2(\phi_2;\zeta,b,\delta,\underline{h}_2)\|_{H^{m-4(N+1)}}
 \leq C \|\nabla \phi_2\|_{H^{m-1}}^2(\underline{h}_2\delta)^{4N+2}.
\end{cases}
\]
\end{lemma}

\begin{proof}
By simple scaling arguments, we have 
\[
\begin{cases}
 B_1(\phi_1;\zeta,\delta,\underline{h}_1)
  = B(\phi_1;-\underline{h}_1^{-1}\zeta,0,\underline{h}_1\delta), \\
 B_2(\phi_2;\zeta,b,\delta,\underline{h}_1)
  = B(\phi_2;\underline{h}_2^{-1}\zeta,\underline{h}_2^{-1}b,\underline{h}_2\delta), \\
 B_1^{(N)}(\phi_1;\zeta,\delta,\underline{h}_1)
  = B^{(N)}(\phi_1;-\underline{h}_1^{-1}\zeta,0,\underline{h}_1\delta), \\
 B_2^{(N)}(\phi_2;\zeta,b,\delta,\underline{h}_1)
  = B^{(N)}(\phi_2;\underline{h}_2^{-1}\zeta,\underline{h}_2^{-1}b,\underline{h}_2\delta).
\end{cases}
\]
Therefore, the results follow from Lemma~\ref{L.L-estimate-Bernoulli}. 
\end{proof}

We can now prove Theorems~\ref{theorem-consistency1} and~\ref{theorem-consistency2}. 
In view of~\eqref{expression-L_k} the errors $(\mathfrak{r}_1,\mathfrak{r}_2,\mathfrak{r}_0)$ and 
$(\tilde{\bm{\mathfrak{r}}}_1,\tilde{\bm{\mathfrak{r}}}_2, \tilde{\mathfrak{r}}_0)$ can be written 
explicitly as 
\[
\begin{cases}
 \mathfrak{r}_1 = \Lambda_1(\zeta,\delta,\underline{h}_1)\phi_1
  - \underline{h}_1\Lambda_1^{(N)}(\zeta,\delta,\underline{h}_1)\phi_1, \\
 \mathfrak{r}_2 = \underline{h}_2\Lambda_2^{(N)}(\zeta,b,\delta,\underline{h}_2)\phi_2
  - \Lambda_2(\zeta,b,\delta,\underline{h}_2)\phi_2, \\
 \mathfrak{r}_0 = \frac12\underline{\rho}_1\bigl( B_1(\phi_1;\zeta,\delta,\underline{h}_1)
  - B_1^{(N)}(\phi_1;\zeta,\delta,\underline{h}_1) \bigr) \\
 \qquad
  - \frac12\underline{\rho}_2\bigl( B_2(\phi_2;\zeta,b,\delta,\underline{h}_2)
  - B_2^{(N)}(\phi_2;\zeta,b,\delta,\underline{h}_2) \bigr), \\
 \tilde{\bm{\mathfrak{r}}}_1 = -\underline{h}_1^{-1}\bm{l}_1(H_1)\mathfrak{r}_1, \quad
  \tilde{\bm{\mathfrak{r}}}_2 = -\underline{h}_2^{-1}\bm{l}_2(H_2)\mathfrak{r}_2, \quad
  \tilde{\mathfrak{r}}_0 = -\mathfrak{r}_0.
\end{cases}
\]
Therefore, the theorems are simple corollaries of the above Lemmas~\ref{L.res1} and~\ref{L.res2}.

\section{Elliptic estimates and time derivatives}\label{S.elliptic}
In this section we derive useful uniform {\em a priori} bounds on regular solutions 
to the Kakinuma model~\eqref{Kakinuma-dimensionless}. 
Firstly, due to the fact that the hypersurface $t=0$ in the space-time $\mathbf{R}^n\times\mathbf{R}$ is 
characteristic for the Kakinuma model, we need the following key elliptic estimate in order to be able to 
estimate time derivatives of the solution. 
Let us recall that the operators  $\mathcal{L}_{1,i}$ for $i=0,1,\ldots,N$ and $\mathcal{L}_{2,i}$ for $i=0,1,\ldots,N^*$ 
are defined by~\eqref{def-calL}, and the vectors $\bm{l}_1(H_1)$ and $\bm{l}_2(H_2)$ are defined by ~\eqref{def-l1l2}. 
We recall the convention that for a vector $\mbox{\boldmath$\phi$}=(\phi_0,\phi_1,\ldots,\phi_N)^\mathrm{T}$ we denote the last $N$ 
components by $\mbox{\boldmath$\phi$}'=(\phi_1,\ldots,\phi_N)^\mathrm{T}$.

\begin{lemma}\label{L.elliptic}
Let $c, M$ be positive constants and $m$ an integer such that $m>\frac{n}{2}+1$. 
There exists a positive constant $C$ such that for any positive parameters 
$\underline{\rho}_1, \underline{\rho}_2, \underline{h}_1, \underline{h}_2, \delta$ satisfying 
$\underline{h}_1\delta, \underline{h}_2\delta \leq 1$, 
if $\zeta\in H^m$, $b\in W^{m,\infty}$, $H_1 = 1 - \underline{h}_1^{-1}\zeta$, 
and $H_2 = 1 + \underline{h}_2^{-1}\zeta - \underline{h}_2^{-1}b$ satisfy~\eqref{eq-Hyp20}, 
then for any $\bm{f}_1' = (f_{1,1},\ldots,f_{1,N})^\mathrm{T} \in (H^{k})^N$, 
$\bm{f}_2' = (f_{2,1},\ldots,$ $f_{2,N^*})^\mathrm{T} \in (H^{k})^{N^*}$, $\bm{f}_3 \in (H^k)^n$, 
and $f_4\in \mathring H^{k+1} $ with  $k\in\{0,1,\ldots,m-1\}$, 
there exists a solution $(\bm{\varphi}_1,\bm{\varphi}_2)$ to 
\begin{equation}\label{equationvarphi}
\begin{cases}
 \mathcal{L}_{1,i}(H_{1},\delta,\underline{h}_1) {\bm \varphi}_{1} =  f_{1,i} \quad\mbox{for}\quad i=1,2,\ldots,N, \\
 \mathcal{L}_{2,i}(H_{2},b,\delta,\underline{h}_2) {\bm \varphi}_{2} = f_{2,i} \quad\mbox{for}\quad i=1,2,\ldots,N^*, \\
 \underline{h}_1\mathcal{L}_{1,0}(H_{1}\delta,\underline{h}_1) {\bm \varphi}_{1}
  + \underline{h}_2\mathcal{L}_{2,0}(H_{2},b,\delta,\underline{h}_2) {\bm \varphi}_{2} = \nabla\cdot \bm{f}_3,\\
 - \underline{\rho}_1{\bm l}_1(H_1) \cdot {\bm \varphi}_1
    + \underline{\rho}_2{\bm l}_2(H_2) \cdot {\bm \varphi}_2 = f_4,
\end{cases}
\end{equation}
satisfying 
\begin{align*}
&\sum_{\ell=1,2}\underline{\rho}_\ell \underline{h}_\ell \bigl(
 \|\nabla\bm{\varphi}_\ell\|_{H^k}^2 + (\underline{h}_\ell\delta)^{-2}\|\bm{\varphi}_\ell'\|_{H^k}^2 \bigr) \\
&\leq C\biggl( \sum_{\ell=1,2}\underline{\rho}_\ell \underline{h}_\ell
 \min\bigl\{ \|\bm{f}_\ell'\|_{H^{k-1}}^2, (\underline{h}_\ell\delta)^2 \|\bm{f}_\ell'\|_{H^k}^2 \bigr\} \\
&\qquad
 + \min\biggl\{\frac{\underline{\rho}_1}{\underline{h}_1},\frac{\underline{\rho}_2}{\underline{h}_2} \biggr\}
  \|\bm{f}_3\|_{H^k}^2
 + \min\biggl\{\frac{\underline{h}_1}{\underline{\rho}_1},\frac{\underline{h}_2}{\underline{\rho}_2} \biggr\}
  \|\nabla f_4\|_{H^k}^2 \biggr).
\end{align*}
Moreover, the solution is unique up to an additive constant of the form 
$(\mathcal{C}\underline{\rho}_2,\mathcal{C}\underline{\rho}_1)$ to $(\varphi_{1,0},\varphi_{2,0})$. 
\end{lemma}

\begin{proof}
The existence and uniqueness up to an additive constant of the solution has been given in the companion 
paper~\cite[Lemma~6.4]{DucheneIguchi2020}. 
We focus here on the derivation of uniform estimates. 
By direct rescaling within the proof of~\cite[Lemma~6.1]{DucheneIguchi2020}, we infer that 
\[
(L_\ell\bm{\varphi}_\ell,\bm{\varphi}_\ell)_{L^2}
\simeq \|\nabla\bm{\varphi}_\ell\|_{L^2}^2 + (\underline{h}_\ell\delta)^{-2}\|\bm{\varphi}_\ell'\|_{L^2}^2
\]
for $\ell=1,2$. 
We note the identities 
\[
\begin{cases}
 L_1\bm{\varphi}_1 = \bm{l}_1\mathcal{L}_{1,0}\bm{\varphi}_1
  + (0,\mathcal{L}_{1,1}\bm{\varphi}_1,\ldots,\mathcal{L}_{1,N}\bm{\varphi}_1)^\mathrm{T}, \\
 L_2\bm{\varphi}_2 = \bm{l}_2\mathcal{L}_{2,0}\bm{\varphi}_2
  + (0,\mathcal{L}_{2,1}\bm{\varphi}_2,\ldots,\mathcal{L}_{2,N^*}\bm{\varphi}_2)^\mathrm{T},
\end{cases}
\]
so that for the solution $(\bm{\varphi}_1,\bm{\varphi}_2)$ to~\eqref{equationvarphi} we have 
\begin{align}\label{EE-identity1}
\sum_{\ell=1,2}\underline{\rho}_\ell\underline{h}_\ell (L_\ell\bm{\varphi}_\ell,\bm{\varphi}_\ell)_{L^2}
&= \sum_{\ell=1,2}\underline{\rho}_\ell\underline{h}_\ell
 (\mathcal{L}_{\ell,0}\bm{\varphi}_\ell,\bm{l}_\ell\cdot\bm{\varphi}_\ell)_{L^2}
 + \sum_{\ell=1,2}\underline{\rho}_\ell\underline{h}_\ell (\bm{f}_\ell',\bm{\varphi}_\ell')_{L^2} \\
&=: I_1+I_2. \nonumber
\end{align}
Therefore, it is sufficient to evaluate $I_1$ and $I_2$. 
As for the term $I_2$ we have 
\begin{align*}
|(\bm{f}_\ell',\bm{\varphi}_\ell')_{L^2}|
&\leq \min\{ \|\bm{f}_\ell'\|_{H^{-1}}\|\bm{\varphi}_\ell'\|_{H^1}, 
\|\bm{f}_\ell'\|_{L^2}\|\bm{\varphi}_\ell'\|_{L^2} \} \\
&\leq \min\{ \|\bm{f}_\ell'\|_{H^{-1}}, (\underline{h}_\ell\delta)\|\bm{f}_\ell'\|_{L^2} \}
( \|\nabla\bm{\varphi}_\ell\|_{L^2} + (\underline{h}_\ell\delta)^{-1}\|\bm{\varphi}_\ell'\|_{L^2} ).
\end{align*}
As for the term $I_1$, we note the trivial identities 
\begin{align*}
&\sum_{\ell=1,2}\underline{\rho}_\ell\underline{h}_\ell
 (\mathcal{L}_{\ell,0}\bm{\varphi}_\ell,\bm{l}_\ell\cdot\bm{\varphi}_\ell)_{L^2} \\
&= 
\begin{cases}
 (\underline{h}_1\mathcal{L}_{1,0}\bm{\varphi}_1+\underline{h}_2\mathcal{L}_{2,0}\bm{\varphi}_2,
  \underline{\rho}_1\bm{l}_1\cdot\bm{\varphi}_1)_{L^2}
 + (\underline{h}_2\mathcal{L}_{2,0}\bm{\varphi}_2,
  \underline{\rho}_2\bm{l}_2\cdot\bm{\varphi}_2-\underline{\rho}_1\bm{l}_1\cdot\bm{\varphi}_1)_{L^2}, \\
 (\underline{h}_1\mathcal{L}_{1,0}\bm{\varphi}_1+\underline{h}_2\mathcal{L}_{2,0}\bm{\varphi}_2,
  \underline{\rho}_2\bm{l}_2\cdot\bm{\varphi}_2)_{L^2}
 + (\underline{h}_1\mathcal{L}_{1,0}\bm{\varphi}_1,
  \underline{\rho}_1\bm{l}_1\cdot\bm{\varphi}_1-\underline{\rho}_2\bm{l}_2\cdot\bm{\varphi}_2)_{L^2}.
\end{cases}
\end{align*}
Therefore, the term $I_1$ in~\eqref{EE-identity1} can be expressed in two ways as 
\[
I_1 = 
\begin{cases}
 \underline{\rho}_1(\nabla\cdot\bm{f}_3,\bm{l}_1\cdot\bm{\varphi}_1)_{L^2}
  + \underline{h}_2(\mathcal{L}_{2,0}\bm{\varphi}_2,f_4)_{L^2}, \\
 \underline{\rho}_2(\nabla\cdot\bm{f}_3,\bm{l}_2\cdot\bm{\varphi}_2)_{L^2}
  - \underline{h}_1(\mathcal{L}_{1,0}\bm{\varphi}_1,f_4)_{L^2}.
\end{cases}
\]
By the linearity of~\eqref{equationvarphi} it is sufficient to evaluate it in the case 
$f_4=0$ and in the case $\bm{f}_3=\bm{0}$, separately. 
In the case $f_4=0$, we evaluate it as 
\begin{align*}
|I_1|
&\leq \min\{ \underline{\rho}_1\|\bm{f}_3\|_{L^2}\|\nabla(\bm{l}_1\cdot\bm{\varphi}_1)\|_{L^2},
  \underline{\rho}_2\|\bm{f}_3\|_{L^2}\|\nabla(\bm{l}_2\cdot\bm{\varphi}_2)\|_{L^2} \} \\
&= \min\biggl\{ \sqrt{\frac{\underline{\rho}_1}{\underline{h}_1}}\|\bm{f}_3\|_{L^2}
  \sqrt{\underline{\rho}_1\underline{h}_1}\|\nabla(\bm{l}_1\cdot\bm{\varphi}_1)\|_{L^2}, 
 \sqrt{\frac{\underline{\rho}_2}{\underline{h}_2}}\|\bm{f}_3\|_{L^2}
  \sqrt{\underline{\rho}_2\underline{h}_2}\|\nabla(\bm{l}_2\cdot\bm{\varphi}_2)\|_{L^2} \biggr\} \\
&\lesssim \min\biggl\{ \sqrt{\frac{\underline{\rho}_1}{\underline{h}_1}},
  \sqrt{\frac{\underline{\rho}_2}{\underline{h}_2}} \biggr\} \|\bm{f}_3\|_{L^2}
 \sum_{\ell=1,2}\sqrt{\underline{\rho}_\ell\underline{h}_\ell}
 ( \|\nabla\bm{\varphi}_\ell\|_{L^2} + \|\bm{\varphi}_\ell'\|_{L^2} ).
\end{align*}
In the case $\bm{f}_3=\bm{0}$ we evaluate it as 
\begin{align*}
|I_1|
&\lesssim \min\{ \underline{h}_1\|\nabla\bm{\varphi}_1\|_{L^2}\|\nabla f_4\|_{L^2}, 
 \underline{h}_2( \|\nabla\bm{\varphi}_2\|_{L^2}+\|\bm{\varphi}_2'\|_{L^2})\|\nabla f_4\|_{L^2} \} \\
&= \min\biggl\{ \sqrt{\frac{\underline{h}_1}{\underline{\rho}_1}}\|\nabla f_4\|_{L^2}
  \sqrt{\underline{\rho}_1\underline{h}_1}\|\nabla\bm{\varphi}_1\|_{L^2}, 
 \sqrt{\frac{\underline{h}_2}{\underline{\rho}_2}}\|\nabla f_4\|_{L^2}
  \sqrt{\underline{\rho}_2\underline{h}_2} (\|\nabla\bm{\varphi}_1\|_{L^2}+\|\bm{\varphi}_2'\|_{L^2}) \biggr\} \\
&\leq \min\biggl\{ \sqrt{\frac{\underline{h}_1}{\underline{\rho}_1}},
  \sqrt{\frac{\underline{h}_2}{\underline{\rho}_2}} \biggr\} \|\nabla f_4\|_{L^2}
 \sum_{\ell=1,2}\sqrt{\underline{\rho}_\ell\underline{h}_\ell}
 ( \|\nabla\bm{\varphi}_\ell\|_{L^2} + \|\bm{\varphi}_\ell'\|_{L^2} ).
\end{align*}
From the above estimates we deduce immediately the desired inequality for $k=0$.

In order to obtain the desired inequality on derivatives, we let $k\in\{1,2,\ldots,m-1\}$ and $\beta$ be a 
multi-index such that $1\leq|\beta|\leq k$. 
Applying the differential operator $\partial^\beta$ to~\eqref{equationvarphi}, we have 
\[
\begin{cases}
 \mathcal{L}_{1,i} \partial^\beta{\bm \varphi}_{1}
  =  \partial^\beta f_{1,i} + f_{1,i,\beta} \quad\mbox{for}\quad i=1,2,\ldots,N, \\
 \mathcal{L}_{2,i} \partial^\beta{\bm \varphi}_{2}
  = \partial^\beta f_{2,i} + f_{2,i,\beta} \quad\mbox{for}\quad i=1,2,\ldots,N^*, \\
 \underline{h}_1\mathcal{L}_{1,0} \partial^\beta{\bm \varphi}_{1}
   + \underline{h}_2\mathcal{L}_{2,0} \partial^\beta{\bm \varphi}_{2}
  = \nabla\cdot (\partial^\beta \bm{f}_3 + \underline{h}_1\bm{f}_{3,1,\beta} + \underline{h}_2\bm{f}_{3,2,\beta}), \\
 - \underline{\rho}_1{\bm l}_1 \cdot \partial^\beta{\bm \varphi}_1
    + \underline{\rho}_2{\bm l}_2 \cdot \partial^\beta{\bm \varphi}_2
   = \partial^\beta f_4 + \underline{\rho}_1f_{4,1,\beta} + \underline{\rho}_2f_{4,2,\beta},
\end{cases}
\]
where 
\[
\begin{cases}
 f_{1,i,\beta} := -[\partial^\beta, \mathcal{L}_{1,i}(H_{1},\delta,\underline{h}_1)]{\bm \varphi}_{1}
  \quad\mbox{for}\quad i=1,2,\ldots,N, \\ 
 f_{2,i,\beta} := -[\partial^\beta, \mathcal{L}_{2,i}(H_{2},b,\delta,\underline{h}_2)]{\bm \varphi}_{2}
  \quad\mbox{for}\quad i=1,2,\ldots,N^*, \\ 
 \nabla\cdot \bm{f}_{3,1,\beta} := -[\partial^\beta, \mathcal{L}_{1,0}(H_{1},\delta,\underline{h}_1)]{\bm \varphi}_{1}, \\
 \nabla\cdot \bm{f}_{3,2,\beta} := -[\partial^\beta, \mathcal{L}_{2,0}(H_{2},b,\delta,\underline{h}_2)]{\bm \varphi}_{2} , \\
 f_{4,1,\beta} :=  [\partial^\beta, {\bm l}_1(H_1) ] \cdot {\bm \varphi}_1, \\
 f_{4,2,\beta} := -[\partial^\beta, {\bm l}_2(H_2) ] \cdot {\bm \varphi}_2.
\end{cases}
\]
We put $\bm{f}_{1,\beta}=(0,f_{1,1,\beta},\ldots,f_{1,N,\beta})$ and $\bm{f}_{2,\beta}=(0,f_{2,1,\beta},\ldots,f_{2,N^*,\beta})$. 
Then, with a suitable decomposition $\bm{f}_{\ell,\beta} = \bm{f}_{\ell,\beta}^\mathrm{high} + \bm{f}_{\ell,\beta}^\mathrm{low}$ 
for $\ell=1,2$, we see that 
\begin{align*}
\|\bm{f}_{\ell,\beta}^\mathrm{high}\|_{H^{-1}} + (\underline{h}_\ell\delta)\|\bm{f}_{\ell,\beta}^\mathrm{low}\|_{L^2}
 + \|\bm{f}_{3,\ell,\beta}\|_{L^2} + \|\nabla f_{4,\ell}\|_{L^2}
\lesssim \|\nabla\bm{\varphi}_\ell\|_{H^{k-1}} + (\underline{h}_\ell\delta)^{-1}\|\bm{\varphi}_\ell'\|_{H^{k-1}}
\end{align*}
for $\ell=1,2$. 
Therefore, in view of the linearity of~\eqref{equationvarphi} the desired inequality for $k\geq1$ follows by induction on $k$. 
\end{proof}

From the above elliptic estimates we deduce the following bounds on time derivatives of regular solutions to the Kakinuma model~\eqref{Kakinuma-dimensionless}. 
We introduce a mathematical energy $E_m(t)$ for a solution $(\zeta,\bm{\phi}_1,\bm{\phi}_2)$ to the Kakinuma model by 
\begin{equation}\label{MathematicalEnergy}
E_m(t) := \|\zeta(t)\|_{H^m}^2 + \sum_{\ell=1,2}\underline{\rho}_\ell\underline{h}_\ell
 ( \|\nabla\bm{\phi}_\ell(t)\|_{H^m}^2 + (\underline{h}_\ell\delta)^{-2}\|\bm{\phi}_\ell'(t)\|_{H^m}^2 ),
\end{equation}
where $\bm{\phi}_1'=(\phi_{1,1},\ldots,\phi_{1,N})^\mathrm{T}$ and $\bm{\phi}_2'=(\phi_{2,1},\ldots,\phi_{2,N^*})^\mathrm{T}$.

\begin{lemma}\label{L.time-derivatives-and-elliptic}
Let $c, M_1, \underline{h}_\mathrm{min}$ be positive constants and $m$ an integer such that $m>\frac{n}{2}+1$. 
There exists a positive constant $C_1$ such that for any positive parameters 
$\underline{\rho}_1, \underline{\rho}_2, \underline{h}_1, \underline{h}_2, \delta$ satisfying the natural 
restrictions~\eqref{parameters}, $\underline{h}_1\delta, \underline{h}_2\delta \leq 1$, and the condition 
$\underline{h}_\mathrm{min} \leq \underline{h}_1,\underline{h}_2$, 
if a regular solution $(\zeta,\bm{\phi}_1,\bm{\phi}_2)$ to the Kakinuma model~\eqref{Kakinuma-dimensionless} with bottom topography $b\in W^{m+1,\infty}$ satisfy 
\[
\begin{cases}
 E_m(t) + \underline{h}_2^{-1}\|b\|_{W^{m+1,\infty}} \leq M_1, \\
 H_1(\bm{x},t) \geq c, \quad H_2(\bm{x},t) \geq c
  \quad\mbox{for}\quad \bm{x}\in\mathbf{R}^n, 0\leq t\leq T,
\end{cases}
\]
then we have 
\begin{align}\label{estimate-time-derivatives}
& \|\partial_t\zeta(t)\|_{H^{m-1}}^2 + \sum_{\ell=1,2}\underline{\rho}_\ell\underline{h}_\ell
 ( \|\nabla\partial_t\bm{\phi}_\ell(t)\|_{H^{m-1}}^2
   + (\underline{h}_\ell\delta)^{-2}\|\partial_t\bm{\phi}_\ell'(t)\|_{H^{m-1}}^2 ) \\
& + \|\partial_t^2\zeta(t)\|_{H^{m-2}}^2 + \sum_{\ell=1,2}\underline{\rho}_\ell\underline{h}_\ell
 ( \|\nabla\partial_t^2\bm{\phi}_\ell(t)\|_{H^{m-2}}^2
   + (\underline{h}_\ell\delta)^{-2}\|\partial_t^2\bm{\phi}_\ell'(t)\|_{H^{m-2}}^2 )
 \leq C_1E_m(t) \nonumber
\end{align}
for $0\leq t\leq T$. 
\end{lemma}

\begin{proof}
First, we recall that the Kakinuma model~\eqref{Kakinuma-dimensionless} can be written compactly as 
\eqref{Kakinuma-dimensionless-compact}. 
It follows from the first component of the first two equations in~\eqref{Kakinuma-dimensionless-compact} that 
$\partial_t\zeta$ can be written in two ways as 
$\partial_t\zeta=-\underline{h}_1\mathcal{L}_{1,0}\bm{\phi}_1=\underline{h}_2\mathcal{L}_{2,0}\bm{\phi}_2$, so that 
\begin{align*}
\|\partial_t\zeta\|_{H^{m-1}}^2
&= \min\{ \underline{h}_1^2\|\mathcal{L}_{1,0}\bm{\phi}_1\|_{H^{m-1}}^2, 
   \underline{h}_2^2\|\mathcal{L}_{2,0}\bm{\phi}_2\|_{H^{m-1}}^2 \} \\
&\lesssim \min\{ \underline{h}_1^2\|\nabla\bm{\phi}_1\|_{H^m}^2, 
   \underline{h}_2^2( \|\nabla\bm{\phi}_2\|_{H^m}^2 + \|\bm{\phi}_2'\|_{H^m}^2) \} \\
&\leq \min\biggl\{ \frac{\underline{h}_1}{\underline{\rho}_1}, \frac{\underline{h}_2}{\underline{\rho}_2} \biggr\} E_m
 \leq 2E_m,
\end{align*}
where we used~\eqref{parameter-relation}.

As for the estimate of $(\partial_t\bm{\phi}_1,\partial_t\bm{\phi}_2)$, we differentiate the compatibility 
conditions~\eqref{necessary-delta} with respect to time and use the last equation in~\eqref{Kakinuma-dimensionless-compact}. 
Then, we have 
\begin{equation}\label{equationdtphi}
\begin{cases}
 \mathcal{L}_{1,i} \partial_t\bm{\phi}_1 = f_{1,i} \quad\mbox{for}\quad i=1,2,\ldots,N, \\
 \mathcal{L}_{2,i} \partial_t\bm{\phi}_2 = f_{2,i} \quad\mbox{for}\quad i=1,2,\ldots,N^*, \\
 \underline{h}_1\mathcal{L}_{1,0} \partial_t\bm{\phi}_1
  + \underline{h}_2\mathcal{L}_{2,0} \partial_t\bm{\phi}_2 = \nabla\cdot\bm{f}_3, \\
 - \underline{\rho}_1\bm{l}_1 \cdot \partial_t\bm{\phi}_1
  + \underline{\rho}_2\bm{l}_2 \cdot \partial_t\bm{\phi}_2 = f_4,
\end{cases}
\end{equation}
where 
\begin{equation}\label{commutator-dt}
\begin{cases}
 f_{1,i} := -[\partial_t, \mathcal{L}_{1,i}(H_{1},\delta,\underline{h}_1)]\bm{\phi}_{1}
  \quad\mbox{for}\quad i=1,2,\ldots,N, \\ 
 f_{2,i} := -[\partial_t, \mathcal{L}_{2,i}(H_{2},b,\delta,\underline{h}_2)]\bm{\phi}_{2}
  \quad\mbox{for}\quad i=1,2,\ldots,N^*, \\ 
 \bm{f}_{3} := (\bm{u}_2-\bm{u}_1)\partial_t\zeta , \\
 f_{4} := \frac12\underline{\rho}_1\bigl( |{\bm u}_1|^2 + (\underline{h}_1\delta)^{-2} w_1^2 \bigr)
  - \frac12\underline{\rho}_2\bigl( |{\bm u}_2|^2 + (\underline{h}_2\delta)^{-2} w_2^2 \bigr)
  - \zeta.
\end{cases}
\end{equation}
Therefore, by Lemma~\ref{L.elliptic} we have 
\begin{align}\label{estimate-dt}
&\sum_{\ell=1,2}\underline{\rho}_\ell\underline{h}_\ell
 ( \|\nabla\partial_t\bm{\phi}_\ell\|_{H^{m-1}}^2
   + (\underline{h}_\ell\delta)^{-2}\|\partial_t\bm{\phi}_\ell'\|_{H^{m-1}}^2 ) \\
&\lesssim
 \sum_{\ell=1,2}\underline{\rho}_\ell\underline{h}_\ell
  (\underline{h}_\ell\delta)^2\|\bm{f}_\ell'\|_{H^{m-1}}^2
 +  \min\biggl\{\frac{\underline{\rho}_1}{\underline{h}_1},\frac{\underline{\rho}_2}{\underline{h}_2} \biggr\}
  \|\bm{f}_3\|_{H^{m-1}}^2
 + \|f_4\|_{H^m}^2, \nonumber
\end{align}
where $\bm{f}_1'=(f_{1,1},\ldots,f_{1,N})^\mathrm{T}$, $\bm{f}_2'=(f_{2,1},\ldots,f_{2,N^*})^\mathrm{T}$, 
and we used~\eqref{parameter-relation}. 
We proceed to evaluate the right-hand side. 
By writing down the operators $\mathcal{L}_{\ell,i}$ explicitly, we see that the operators do not include 
any derivatives of $H_{\ell}$. 
Therefore, we can write $f_{\ell,i}$ as 
\[
f_{1,i} = \biggl(\biggl(\frac{\partial}{\partial H_1}\mathcal{L}_{1,i}\biggr)\bm{\phi}_1\biggr)
  \underline{h}_1^{-1}\partial_t\zeta, \quad
f_{2,i} = -\biggl(\biggl(\frac{\partial}{\partial H_2}\mathcal{L}_{2,i}\biggr)\bm{\phi}_2\biggr)
  \underline{h}_2^{-1}\partial_t\zeta.
\]
We note also that the differential operators $\frac{\partial}{\partial H_\ell}\mathcal{L}_{\ell,i}$ have a 
similar structure as $\mathcal{L}_{\ell,i}$. 
Therefore, 
\begin{align*}
\underline{\rho}_\ell\underline{h}_\ell
  (\underline{h}_\ell\delta)^2\|\bm{f}_\ell'\|_{H^{m-1}}^2
&\lesssim \underline{\rho}_\ell\underline{h}_\ell
  (\underline{h}_\ell\delta)^2 ( \|\nabla\bm{\phi}_\ell\|_{H^m}^2
   + (\underline{h}_\ell\delta)^{-4}\|\bm{\phi}_\ell'\|_{H^{m-1}}^2 ) \|\underline{h}_\ell^{-1}\partial_t\zeta\|_{H^{m-1}}^2 \\
&\lesssim E_m^2 \quad\mbox{for}\quad \ell=1,2,
\end{align*}
where, here and henceforth, we utilize fully our restriction $\underline{h}_1^{-1},\underline{h}_2^{-1} \lesssim 1$. 
In view of the definition~\eqref{def-uw} of $\bm{u}_1, \bm{u}_2,w_1$, and $w_2$, we see easily that 
\begin{equation}\label{estimates-uw}
\sum_{\ell=1,2}\underline{\rho}_\ell\underline{h}_\ell(\|\bm{u}_\ell\|_{H^m}^2
 + (\underline{h}_\ell\delta)^{-2}\|w_\ell\|_{H^m}^2 ) \lesssim E_m.
\end{equation}
We evaluate the term on $\bm{f}_3$ as 
\begin{align*}
\min\biggl\{\frac{\underline{\rho}_1}{\underline{h}_1},\frac{\underline{\rho}_2}{\underline{h}_2} \biggr\}
  \|\bm{f}_3\|_{H^{m-1}}^2
&\lesssim \sum_{\ell=1,2} \frac{\underline{\rho}_\ell}{\underline{h}_\ell}\|\bm{u}_\ell\partial_t\zeta\|_{H^{m-1}}^2 \\
&\lesssim \sum_{\ell=1,2} \underline{\rho}_\ell\underline{h}_\ell
  \|\bm{u}_\ell\|_{H^{m-1}}^2\|\underline{h}_\ell^{-1}\partial_t\zeta\|_{H^{m-1}}^2 \\
&\lesssim E_m^2.
\end{align*}
Similarly, we have 
\begin{align*}
\|f_4\|_{H^m}^2
&\lesssim \sum_{\ell=1,2} \underline{\rho}_\ell^2 ( \|{\bm u}_\ell\|_{H^m}^2 + (\underline{h}_1\delta)^{-2} \|w_\ell\|_{H^m}^2 )^2
 + \|\zeta\|_{H^m}^2 \\
&\lesssim \sum_{\ell=1,2} \underline{h}_\ell^{-2} \{ \underline{\rho}_\ell\underline{h}_\ell
 ( \|{\bm u}_\ell\|_{H^m}^2 + (\underline{h}_1\delta)^{-2} \|w_\ell\|_{H^m}^2 ) \}^2 + \|\zeta\|_{H^m}^2 \\
&\lesssim E_m^2+E_m. 
\end{align*}
Plugging in~\eqref{estimate-dt} the above estimates, we obtain the desired estimate for $(\partial_t\bm{\phi}_1,\partial_t\bm{\phi}_2)$.

Finally, the estimate of $\partial_t^2\zeta$ can be obtained by differentiating 
$\partial_t\zeta=-\underline{h}_1\mathcal{L}_{1,0}\bm{\phi}_1=\underline{h}_2\mathcal{L}_{2,0}\bm{\phi}_2$ 
with respect to time. 
Then, the estimate of $(\partial_t^2\bm{\phi}_1,\partial_t^2\bm{\phi}_2)$ can be obtained by differentiating 
\eqref{equationdtphi} with respect to time once more and applying Lemma~\ref{L.elliptic}. 
\end{proof}

\begin{remark}\label{remark-ID}
{\rm
In view of the above arguments, we see easily that for the Kakinuma model~\eqref{Kakinuma-dimensionless}, 
$(\partial_t\bm{\phi}_1, \partial_t\bm{\phi}_2)|_{t=0}$ can be determined from the initial data 
$(\zeta_{(0)},\bm{\phi}_{1(0)},\bm{\phi}_{2(0)})$ and the bottom topography $b$, although the hypersurface $t=0$ is characteristic for the model. 
They are unique up to an additive constant of the form $(\mathcal{C}\underline{\rho}_2,\mathcal{C}\underline{\rho}_1)$ 
to $(\partial_t\phi_{1,0},\partial_t\phi_{2,0})|_{t=0}$. 
Particularly, $(\partial_t\bm{\phi}_1', \partial_t\bm{\phi}_2')|_{t=0}$ and hence $a|_{t=0}$ with the function 
$a$ given in~\eqref{def-a} can be uniquely determined from the data. 
}
\end{remark}

\section{Uniform energy estimates; proof of Theorem~\ref{theorem-uniform}}\label{S.hyperbolic}

In this section we provide uniform energy estimates for solutions to the Kakinuma model. 
Consequently, we prove Theorem~\ref{theorem-uniform}. 
We recall that the Kakinuma model~\eqref{Kakinuma-dimensionless} can be written compactly as 
\begin{equation}\label{Kakinuma-approx}
\begin{cases}
 \displaystyle
 {\bm l}_1(H_1)\partial_t\zeta + \underline{h}_1 L_1(H_1,\delta,\underline{h}_1){\bm \phi}_1 = {\bm 0}, \\
 \displaystyle
 {\bm l}_2(H_2)\partial_t\zeta - \underline{h}_2 L_2(H_2,b,\delta,\underline{h}_2){\bm \phi}_2 = {\bm 0}, \\
 \underline{\rho}_1\bigl\{ {\bm l}_1(H_1) \cdot \partial_t{\bm \phi}_1 
   + \frac12\bigl( |{\bm u}_1|^2 + (\underline{h}_1\delta)^{-2} w_1^2 \bigr) \bigr\} \\
 \quad
 - \underline{\rho}_2\bigl\{ {\bm l}_2(H_2) \cdot \partial_t{\bm \phi}_2 
  + \frac12\bigl( |{\bm u}_2|^2 +  (\underline{h}_2\delta)^{-2} w_2^2 \bigr) \bigr\} 
 - \zeta = 0, 
\end{cases}
\end{equation}
where we recall that $H_1 := 1 - \underline{h}_1^{-1}\zeta$, $H_2:= 1 + \underline{h}_2^{-1}\zeta- \underline{h}_2^{-1}b$, 
${\bm \phi}_1 := (\phi_{1,0},\phi_{1,1},\ldots,\phi_{1,N})^\mathrm{T}$, 
${\bm \phi}_2 := (\phi_{2,0},\phi_{2,1},\ldots,\phi_{2,N^*})^\mathrm{T}$, and 
${\bm l}_1$, ${\bm l}_2$, $L_1$, $L_2$, ${\bm u}_1$, ${\bm u}_2$, $w_1$, $w_2$ are defined in Section~\ref{S.main-results}.

\subsection{Analysis of linearized equations}\label{Analysis-LEs}
Before deriving linearized equations to the Kakinuma model~\eqref{Kakinuma-approx}, 
we introduce some more notations. 
For $\ell=1,2$, the coefficient matrices of the principal part and the singular part with respect to 
the small parameter $\delta_\ell=\underline{h}_\ell\delta$ of the operator $L_{\ell}$ are denoted by 
$A_\ell(H_\ell)$ and $C_\ell(H_\ell)$, respectively, that is, 
\begin{equation}\label{def-A}
\begin{cases}
 \displaystyle
 A_1(H_1) := \left( \frac{1}{2(i+j)+1}H_1^{2(i+j)+1} \right)_{0\leq i,j\leq N}, \\
 \displaystyle
 A_2(H_2) := \left( \frac{1}{p_i+p_j+1}H_2^{p_i+p_j+1} \right)_{0\leq i,j\leq N^*}, 
\end{cases}
\end{equation}
and 
\begin{equation}\label{def-C}
\begin{cases}
 \displaystyle
 C_1(H_1) := \left( \frac{4ij}{2(i+j)-1}H_1^{2(i+j)-1} \right)_{0\leq i,j\leq N}, \\
 \displaystyle
 C_2(H_2) := \left( \frac{p_ip_j}{p_i+p_j-1}H_2^{p_i+p_j-1} \right)_{0\leq i.j\leq N^*}.
\end{cases}
\end{equation}
We put also 
\begin{equation}\label{def-B2}
\begin{cases}
 \displaystyle
 B_2(H_2) := \left( \frac{p_j}{p_i+p_j}H_2^{p_i+p_j} \right)_{0\leq i,j\leq N^*}, \\
 \tilde{B}_2(H_2) := B_2(H_2) - B_2(H_2)^\mathrm{T}, \\
 \tilde{C}_2(H_2,\underline{h}_2^{-1}b) := |\underline{h}_2^{-1}\nabla b|^2 C_2(H_2) + \underline{h}_2^{-1}(\Delta b)B_2(H_2). 
\end{cases}
\end{equation}
In the above expressions, we used the notational convention $0/0 = 0$.
Then, the operators $L_1$ and $L_2$ can also be written as 
\begin{equation}\label{AE-L}
\begin{cases}
 L_1\bm{\phi}_1 = -A_1\Delta\bm{\phi}_1 - \bm{l}_1(\bm{u}_1\cdot\nabla H_1) + (\underline{h}_1\delta)^{-2}C_1\bm{\phi}_1, \\
 L_2\bm{\phi}_2 = -A_2\Delta\bm{\phi}_2 - \bm{l}_2(\bm{u}_2\cdot\nabla H_2) + (\underline{h}_2\delta)^{-2}C_2\bm{\phi}_2
  + \tilde{B}_2 (\underline{h}_2^{-1}\nabla b\cdot\nabla)\bm{\phi}_2 + \tilde{C}_2\bm{\phi}_2. 
\end{cases}
\end{equation}
For $\ell=1,2$, we decompose the operator $L_\ell$ as $L_\ell=L_\ell^\mathrm{pr}+L_\ell^\mathrm{low}$, where 
\begin{equation}\label{del-A-principal}
L_\ell^\mathrm{pr}(H_\ell)\bm{\varphi}_\ell := -\sum_{l=1}^n\partial_l(A_\ell(H_\ell)\partial_l\bm{\varphi}_\ell)
 + (\underline{h}_\ell\delta)^{-2}C_\ell(H_\ell)\bm{\varphi}_\ell. 
\end{equation}

We now linearize the Kakinuma model~\eqref{Kakinuma-approx} around an arbitrary flow $(\zeta,\bm{\phi}_1,\bm{\phi}_2)$ 
and denote the variation by $(\dot{\zeta},\dot{\bm{\phi}}_1,\dot{\bm{\phi}}_2)$. 
After neglecting lower order terms, the linearized equations have the form 
\begin{equation}\label{linear-Kakinuma}
\begin{cases}
 \displaystyle
 {\bm l}_1(H_1)(\partial_t+\bm{u}_1\cdot\nabla)\dot{\zeta}
  + \underline{h}_1 L_1^\mathrm{pr}(H_1,\delta,\underline{h}_1)\dot{\bm{\phi}}_1 = \dot{{\bm f}}_1, \\
 \displaystyle
 {\bm l}_2(H_2)(\partial_t+\bm{u}_2\cdot\nabla)\dot{\zeta}
  - \underline{h}_2 L_2^\mathrm{pr}(H_2,\delta,\underline{h}_2)\dot{\bm{\phi}}_2 = \dot{{\bm f}}_2, \\
 \underline{\rho}_1{\bm l}_1(H_1)\cdot( \partial_t+\bm{u}_1\cdot\nabla )\dot{\bm{\phi}}_1 
  - \underline{\rho}_2{\bm l}_2(H_2)\cdot( \partial_t+\bm{u}_1\cdot\nabla )\dot{\bm{\phi}}_2 
  - a\dot{\zeta} = \dot{f}_0,
\end{cases}
\end{equation}
where the function $a$ is defined by~\eqref{def-a}. 
In order to derive a good symmetric structure of the equations, following the companion paper~\cite{DucheneIguchi2020} 
we introduce 
\begin{equation}\label{def-theta}
\theta_1 := \frac{\underline{\rho}_2\underline{h}_1H_1 \alpha_1}{
  \underline{\rho}_1\underline{h}_2H_2 \alpha_2 + \underline{\rho}_2\underline{h}_1H_1 \alpha_1}, \qquad
\theta_2 := \frac{\underline{\rho}_1\underline{h}_2H_2 \alpha_2}{
  \underline{\rho}_1\underline{h}_2H_2 \alpha_2 + \underline{\rho}_2\underline{h}_1H_1 \alpha_1},
\end{equation}
where
\begin{equation}\label{def-alpha}
\alpha_\ell := \frac{\det A_{\ell,0}}{\det \tilde{A}_{\ell,0}}, \qquad
\tilde{A}_{\ell,0} := 
\begin{pmatrix}
 0 & \bm{1}^\mathrm{T} \\
 -\bm{1} & A_{\ell,0}
\end{pmatrix}, \qquad
A_{\ell,0} := A_\ell(1)
\end{equation}
for $\ell=1,2$ and $\bm{1}:=(1,\ldots,1)^\mathrm{T}$. 
Then, we have $\theta_1+\theta_2=1$.
We recall that $\alpha_1$ and $\alpha_2$ are positive constants depending only on $N$ and the nonnegative integers $0=p_0<p_1<\ldots<p_{N^*}$, respectively, 
and go to $0$ as $N,N^*\to\infty$. 
We also introduce 
\[
\bm{u} := \theta_2\bm{u}_1+\theta_1\bm{u}_2, \qquad \bm{v}:=\bm{u}_2-\bm{u}_1.
\]
Then, we have $\bm{u}_1=\bm{u}-\theta_1\bm{v}$ and $\bm{u}_2=\bm{u}+\theta_2\bm{v}$. 
Plugging these into the linearized equations~\eqref{linear-Kakinuma}, 
we can write them in a matrix form as 
\begin{equation}\label{Kakinuma-MF}
\mathscr{A}_1(\partial_t+\bm{u}\cdot\nabla)\dot{\bm{U}} + \mathscr{A}_0^\mathrm{mod}\dot{\bm{U}} = \dot{\bm{F}},
\end{equation}
where 
\[
\dot{\bm{U}} := \begin{pmatrix} \dot{\zeta} \\ \dot{\bm{\phi}}_1 \\ \dot{\bm{\phi}}_2 \end{pmatrix}, \qquad
\dot{\bm{F}} := \begin{pmatrix} \dot{f}_0 \\
 \underline{\rho}_1(\dot{\bm{f}}_1-(\nabla\cdot(\theta_1\bm{l}_1\otimes\bm{v}))\dot{\zeta} \\
 \underline{\rho}_2(\dot{\bm{f}}_2-(\nabla\cdot(\theta_2\bm{l}_2\otimes\bm{v}))\dot{\zeta} \end{pmatrix}, 
\]
and 
\begin{align*}
& \mathscr{A}_1 := 
\begin{pmatrix}
 0 & -\underline{\rho}_1\bm{l}_1^\mathrm{T} & \underline{\rho}_2\bm{l}_2^\mathrm{T} \\
 \underline{\rho}_1\bm{l}_1 & O & O \\
 -\underline{\rho}_2\bm{l}_2 & O & O
\end{pmatrix}, \\
& \mathscr{A}_0^\mathrm{mod} := 
\begin{pmatrix}
 a & \underline{\rho}_1\theta_1\bm{l}_1^\mathrm{T}(\bm{v}\cdot\nabla) 
  & \underline{\rho}_2\theta_2\bm{l}_2^\mathrm{T}(\bm{v}\cdot\nabla) \\
 (\bm{v}\cdot\nabla)^*(\underline{\rho}_1\theta_1\bm{l}_1\,\cdot\,) 
  & \underline{\rho}_1\underline{h}_1L_1^\mathrm{pr} & O \\
 (\bm{v}\cdot\nabla)^*(\underline{\rho}_2\theta_1\bm{l}_2\,\cdot\,) 
  & O & \underline{\rho}_2\underline{h}_2L_2^\mathrm{pr}
\end{pmatrix}.
\end{align*}
Here, $(\bm{v}\cdot\nabla)^*$ denotes the adjoint operator of $\bm{v}\cdot\nabla$ in $L^2$, that is, 
$(\bm{v}\cdot\nabla)^*f=-\nabla\cdot(f\bm{v})$. 
We note that $\mathscr{A}_1$ is a skew-symmetric matrix and $\mathscr{A}_0^\mathrm{mod}$ is symmetric in $L^2$. 
Therefore, the corresponding energy function is given by 
$(\mathscr{A}_0^\mathrm{mod}\dot{\bm{U}},\dot{\bm{U}})_{L^2}$. 
We put 
\begin{equation}\label{def-E}
\mathscr{E}(\dot{\bm{U}}) := \|\dot{\zeta}\|_{L^2}^2
 + \sum_{\ell=1,2}\underline{\rho}_\ell\underline{h}_\ell( \|\nabla\dot{\bm{\phi}}_\ell\|_{L^2}^2
  + (\underline{h}_\ell\delta)^{-2}\|\dot{\bm{\phi}}_\ell'\|_{L^2}^2 ).
\end{equation}
The following lemma shows that $(\mathscr{A}_0^\mathrm{mod}\dot{U},\dot{U})_{L^2} \simeq \mathscr{E}(\dot{\bm{U}})$ 
under
the non-cavitation assumption and the stability condition, 
stated respectively as~\eqref{NonCavitation} and~\eqref{Stability} in Theorem~\ref{theorem-uniform}.

\begin{lemma}\label{L.A0-coercive}
Let $c, M, \underline{h}_\mathrm{min}$ be positive constants. 
There exists a positive constant $C$ such that for any positive parameters 
$\underline{\rho}_1,\underline{\rho}_2,\underline{h}_1,\underline{h}_2,\delta$ satisfying the condition 
$\underline{h}_\mathrm{min} \leq \underline{h}_1,\underline{h}_2$, if $H_1,H_2,\bm{u}_1,\bm{u}_2$, and the function $a$ satisfy 
\begin{equation}\label{conditions-linear1}
\begin{cases}
 \displaystyle
 \sum_{\ell=1,2}\bigl( \|H_\ell\|_{L^\infty} + \sqrt{ \underline{\rho}_\ell\underline{h}_\ell }
  \|\bm{u}_\ell\|_{L^\infty} \bigr) + \|a\|_{L^\infty} \leq M, \\
 \displaystyle
 a(\bm{x}) - \frac{ \underline{\rho}_1\underline{\rho}_2 }{ \underline{\rho}_1\underline{h}_2H_2(\bm{x})\alpha_2
  + \underline{\rho}_2\underline{h}_1H_1(\bm{x})\alpha_1 }|\bm{u}_2(\bm{x})-\bm{u}_1(\bm{x})|^2 \geq c, \\
 H_1(\bm{x}) \geq c, \quad H_2(\bm{x}) \geq c \quad\mbox{for}\quad \bm{x}\in\mathbf{R}^n,
\end{cases}
\end{equation}
then for any $\dot{\bm{U}} = (\dot{\zeta},\dot{\bm{\phi}}_1,\dot{\bm{\phi}}_2)^\mathrm{T} \in 
L^2\times(\mathring{H}^1\times(H^1)^N)\times(\mathring{H}^1\times(H^1)^{N^*})$ we have 
\[
C^{-1}\mathscr{E}(\dot{\bm{U}}) \leq (\mathscr{A}_0^\mathrm{mod}\dot{\bm{U}},\dot{\bm{U}})_{L^2}
\leq C\mathscr{E}(\dot{\bm{U}}). 
\]
\end{lemma}

\begin{proof}
This lemma can be shown along with the proof of~\cite[Lemma 7.4]{DucheneIguchi2020}. 
For the sake of completeness, we sketch the proof. 
We first note that 
\begin{align*}
(\mathscr{A}_0^\mathrm{mod}\dot{\bm{U}},\dot{\bm{U}})_{L^2}
&= (a\dot{\zeta},\dot{\zeta})_{L^2}
 + \sum_{\ell=1,2} \{ \underline{\rho}_\ell\underline{h}_\ell (L_\ell^\mathrm{pr}\dot{\bm{\phi}}_\ell, \dot{\bm{\phi}}_\ell)_{L^2}
 + 2\underline{\rho}_\ell (\theta_\ell\bm{l}_\ell\cdot(\bm{v}\cdot\nabla)\dot{\bm{\phi}}_\ell,\dot{\zeta})_{L^2} \} \\
&= (a\dot{\zeta},\dot{\zeta})_{L^2}
 + \sum_{\ell=1,2} \biggl\{ \underline{\rho}_\ell\underline{h}_\ell \biggl(
  \sum_{l=1}^n(A_\ell\partial_l\dot{\bm{\phi}}_\ell,\partial_l\dot{\bm{\phi}}_\ell)_{L^2}
   + (\underline{h}_\ell\delta)^{-2}(C_\ell\dot{\bm{\phi}}_\ell,\dot{\bm{\phi}}_\ell)_{L^2} \biggr) \\
&\makebox[9em]{}
 + 2\underline{\rho}_\ell(\theta_\ell\bm{v}\cdot(\bm{l}_\ell\otimes\nabla)^\mathrm{T}\dot{\bm{\phi}}_\ell, \dot{\zeta})_{L^2}
 \biggr\},
\end{align*}
where we used the identity $\bm{a}\cdot(\bm{v}\cdot\nabla)\bm{\varphi} = \bm{v}\cdot(\bm{a}\otimes\nabla)^\mathrm{T}\bm{\varphi}$. 
On the other hand, we can put 
\[
\begin{pmatrix}
 q_\ell(H_\ell) & \bm{q}_\ell(H_\ell)^\mathrm{T} \\
 -\bm{q}_\ell(H_\ell) & Q_\ell(H_\ell)
\end{pmatrix}
:= 
\begin{pmatrix}
 0 & \bm{l}_\ell(H_\ell)^\mathrm{T} \\
 -\bm{l}_\ell(H_\ell) & A_\ell(H_\ell)
\end{pmatrix}^{-1}
\]
for $\ell=1,2$. 
Then, we see that $q_\ell(H_\ell)=H_\ell\alpha_\ell$ and that $Q_\ell(H_\ell)$ is nonnegative. 
Moreover, the identity 
\begin{equation}\label{decomposition-A}
A_\ell(H_\ell)\bm{\varphi}_\ell\cdot\bm{\varphi}_\ell
= q_\ell(H_\ell)(\bm{l}_\ell(H_\ell) \cdot \bm{\varphi}_\ell)^2
 + Q_\ell(H_\ell)A_\ell(H_\ell)\bm{\varphi}_\ell \cdot A_\ell(H_\ell)\bm{\varphi}_\ell
\end{equation}
holds for any $\bm{\varphi}_\ell$. 
Therefore, 
\begin{align*}
\sum_{l=1}^n(A_\ell\partial_l\dot{\bm{\phi}}_\ell,\partial_l\dot{\bm{\phi}}_\ell)_{L^2}
&= \sum_{l=1}^n\{ (q_\ell\bm{l}_\ell\cdot\partial_l\dot{\bm{\phi}}_\ell,\bm{l}_\ell\cdot\partial_l\dot{\bm{\phi}}_\ell)_{L^2}
 + (Q_\ell A_\ell\partial_l\dot{\bm{\phi}}_\ell,A_\ell\partial_l\dot{\bm{\phi}}_\ell)_{L^2} \} \\
&= (H_\ell\alpha_\ell (\bm{l}_\ell\otimes\nabla)^\mathrm{T}\dot{\bm{\phi}}_\ell, 
  (\bm{l}_\ell\otimes\nabla)^\mathrm{T}\dot{\bm{\phi}}_\ell)_{L^2}
 + \sum_{l=1}^n(Q_\ell A_\ell\partial_l\dot{\bm{\phi}}_\ell,A_\ell\partial_l\dot{\bm{\phi}}_\ell)_{L^2},
\end{align*}
so that 
\begin{align*}
(\mathscr{A}_0^\mathrm{mod}\dot{\bm{U}},\dot{\bm{U}})_{L^2}
&= (a\dot{\zeta},\dot{\zeta})_{L^2}
 + \sum_{\ell=1,2} \{ \underline{\rho}_\ell\underline{h}_\ell
   (H_\ell\alpha_\ell (\bm{l}_\ell\otimes\nabla)^\mathrm{T}\dot{\bm{\phi}}_\ell, 
    (\bm{l}_\ell\otimes\nabla)^\mathrm{T}\dot{\bm{\phi}}_\ell)_{L^2} \\
&\makebox[8em]{}
  + 2\underline{\rho}_\ell(\theta_\ell\bm{v}\cdot(\bm{l}_\ell\otimes\nabla)^\mathrm{T}\dot{\bm{\phi}}_\ell, \dot{\zeta})_{L^2} \} \\
&\quad\;
 + \sum_{\ell=1,2} \underline{\rho}_\ell\underline{h}_\ell \biggl\{
  \sum_{l=1}^n (Q_\ell A_\ell\partial_l\dot{\bm{\phi}}_\ell,A_\ell\partial_l\dot{\bm{\phi}}_\ell)_{L^2}
  + (\underline{h}_\ell\delta)^{-2}(C_\ell\dot{\bm{\phi}}_\ell,\dot{\bm{\phi}}_\ell)_{L^2} \biggr\} \\
&=: I_1+I_2.
\end{align*}
We proceed to evaluate $I_1$. 
\begin{align*}
I_1 &\geq \int_{\mathbf{R}^n}\biggl\{ a\dot{\zeta}^2
 + \sum_{\ell=1,2}\bigl(
  \underline{\rho}_\ell\underline{h}_\ell H_\ell\alpha_\ell|(\bm{l}_\ell\otimes\nabla)^\mathrm{T}\dot{\bm{\phi}}_\ell|^2
 - 2\underline{\rho}_\ell\theta_\ell|\bm{v}||(\bm{l}_\ell\otimes\nabla)^\mathrm{T}\dot{\bm{\phi}}_\ell||\dot{\zeta}| \bigr)
 \biggr\}\mathrm{d}\bm{x} \\
&= \int_{\mathbf{R}^n} \mathfrak{A}_0
\begin{pmatrix}
 \dot{\zeta} \\ 
 \sqrt{ \underline{\rho}_1\underline{h}_1 } |(\bm{l}_1\otimes\nabla)^\mathrm{T}\dot{\bm{\phi}}_1| \\
 \sqrt{ \underline{\rho}_2\underline{h}_2 } |(\bm{l}_2\otimes\nabla)^\mathrm{T}\dot{\bm{\phi}}_2|
\end{pmatrix}
\cdot
\begin{pmatrix}
 \dot{\zeta} \\ 
 \sqrt{ \underline{\rho}_1\underline{h}_1 } |(\bm{l}_1\otimes\nabla)^\mathrm{T}\dot{\bm{\phi}}_1| \\
 \sqrt{ \underline{\rho}_2\underline{h}_2 } |(\bm{l}_2\otimes\nabla)^\mathrm{T}\dot{\bm{\phi}}_2|
\end{pmatrix}
\mathrm{d}\bm{x},
\end{align*}
where the matrix $\mathfrak{A}_0$ is given by 
\[
\mathfrak{A}_0 = 
\begin{pmatrix}
 a & - \sqrt{\underline{\rho}_1/\underline{h}_1} \theta_1|\bm{v}|
  & - \sqrt{\underline{\rho}_2/\underline{h}_2} \theta_2|\bm{v}| \\
 - \sqrt{\underline{\rho}_1/\underline{h}_1} \theta_1|\bm{v}| & H_1\alpha_1 & 0 \\
 - \sqrt{\underline{\rho}_2/\underline{h}_2} \theta_2|\bm{v}| & 0 & H_2\alpha_2
\end{pmatrix}.
\]
Here, we see that 
\begin{align*}
\det\mathfrak{A}_0
= H_1H_2\alpha_1\alpha_2\biggl( a
 - \frac{ \underline{\rho}_1\underline{\rho}_2 }{ \underline{\rho}_1\underline{h}_2H_2\alpha_2
  + \underline{\rho}_2\underline{h}_1H_1\alpha_1 }|\bm{v}|^2 \biggr)
\geq c^3\alpha_1\alpha_2 > 0,
\end{align*}
so that $\mathfrak{A}_0$ is positive definite by Sylvester's criterion. 
Moreover, $\operatorname{tr}\mathfrak{A}_0 \leq \max\{1,\alpha_1,\alpha_2\}\, M \lesssim 1$ 
and the minimal eigenvalue of the matrix $\mathfrak{A}_0$ is bounded from below by 
$4\det\mathfrak{A}_0/(\operatorname{tr}\mathfrak{A}_0)^2 \gtrsim 1$. 
Therefore, we obtain 
\[
I_1 \gtrsim \int_{\mathbf{R}^n} \biggl( \dot{\zeta}^2
 + \sum_{\ell=1,2}\underline{\rho}_\ell\underline{h}_\ell H_\ell\alpha_\ell
  |(\bm{l}_\ell\otimes\nabla)^\mathrm{T}\dot{\bm{\phi}}_\ell|^2 \biggr)\mathrm{d}\bm{x}.
\]
As for $I_2$, it is easy to see that 
$(C_\ell\dot{\bm{\phi}}_\ell,\dot{\bm{\phi}}_\ell)_{L^2} \simeq \|\dot{\bm{\phi}}_\ell'\|_{L^2}^2$ for $\ell=1,2$. 
Summarizing the above estimates and using the decomposition~\eqref{decomposition-A} again, we obtain 
$(\mathscr{A}_0^\mathrm{mod}\dot{\bm{U}},\dot{\bm{U}})_{L^2} \gtrsim \mathscr{E}(\dot{\bm{U}})$.

In order to obtain the estimate of $(\mathscr{A}_0^\mathrm{mod}\dot{\bm{U}},\dot{\bm{U}})_{L^2}$ from above, 
it is sufficient to show that each element of the matrix $\mathfrak{A}_0$ is uniformly bounded. 
Since $\theta_1+\theta_2=1$, we have 
\[
\begin{cases}
 \sqrt{\underline{\rho}_1/\underline{h}_1} \theta_1|\bm{v}|
  \leq \underline{h}_1^{-1}\sqrt{\underline{\rho}_1\underline{h}_1}|\bm{u}_1|
   + \sqrt{\underline{\rho}_1/\underline{h}_1} \theta_1|\bm{u}_2|, \\
 \sqrt{\underline{\rho}_2/\underline{h}_2} \theta_2|\bm{v}|
  \leq \sqrt{\underline{\rho}_2/\underline{h}_2} \theta_2|\bm{u}_1|
   + \underline{h}_2^{-1}\sqrt{\underline{\rho}_2\underline{h}_2}|\bm{u}_2|.
\end{cases}
\]
Here, we see that 
\begin{align*}
\sqrt{\underline{\rho}_1/\underline{h}_1} \theta_1|\bm{u}_2|
&= \frac{1}{\underline{h}_2}\sqrt{\frac{H_1\alpha_1}{H_2\alpha_2}}
 \frac{\sqrt{ (\underline{\rho}_1\underline{h}_2H_2\alpha_2) (\underline{\rho}_2\underline{h}_1H_1\alpha_1) }}{
  \underline{\rho}_1\underline{h}_2H_2\alpha_2+\underline{\rho}_2\underline{h}_1H_1\alpha_1}
 \sqrt{\underline{\rho}_2\underline{h}_2}|\bm{u}_2| \\
&\leq \frac{1}{2\underline{h}_2}\sqrt{\frac{H_1\alpha_1}{H_2\alpha_2}}\sqrt{\underline{\rho}_2\underline{h}_2}|\bm{u}_2| \\
&\leq \frac{1}{2\underline{h}_\mathrm{min} \sqrt{\frac{M\alpha_1}{c\alpha_2}}M}
 \lesssim 1.
\end{align*}
Similarly, we have $\sqrt{\underline{\rho}_2/\underline{h}_2} \theta_2|\bm{u}_1| \lesssim 1$. 
Therefore, we obtain 
$(\mathscr{A}_0^\mathrm{mod}\dot{\bm{U}},\dot{\bm{U}})_{L^2} \lesssim \mathscr{E}(\dot{\bm{U}})$. 
\end{proof}

In the following Lemma we provide uniform energy estimates for regular solutions to the linearized Kakinuma model~\eqref{linear-Kakinuma}.

\begin{proposition}\label{L.energy-estimate}
Let $c,M,M_1,\underline{h}_\mathrm{min}$ be positive constants. 
There exist positive constants $C=C(c,M,\underline{h}_\mathrm{min})$ and $C_1=C_1(c,M,M_1,\underline{h}_\mathrm{min})$ 
such that for any positive parameters $\underline{\rho}_1, \underline{\rho}_2, \underline{h}_1, \underline{h}_2, 
\delta$ satisfying the natural restrictions~\eqref{parameters} and the condition 
${\underline{h}_\mathrm{min} \leq \underline{h}_1, \underline{h}_2}$, if $H_1,H_2,\bm{u}_1,\bm{u}_2$, and the function $a$ satisfy 
\eqref{conditions-linear1} and 
\[
\sum_{\ell=1,2}\bigl( \|\partial_t H_\ell\|_{L^\infty} + \|\nabla H_\ell\|_{L^\infty} + \underline{\rho}_\ell\underline{h}_\ell
  (\|\partial_t\bm{u}_\ell\|_{L^\infty}^2+\|\nabla\bm{u}_\ell\|_{L^\infty}^2) \bigr)
  + \|\partial_t a\|_{L^\infty} + \|\nabla a\|_{L^\infty} \leq M_1, \\
\]
then for any regular solution $\dot{\bm{U}}=(\dot{\zeta},\dot{\bm{\phi}}_1,\dot{\bm{\phi}}_2)^\mathrm{T}$ to the 
linearized Kakinuma model~\eqref{linear-Kakinuma} we have 
\begin{multline*}
\mathscr{E}(\dot{\bm{U}}(t))
\leq C\mathrm{e}^{C_1t}\mathscr{E}(\dot{\bm{U}}(0))
 +C_1\int_0^t\mathrm{e}^{C_1(t-\tau)}\biggl\{ 
  \|\dot{f}_0(\tau)\|_{H^1}( \|\partial_t\dot{\zeta}(\tau)\|_{H^{-1}}+\|\dot{\zeta}(\tau)\|_{L^2} ) \\
 + \sum_{\ell=1,2}\underline{\rho}_\ell( \|\dot{\bm{f}}_\ell(\tau)\|_{L^2} + \|\dot{\zeta}(\tau)\|_{L^2})
  \|(\partial_t\dot{\bm{\phi}}_\ell(\tau),\nabla\dot{\bm{\phi}}_\ell(\tau))\|_{L^2} \biggr\}\mathrm{d}\tau.
\end{multline*}
\end{proposition}

\begin{proof}
We deduce from~\eqref{Kakinuma-MF} that 
\begin{align*}
&\frac{\mathrm{d}}{\mathrm{d}t}(\mathscr{A}_0^\mathrm{mod}\dot{\bm{U}},\dot{\bm{U}})_{L^2} \\
&= ([\partial_t,\mathscr{A}_0^\mathrm{mod}]\dot{\bm{U}},\dot{\bm{U}})_{L^2}
 + 2(\mathscr{A}_0^\mathrm{mod}\partial_t\dot{\bm{U}},\dot{\bm{U}})_{L^2} \\
&= ([\partial_t,\mathscr{A}_0^\mathrm{mod}]\dot{\bm{U}},\dot{\bm{U}})_{L^2}
 + 2((\partial_t+\bm{u}\cdot\nabla)\dot{\bm{U}},\mathscr{A}_0^\mathrm{mod}\dot{\bm{U}})_{L^2}
 - 2((\bm{u}\cdot\nabla)\dot{\bm{U}},\mathscr{A}_0^\mathrm{mod}\dot{\bm{U}})_{L^2} \\
&= ([\partial_t,\mathscr{A}_0^\mathrm{mod}]\dot{\bm{U}},\dot{\bm{U}})_{L^2}
 - 2((\bm{u}\cdot\nabla)\dot{\bm{U}},\mathscr{A}_0^\mathrm{mod}\dot{\bm{U}})_{L^2}
 + 2((\partial_t+\bm{u}\cdot\nabla)\dot{\bm{U}},\dot{\bm{F}})_{L^2} \\
&=: I_1+I_2+I_3,
\end{align*}
where we used the fact that $\mathscr{A}_0^\mathrm{mod}$ is a symmetric operator in $L^2$ and that $\mathscr{A}_1$ 
is a skew-symmetric matrix. 
As for $I_1$, we have 
\begin{align*}
I_1 = ((\partial_t a)\dot{\zeta},\dot{\zeta})_{L^2}
 + \sum_{\ell=1,2} \biggl\{ \underline{\rho}_\ell\underline{h}_\ell \biggl(
  \sum_{l=1}^n((\partial_t A_\ell)\partial_l\dot{\bm{\phi}}_\ell,\partial_l\dot{\bm{\phi}}_\ell)_{L^2}
   + (\underline{h}_\ell\delta)^{-2}((\partial_t C_\ell)\dot{\bm{\phi}}_\ell,\dot{\bm{\phi}}_\ell)_{L^2} \biggr) \\
 + 2\underline{\rho}_\ell([\partial_t, \theta_\ell\bm{l}_\ell^\mathrm{T}(\bm{v}\cdot\nabla)]\dot{\bm{\phi}}_\ell, \dot{\zeta})_{L^2}
 \biggr\}.
\end{align*}
Here, as in the proof of Lemma~\ref{L.A0-coercive} we have 
$\sqrt{\underline{\rho}_\ell/\underline{h}_\ell}\theta_\ell(|\bm{v}|+|\partial_t \bm{v}|) \lesssim 1$ for $\ell=1,2$. 
In view of the relations $\partial_t\theta_1=-\partial_t\theta_2=\theta_1\theta_2(H_1^{-1}\partial_t H_1 - H_2^{-1}\partial_t H_2)$, 
we have $|\partial_t\theta_\ell| \lesssim \theta_1\theta_2$ for $\ell=1,2$. 
Therefore, we obtain $|I_1| \lesssim \mathscr{E}(\dot{\bm{U}})$. 
As for $I_2$, by integration by parts we have 
\begin{align*}
I_2 &= ((\nabla\cdot(a\bm{u}))\dot{\zeta},\dot{\zeta})_{L^2} \\
&\quad\;
 -\sum_{\ell=1,2}\underline{\rho}_\ell\underline{h}_\ell \biggl\{ \sum_{l=1}^n \bigl\{
  2(A_\ell\partial_l\dot{\bm{\phi}}_\ell,((\partial_l\bm{u})\cdot\nabla)\dot{\bm{\phi}}_\ell)_{L^2}
  + (((\bm{u}\cdot\nabla)^*A_\ell)\partial_l\dot{\bm{\phi}}_\ell,\partial_\ell\dot{\bm{\phi}}_\ell)_{L^2} \bigr\} \\
&\makebox[7em]{}
 +(\underline{h}_\ell\delta)^{-2}(((\bm{u}\cdot\nabla)^*C_\ell)\dot{\bm{\phi}}_\ell,\dot{\bm{\phi}}_\ell)_{L^2} \biggr\} \\
&\quad\;
 +2\sum_{\ell=1,2}\underline{\rho}_\ell\bigl\{
  ((\nabla\cdot\bm{u})\dot{\zeta},\theta_\ell\bm{l}_\ell^\mathrm{T}(\bm{v}\cdot\nabla)\dot{\bm{\phi}}_\ell)_{L^2}
  + (\dot{\zeta}, [\bm{u}\cdot\nabla,\theta_\ell\bm{l}_\ell^\mathrm{T}(\bm{v}\cdot\nabla)]\dot{\bm{\phi}}_\ell)_{L^2}.
\end{align*}
By using~\eqref{parameters}, we see that 
\[
\theta_1 \simeq \frac{\underline{\rho}_2\underline{h}_1}{\underline{\rho}_1\underline{h}_2 + \underline{\rho}_2\underline{h}_1}
 = \frac{\underline{\rho}_2}{\underline{h}_2}, \qquad
\theta_2 \simeq \frac{\underline{\rho}_1\underline{h}_2}{\underline{\rho}_1\underline{h}_2 + \underline{\rho}_2\underline{h}_1}
 = \frac{\underline{\rho}_1}{\underline{h}_1}.
\]
Therefore, we have $|\bm{u}| \leq \theta_2|\bm{u}_1|+\theta_1|\bm{u}_2| \lesssim 1$. 
In view of $|\nabla\theta_\ell| \lesssim \theta_1\theta_2$ for $\ell=1,2$, we have also $|\nabla\bm{u}| \lesssim 1$ and 
$\sqrt{\rho_\ell/\underline{h}_\ell}\theta_l|\nabla\bm{v}| \lesssim 1$ for $\ell=1,2$. 
Hence, we obtain $|I_2| \lesssim \mathscr{E}(\dot{\bm{U}})$. 
Finally, as for $I_3$, we have 
\begin{align*}
I_3 &= 2(\partial_t\dot{\zeta},\dot{f}_0)_{L^2} -2 (\dot{\zeta},\nabla\cdot(\bm{u}\dot{f}_0))_{L^2} \\
&\quad\;
 +2\sum_{\ell=1,2}\underline{\rho}_\ell((\partial_t+\bm{u}\cdot\nabla)\dot{\bm{\phi}}_\ell,
  \dot{\bm{f}}_\ell - (\nabla\cdot(\theta_\ell\bm{l}_\ell\otimes\bm{v}))\dot{\zeta})_{L^2} \\
&\lesssim \|\dot{f}_0\|_{H^1}( \|\partial_t\dot{\zeta}\|_{H^{-1}}+\|\dot{\zeta}\|_{L^2} )
 + \sum_{\ell=1,2}\underline{\rho}_\ell( \|\dot{\bm{f}}_\ell\|_{L^2} + \|\dot{\zeta}\|_{L^2})
  \|(\partial_t\dot{\bm{\phi}}_\ell,\nabla\dot{\bm{\phi}}_\ell)\|_{L^2}.
\end{align*}
Summarizing the above estimates we obtain 
\begin{align*}
\frac{\mathrm{d}}{\mathrm{d}t}(\mathscr{A}_0^\mathrm{mod}\dot{\bm{U}},\dot{\bm{U}})_{L^2}
&\lesssim \mathscr{E}(\dot{\bm{U}}) + 
 \|\dot{f}_0\|_{H^1}( \|\partial_t\dot{\zeta}\|_{H^{-1}}+\|\dot{\zeta}\|_{L^2} ) \\
&\quad\;
 + \sum_{\ell=1,2}\underline{\rho}_\ell( \|\dot{\bm{f}}_\ell\|_{L^2} + \|\dot{\zeta}\|_{L^2})
  \|(\partial_t\dot{\bm{\phi}}_\ell,\nabla\dot{\bm{\phi}}_\ell)\|_{L^2}.
\end{align*}
This together with Lemma~\ref{L.A0-coercive} and Gronwall's inequality gives the desired estimate. 
\end{proof}

\subsection{Energy estimates}\label{S.energy-estimates}
In this subsection, we will complete the proof of Theorem~\ref{theorem-uniform}. 
The existence and the uniqueness of the solution to the initial value problem for the Kakinuma model 
\eqref{Kakinuma-approx} has already been established in the companion paper~\cite{DucheneIguchi2020}, 
so that it is sufficient to derive the uniform bound~\eqref{uniform-sol} of the solution for some 
time interval $[0,T]$ independent of parameters. 
The following lemma can be shown in the same way as the proof of~\cite[Lemma 4.2]{Iguchi2018-2}.

\begin{lemma}\label{L.EE-low}
Let $c,M$ be positive constants and $m$ an integer such that $m>\frac{n}{2}+1$. 
There exists a positive constant $C$ such that for any  positive parameters $\underline{h}_1, \underline{h}_2, \delta$ 
satisfying $\underline{h}_1\delta, \underline{h}_2\delta \leq 1$, if $\zeta\in H^{m-1}$, $b\in W^{m,\infty}$, 
$H_1=1-\underline{h}_1^{-1}\zeta$, and $H_2=1+\underline{h}_2^{-1}\zeta-\underline{h}_2^{-1}b$ satisfy 
\[
\begin{cases}
 \underline{h}_1^{-1}\|\zeta\|_{H^{m-1}} + \underline{h}_2^{-1}\|\zeta\|_{H^{m-1}}
  + \underline{h}_2^{-1}\|b\|_{W^{m,\infty}} \leq M, \\
 H_1(\bm{x})\geq c, \quad H_2(\bm{x})\geq c \quad\mbox{for}\quad \bm{x}\in\mathbf{R}^n,
\end{cases}
\]
and if $\bm{\varphi}_1$ and $\bm{\varphi}_2$ satisfy 
\[
\begin{cases}
 \mathcal{L}_{1,i}(H_1,\delta,\underline{h}_1)\bm{\varphi}_1=f_{1,i}
  \quad\mbox{for}\quad i=1,2,\ldots,N, \\
 \mathcal{L}_{2,i}(H_2,b,\delta,\underline{h}_2)\bm{\varphi}_2=f_{2,i}
  \quad\mbox{for}\quad i=1,2,\ldots,N^*,
\end{cases}
\]
then for any $k=0,\pm1,\ldots,\pm(m-1)$ we have 
\[
(\underline{h}_\ell\delta)^{-2}\|\bm{\varphi}_\ell'\|_{H^k} \leq
 C(\|\nabla\bm{\varphi}_\ell\|_{H^{k+1}} + \|\bm{\varphi}_\ell'\|_{H^{k+1}} + \|\bm{f}_\ell'\|_{H^k}) \qquad (\ell=1,2). 
\]
\end{lemma}

The next lemma gives an energy estimate of the solution to the Kakinuma model~\eqref{Kakinuma-approx} 
under appropriate assumptions on the solution. 
We recall that the mathematical energy function $E_m(t)$ is defined by~\eqref{MathematicalEnergy}.

\begin{lemma}\label{L.energy-estimate2}
Let $c,M,M_1,\underline{h}_\mathrm{min}$ be positive constants. 
There exist two positive constants $C=C(c,M,\underline{h}_\mathrm{min})$ and $C_1=C_1(c,M,M_1,\underline{h}_\mathrm{min})$ 
such that for any positive parameters $\underline{\rho}_1, \underline{\rho}_2, \underline{h}_1, \underline{h}_2, 
\delta$ satisfying the natural restrictions~\eqref{parameters}, $\underline{h}_1\delta, \underline{h}_2\delta \leq 1$, and the condition 
${\underline{h}_\mathrm{min} \leq \underline{h}_1, \underline{h}_2}$, if a regular solution $(\zeta,\bm{\phi}_1,\bm{\phi}_2)$ 
to the Kakinuma model~\eqref{Kakinuma-approx} with a bottom topography $b$ satisfies~\eqref{conditions-linear1}, 
$\underline{h}_2^{-1}\bigl(\|b\|_{W^{m+1,\infty}} + (\underline{h}_2\delta)\|b\|_{W^{m+2,\infty}} \bigr) \leq M_1$, 
and $E_m(t) \leq M_1$ for some time interval $[0,T]$, 
then we have $E_m(t) \leq C\mathrm{e}^{C_1t}E_m(0)$ for $0\leq t\leq T$. 
\end{lemma}

\begin{proof}
Let $\beta$ be a multi-index such that $1\leq|\beta|\leq m$. 
Applying $\partial^\beta$ to the Kakinuma model~\eqref{Kakinuma-approx}, after a tedious but straightforward calculation, we obtain 
\begin{equation}\label{Kakinuma-principal}
\begin{cases}
 \displaystyle
 {\bm l}_1(H_1)(\partial_t+\bm{u}_1\cdot\nabla)\partial^\beta\zeta
  + \underline{h}_1 L_1^\mathrm{pr}(H_1,\delta,\underline{h}_1)\partial^\beta\bm{\phi}_1 = \bm{f}_{1,\beta}, \\
 \displaystyle
 {\bm l}_2(H_2)(\partial_t+\bm{u}_2\cdot\nabla)\partial^\beta\zeta
  - \underline{h}_2 L_2^\mathrm{pr}(H_2,\delta,\underline{h}_2)\partial^\beta\bm{\phi}_2 = \bm{f}_{2,\beta}, \\
 \underline{\rho}_1{\bm l}_1(H_1)\cdot( \partial_t+\bm{u}_1\cdot\nabla )\partial^\beta\bm{\phi}_1 
  - \underline{\rho}_2{\bm l}_2(H_2)\cdot( \partial_t+\bm{u}_2\cdot\nabla )\partial^\beta\bm{\phi}_2 
  - a\partial^\beta\zeta = f_{0,\beta},
\end{cases}
\end{equation}
where $L_1^\mathrm{pr}$ and $L_2^\mathrm{pr}$ are defined by~\eqref{del-A-principal}, the function $a$ by~\eqref{def-a}, and 
\begin{align}
\bm{f}_{1,\beta} \label{f1beta}
&:= -[\partial^\beta,\bm{l}_1(H_1)]\partial_t\zeta
  + \underline{h}_1\bigl\{ [\partial^\beta,A_1(H_1)]\Delta\bm{\phi}_1
  - (\bm{l}_1(H_1)\otimes \bm{l}_1(H_1))(\nabla H_1\cdot\nabla)\partial^\beta\bm{\phi}_1 \\
&\qquad
  + [\partial^\beta,\bm{l}_1(H_1)\otimes\bm{u}_1]\nabla H_1
  - (\underline{h}_1\delta)^{-2}[\partial^\beta,C_1(H_1)]\bm{\phi}_1 \bigr\}, \nonumber \\
\bm{f}_{2,\beta} \label{f2beta}
&:= -[\partial^\beta,\bm{l}_2(H_2)]\partial_t\zeta
  - \underline{h}_2\bigl\{ [\partial^\beta,A_2(H_2)]\Delta\bm{\phi}_2
  - (\bm{l}_2(H_2)\otimes\bm{l}_2(H_2))(\nabla H_2\cdot\nabla)\partial^\beta\bm{\phi}_2  \\
&\qquad
  + [\partial^\beta,\bm{l}_2(H_2)\otimes\bm{u}_2]\nabla H_2 
  - (\underline{h}_2\delta)^{-2}[\partial^\beta,C_2(H_2)]\bm{\phi}_2  \nonumber \\
&\qquad
  - \bm{l}_2(H_2)(\bm{u}_2\cdot\partial^\beta(\underline{h}_2^{-1}\nabla b))
  - \partial^\beta\bigl( \tilde{B}_2(H_2) (\underline{h}_2^{-1}\nabla b\cdot\nabla)\bm{\phi}_2
   + \tilde{C}_2(H_2,\underline{h}_2^{-1}b)\bm{\phi}_2 \bigr) \bigr\},  \nonumber \\
f_{0,\beta} \label{f0beta}
&:= -\underline{\rho}_1\bigl\{
 \bigl( [\partial^\beta,\bm{l}_1(H_1)] - \bm{l}_1'(H_1)(\partial^\beta H_1) \bigr)^\mathrm{T}
  \partial_t\bm{\phi}_1 \\
&\qquad
  + \tfrac12[\partial^\beta;\bm{u}_1,\bm{u}_1] + \tfrac12(\underline{h}_1\delta)^{-2}[\partial^\beta;w_1,w_1]  \nonumber \\
&\qquad
 + \bm{u}_1\cdot\bigl( \bigl( [\partial^\beta,\bm{l}_1(H_1)] - \bm{l}_1'(H_1)(\partial^\beta H_1)
  \bigr) \otimes \nabla \bigr)^\mathrm{T}\bm{\phi}_1 \nonumber \\
&\qquad
 - (\underline{h}_1\delta)^{-2}w_1\bigl( \bigl( 
  [\partial^\beta,\bm{l}_1'(H_1)] - \bm{l}_1''(H_1)(\partial^\beta H_1) \bigr)^\mathrm{T}\bm{\phi}_1 
   +\bm{l}_1'(H_1)\cdot \partial^\beta \bm{\phi}_1 \bigr) \bigr\}
   \nonumber \\
&\quad\;
 + \underline{\rho}_2\bigl\{
 \bigl( [\partial^\beta,\bm{l}_2(H_2)]-\bm{l}_2'(H_2)(\partial^\beta H_2)
  - \bm{l}_2'(H_2)(\partial^\beta(\underline{h}_2^{-1}b))  \bigr)^\mathrm{T}
  \partial_t\bm{\phi}_2 \nonumber \\
&\qquad
  + \tfrac12[\partial^\beta;\bm{u}_2,\bm{u}_2] + \tfrac12(\underline{h}_2\delta)^{-2}[\partial^\beta;w_2,w_2] \nonumber \\
&\qquad
 + \bm{u}_2\cdot\bigl( \bigl( [\partial^\beta,\bm{l}_2(H_2)]-\bm{l}_2'(H_2)(\partial^\beta H_2)
  - \bm{l}_2'(H_2)(\partial^\beta(\underline{h}_2^{-1}b))
  \bigr) \otimes \nabla \bigr)^\mathrm{T}\bm{\phi}_2 \nonumber \\
&\qquad
 - \bm{u}_2\cdot[\partial^\beta,\underline{h}_2^{-1}\nabla b\otimes\bm{\phi}_2]\bm{l}_2'(H_2) \nonumber \\
&\qquad
 - (\bm{u}_2\cdot\underline{h}_2^{-1}\nabla b)\bm{\phi}_2\cdot
  \bigl( \partial^\beta\bm{l}_2'(H_2)-\bm{l}_2''(H_2)(\partial^\beta H_2)
 - \bm{l}_2''(H_2)(\partial^\beta(\underline{h}_2^{-1}b)) \bigr) \nonumber \\
&\qquad
 + (\underline{h}_2\delta)^{-2}w_2\bigl(\bigl( 
  [\partial^\beta,\bm{l}_2'(H_2)]-\bm{l}_2''(H_2)(\partial^\beta H_2)
   - \bm{l}_2''(H_2)(\partial^\beta(\underline{h}_2^{-1}b)) \bigr)^\mathrm{T}\bm{\phi}_2 \nonumber \\
&\qquad\qquad
   + \bm{l}_2'(H_2)\cdot \partial^\beta \bm{\phi}_2 \bigr) \bigr\}.
  \nonumber 
\end{align}
Here, $[\partial^\beta;u,v] = \partial^\beta(uv)-(\partial^\beta u)v-u(\partial^\beta v)$ is the symmetric commutator. 
For vector valued functions, it is defined by $[\partial^\beta;\bm{u},\bm{v}] = \partial^\beta(\bm{u}\cdot\bm{v})
 - (\partial^\beta \bm{u})\cdot\bm{v} - \bm{u}\cdot(\partial^\beta\bm{v})$.

On the other hand, by Lemma~\ref{L.time-derivatives-and-elliptic} we have the estimate~\eqref{estimate-time-derivatives} 
for time derivatives of the solution. 
Particularly, we have 
\begin{equation}\label{estimate-suppl}
\sum_{\ell=1,2}\underline{\rho}_\ell\underline{h}_\ell \bigl( 
 \|\partial_t\bm{u}_\ell\|_{H^{m-1}}^2 + (\underline{h}_\ell\delta)^{-2}\|\partial_t w_\ell\|_{H^{m-1}}^2
 + \|\partial_t\bm{\phi}_\ell'\|_{H^m}^2 + \|\partial_t^2\bm{\phi}_\ell'\|_{H^{m-1}}^2 \bigr) \lesssim E_m. 
\end{equation}
Note that we have also the estimate~\eqref{estimates-uw} for the velocities $(\bm{u}_\ell,w_\ell)$ $(\ell=1,2)$. 
Moreover, it follows from Lemma~\ref{L.EE-low} that 
$\underline{\rho}_\ell\underline{h}_\ell (\underline{h}_\ell\delta)^{-4}\|\bm{\phi}_\ell'\|_{H^{m-1}}^2 \lesssim E_m$ 
for $\ell=1,2$. 
In view of the definition~\eqref{def-a} of the function $a$, it is not difficult to check the estimate 
$\|a-1\|_{H^m}^2 + \|\partial_t a\|_{H^{m-1}}^2 \lesssim E_m$. 
Therefore, by the Sobolev imbedding theorem we see that all the assumptions in Proposition~\ref{L.energy-estimate} are satisfied, 
so that for the solution $\bm{U}=(\zeta,\bm{\phi}_1,\bm{\phi}_2)^\mathrm{T}$ we have 
\[
\mathscr{E}(\partial^\beta{\bm{U}}(t)) \leq C\mathrm{e}^{C_1t}\mathscr{E}(\partial^\beta{\bm{U}}(0))
 + C_1\int_0^t\mathrm{e}^{C_1(t-\tau)}\mathscr{F}_\beta(\tau) \mathrm{d}\tau,
\]
where 
\begin{align*}
\mathscr{F}_\beta
&=  \|f_{0,\beta}\|_{H^1}( \|\partial_t\partial^\beta \zeta\|_{H^{-1}} + \|\partial^\beta \zeta\|_{L^2} ) \\
&\quad\;
 + \sum_{\ell=1,2}\underline{\rho}_\ell( \|\bm{f}_{\ell,\beta}\|_{L^2} + \|\partial^\beta \zeta\|_{L^2})
  \|(\partial_t\partial^\beta \bm{\phi}_\ell,\nabla\partial^\beta \bm{\phi}_\ell)\|_{L^2}.
\end{align*}
In view of the estimates~\eqref{estimate-time-derivatives},~\eqref{estimates-uw}, and~\eqref{estimate-suppl} 
together with 
\[
\| ([\partial^\beta,\bm{l}_\ell(H_\ell)]-\bm{l}_\ell'(H_1)(\partial^\beta H_\ell) )^\mathrm{T}
  \bm{ \varphi}_\ell\|_{H^1} \lesssim \|\bm{\varphi}_\ell'\|_{H^m}
\]
for $\ell=1,2$, we obtain $\mathscr{F}_\beta \lesssim E_m$. 
We note that the multi-index $\beta$ is assumed to satisfy ${1\leq |\beta|\leq m}$. 
As for the case $\beta=0$, in view of $\frac{\mathrm{d}}{\mathrm{d}t}\mathscr{E}(\bm{U}(t)) \lesssim E_m(t)$ 
we infer the inequality $\mathscr{E}({\bm{U}}(t)) \leq \mathscr{E}({\bm{U}}(0)) + C_1\int_0^t E_m(\tau) \mathrm{d}\tau$. 
Summarizing the above estimates we obtain 
\[
E_m(t) \leq C\mathrm{e}^{C_1t}E_m(0)
 + C_1\int_0^t\mathrm{e}^{C_1(t-\tau)}E_m(\tau) \mathrm{d}\tau
\]
with constants $C=C(c,M,\underline{h}_\mathrm{min})$ and $C_1=C_1(c,M,M_1,\underline{h}_\mathrm{min})$. 
Therefore, Gronwall's inequality gives the desired estimate. 
\end{proof}

Now, we are ready to prove Theorem~\ref{theorem-uniform}. 
Suppose that the initial data $(\zeta_{(0)},\bm{\phi}_{1(0)},\bm{\phi}_{2(0)})$ and the bottom topography $b$ 
satisfy~\eqref{uniform-ini}--\eqref{Compatibility}. 
Let $C_0$ be a positive constant such that 
\[
\sum_{\ell=1,2}( \|H_{\ell(0)}\|_{L^\infty} + \underline{\rho}_\ell\underline{h}_\ell
  \|\bm{u}_{\ell(0)}\|_{L^\infty}^2 ) + \|a_{(0)}\|_{L^\infty} \leq C_0. 
\]
Such a constant $C_0$ exists as a constant depending on $c_0, M_0, \underline{h}_\mathrm{min}$, and $m$. 
We will show that the solution $(\zeta,\bm{\phi}_1,\bm{\phi}_2)$ satisfies~\eqref{uniform-sol}, 
\eqref{uniform-below}, and 
\begin{equation}\label{uniform-above}
\sum_{\ell=1,2}( \|H_{\ell}(t)\|_{L^\infty} + \underline{\rho}_\ell\underline{h}_\ell
  \|\bm{u}_{\ell}(t)\|_{L^\infty}^2 ) + \|a(t)\|_{L^\infty} \leq 2C_0
\end{equation}
for $0\leq t\leq T$ with a constant $M$ and a time $T$ which will be determined below. 
We note that~\eqref{uniform-sol} is equivalent to $E_m(t) \leq M$. 
To this end, we assume that the solution satisfies~\eqref{uniform-sol},~\eqref{uniform-below}, 
and~\eqref{uniform-above} for $0\leq t\leq T$. 
In the following, the constant depending on $c_0,C_0,\underline{h}_\mathrm{min},m$ but not on $M$ 
is denoted by $C$ and the constant depending also on $M$ by $C_1$. 
These constants may change from line to line. 
Then, it follows from Lemma~\ref{L.energy-estimate2} that $E_m(t)\leq C\mbox{e}^{C_1t}M_0$ for $0\leq t\leq T$. 
Therefore, if we chose $M=2CM_0$ and if $T$ is so small that $T \leq C_1^{-1}\log 2$, 
then~\eqref{uniform-sol} holds in fact for $0\leq t\leq T$. 
It remains to show~\eqref{uniform-below} and~\eqref{uniform-above}. 
As before, we can check 
\[
\begin{cases}
\sum_{\ell=1,2}\bigl( \|\partial_t H_\ell(t)\|_{L^\infty} + \sqrt{ \underline{\rho}_\ell\underline{h}_\ell }
  \|\partial_t\bm{u}_\ell(t)\|_{L^\infty} \bigr) + \|\partial_t a(t)\|_{L^\infty} \leq C_1, \\
\displaystyle
\|\partial_t \biggl( a(t)- \frac{\underline{\rho}_1\underline{\rho}_2}{
  \underline{\rho}_1 H_{2}(t)\alpha_2 + \underline{\rho}_2 H_1(t)\alpha_1 } |{\bm u}_1(t)-{\bm u}_2(t)|^2\biggr) \|_{L^\infty}
  \leq C_1.
\end{cases}
\]
Therefore, if $T$ is so small that $T\leq (2C_1)^{-1}c_0$ and $T\leq ((2C_0^{1/2}+1)C_1)^{-1}C_0$, 
then the lower bound~\eqref{uniform-below} and the upper bound~\eqref{uniform-above} hold in fact for $0\leq t\leq T$. 
This completes the proof of Theorem~\ref{theorem-uniform}.

\section{Approximation of solutions; proof of Theorem~\ref{theorem-justification}}\label{S.convergence}
In this section we prove Theorem~\ref{theorem-justification}, which gives a rigorous justification of the Kakinuma model~\eqref{Kakinuma-dimensionless}
as a higher order shallow water approximation to the full model for interfacial gravity waves~\eqref{full-model-evolution}
under the hypothesis of the existence of the solution to the full model with uniform bounds.

\subsection{Supplementary estimate for the Dirichlet-to-Neumann map}
In this subsection, we give a supplementary estimate to Lemma~\ref{L.L-estimate-D2N} for the Dirichlet-to-Neumann map 
$\Lambda(\zeta,b,\delta)$ defined by~\eqref{def-DNmap} appearing in the framework of surface waves. 
We recall the map $\Lambda^{(N)}(\zeta,b,\delta) \colon \phi \mapsto \mathcal{L}_0(H,b,\delta)\bm{\phi}$, 
where $\mathcal{L}_0(H,b,\delta)$ is defined by~\eqref{def-L} and $\bm{\phi}$ is the unique solution to~\eqref{condition}.
In this section we omit the dependence of $t$ in notations.

\begin{lemma}\label{L.L-estimate-D2N-2}
Let $c, M$ be positive constants and $m, j$ integers such that $m>\frac{n}{2}+1$, $m\geq 2(j+1)$, and $1\leq j\leq 2N+1$. 
We assume {\rm (H1)} or {\rm (H2)}. 
There exists a positive constant $C$ such that if $\zeta\in H^m$, $b\in W^{m+1,\infty}$, and $H=1+\zeta-b$ satisfy~\eqref{eq-Hyp}, 
then for any $\phi\in\mathring{H}^{k+2(j+1)}$ with $0\leq k\leq m-2(j+1)$ and any $\delta\in(0,1]$ we have 
\[
\|(-\Delta)^{-\frac12}(\Lambda^{(N)}(\zeta,b,\delta)\phi-\Lambda(\zeta,b,\delta)\phi)\|_{H^k}
 \leq C \delta^{2j}\|\nabla \phi\|_{H^{k+2j+1}}.
\]
\end{lemma}

\begin{proof}
This lemma can be proved in a similar way to the proof of Lemma~\ref{L.L-estimate-D2N} with a slight modification. 
For the completeness, we sketch the proof. 
By the duality $(H^k)^*=H^{-k}$ and the symmetry of the operator $(-\Delta)^{-\frac12}$, it is sufficient to show the estimate 
\[
|((\Lambda-\Lambda^{(N)})\phi,\psi)_{L^2}|
\lesssim \delta^{2j}\|\nabla \phi\|_{H^{k+2j+1}} \|\nabla\psi\|_{H^{-k}}
\]
for any $\phi\in\mathring{H}^{k+2(j+1)}$ and any $\psi\in H^{1-k}$. 
We decompose it as 
\begin{align*}
((\Lambda-\Lambda^{(N)})\phi,\psi)_{L^2}
&= ((\Lambda-\Lambda^{(2N+2)})\phi,\psi)_{L^2} + ((\Lambda^{(2N+2)}-\Lambda^{(N)})\phi,\psi)_{L^2} \\
&=: I_1+I_2
\end{align*}
and evaluate the two components of the right-hand side separately.

We recall the definitions~\eqref{def-l} of the $(N^*+1)$ vector-valued function $\bm{l}(H)$ and~\eqref{def-L} of 
the operator $\mathcal{L}_i(H,b,\delta)$, which acts on $(N^*+1)$ vector-valued functions. 
These depend on $N$, so that we denote them by $\bm{l}^{(N)}(H)$ and $\mathcal{L}_i^{(N)}(H,b,\delta)$, 
respectively, in the following argument. 
Let $\Phi$ be the solution to the boundary value problem~\eqref{BVP} and let 
$\bm{\phi}=(\phi_0,\phi_1,\ldots,\phi_{N^*})$, $\tilde{\bm{\phi}}=(\tilde{\phi}_0,\tilde{\phi}_1,\ldots,\tilde{\phi}_{2N^*+2})$, 
and $\bm{\psi}=(\psi_0,\psi_1,\ldots,\psi_{2N^*+2})$ 
be the solutions to the problems 
\[
\begin{cases}
 \mathcal{L}_i^{(N)}(H,b,\delta)\bm{\phi} = 0 \quad \mbox{for}\quad i=1,2,\ldots,N^*, \\
 {\bm l}^{(N)}(H) \cdot \bm{\phi} = \phi,
\end{cases}
\]
\[
\begin{cases}
 \mathcal{L}_i^{(2N+2)}(H,b,\delta)\tilde{\bm{\phi}} = 0 \quad \mbox{for}\quad i=1,2,\ldots,2N^*+2, \\
 {\bm l}^{(2N+2)}(H) \cdot \tilde{\bm{\phi}} = \phi,
\end{cases}
\]
and 
\[
\begin{cases}
 \mathcal{L}_i^{(2N+2)}(H,b,\delta){\bm{\psi}} = 0 \quad \mbox{for}\quad i=1,2,\ldots,2N^*+2, \\
 {\bm l}^{(2N+2)}(H) \cdot {\bm{\psi}} = \psi,
\end{cases}
\]
respectively. 
Put 
\begin{equation}\label{def-PhiPsi}
\begin{cases}
 \displaystyle
  \tilde{\Phi}^\mathrm{app}(\bm{x},z) := \sum_{i=0}^{2N^*+2}(z+1-b(x))^{p_i}\tilde{\phi}_i(\bm{x}), \\
 \displaystyle
  \Psi(\bm{x},z) := \sum_{i=0}^{2N^*+2}(z+1-b(x))^{p_i}\psi_i(\bm{x}),
\end{cases}
\end{equation}
and $\Phi^\mathrm{res}:=\Phi-\tilde{\Phi}^\mathrm{app}$. 
We note that $\tilde{\Phi}^\mathrm{app}$ is a higher order approximation of the velocity potential $\Phi$ 
and that it satisfies the boundary value problem~\eqref{BVP} approximately in the sense that 
\[
\begin{cases}
\Delta\tilde{\Phi}^\mathrm{app} + \delta^{-2}\partial_z^2\tilde{\Phi}^\mathrm{app} = R & \mbox{in}\quad -1+b(\bm{x})<z<\zeta(\bm{x}), \\
\tilde{\Phi}^\mathrm{app}=\phi& \mbox{on}\quad z=\zeta(\bm{x}), \\
\nabla b\cdot\nabla\tilde{\Phi}^\mathrm{app} - \delta^{-2}\partial_z\tilde{\Phi}^\mathrm{app} = r_B & \mbox{on}\quad z=-1+b(\bm{x}),
\end{cases}
\]
where the residual $R$ can be written in the form 
\[
R(\bm{x},z) = \sum_{i=0}^{2N^*+2}(z+1-b(\bm{x}))^{p_i}r_i(\bm{x}).
\]
Estimates for the residuals $(r_0,r_1,\ldots,r_{2N^*+2})$ and $r_B$ were given in~\cite[Lemmas 6.4 and 6.9]{Iguchi2018-2}. 
In fact, we have $\|(r_0,r_1,\ldots,r_{2N^*+2})\|_{H^k}+\|r_B\|_{H^k} \lesssim \delta^{2j}\|\nabla\phi\|_{H^{k+2j+1}}$ 
for $-m\leq k\leq m-2(j+1)$ and $0\leq j\leq 2N+1$.

Now, with a slight modification from the strategy in~\cite{Iguchi2018-2}, we use the identity 
\[
I_1 = \int_\Omega I_\delta\nabla_X\Phi^\mathrm{res} \cdot I_\delta\nabla_X\Psi \mathrm{d}X,
\]
where we denote $\Omega:=\{X=(\bm{x},z) \,;\, -1+b(\bm{x})<z<\zeta(\bm{x})\}$, 
$I_\delta := \operatorname{diag}(1,\ldots,1,\delta^{-1})$, and 
$\nabla_X:=(\nabla,\partial_z)=(\partial_1,\ldots,\partial_n,\partial_z)$.
Indeed, we have on one hand 
\[
(\Lambda \phi,\psi)_{L^2} =\int_\Omega I_\delta\nabla_X \Phi\cdot I_\delta\nabla_X\Psi \mathrm{d}X
\]
as a consequence of~\eqref{BVP}, $\Psi(\bm{x},\zeta(\bm{x}))=\psi(\bm{x})$, and Green's identity, 
and on the other hand 
\begin{align*}
(\Lambda^{(2N+2)}\phi,\psi)_{L^2} 
& =  (\mathcal{L}_0^{(2N+2)}\tilde{\bm{\phi}}, {\bm l}^{(2N+2)} \cdot {\bm{\psi}} )_{L^2} 
	= \sum_{i=0}^{2N^*+2} (H^{p_i}\mathcal{L}_0^{(2N+2)}\tilde{\bm{\phi}},  \psi_i )_{L^2} \\
& = \sum_{i,j=0}^{2N^*+2} (L_{ij}\tilde{\phi}_j,  \psi_i )_{L^2} 
	= \int_\Omega I_\delta\nabla_X\tilde{\Phi}^\mathrm{app} \cdot I_\delta\nabla_X\Psi \mathrm{d}X,
\end{align*}
where the last identity follows from the expressions~\eqref{def-Lij} and~\eqref{def-PhiPsi}.

To evaluate $I_1$, it is convenient to transform the water region $\Omega$ into a simple flat domain 
$\Omega_0=\mathbf{R}^n\times(-1,0)$ by using a diffeomorphism which simply stretches the vertical direction 
$\Theta(\bm{x},z)=(\bm{x},\theta(\bm{x},z)) \colon \Omega_0\to\Omega$, 
where $\theta(\bm{x},z)=\zeta(\bm{x})(z+1)+(1-b(\bm{x}))z$. 
Put $\tilde{\Phi}^\mathrm{res}=\Phi^\mathrm{res}\circ\Theta$ and $\tilde{\Psi}=\Psi\circ\Theta$. 
Then, the above integral is transformed into 
\[
I_1 = \int_{\Omega_0} \mathcal{P}I_\delta\nabla_X\tilde{\Phi}^\mathrm{res} \cdot I_\delta\nabla_X\tilde{\Psi} \mathrm{d}X,
\]
where 
\[
\mathcal{P}
= \det\biggl( \frac{\partial\Theta}{\partial X} \biggr)I_\delta^{-1}\biggl( \frac{\partial\Theta}{\partial X} \biggr)^{-1}
 I_\delta^2\biggl( \biggl( \frac{\partial\Theta}{\partial X} \biggr)^{-1} \biggr)^\mathrm{T}I_\delta^{-1}.
\]
Therefore, under the restriction $|k|\leq m-1$ and using the hypothesis~\eqref{eq-Hyp}, we have 
\[
|I_1| \lesssim \|J^k I_\delta\nabla_X\tilde{\Phi}^\mathrm{res}\|_{L^2(\Omega_0)}
 \|J^{-k} I_\delta\nabla_X\tilde{\Psi}\|_{L^2(\Omega_0)},
\]
where $J=(1-\Delta)^\frac12$. 
Moreover, $\tilde{\Phi}^\mathrm{res}$ satisfies the boundary value problem 
\[
\begin{cases}
\nabla_X\cdot I_\delta\mathcal{P}I_\delta\nabla_X\tilde{\Phi}^\mathrm{res} = -\tilde{R} & \mbox{in}\quad \Omega_0, \\
\tilde{\Phi}^\mathrm{res} = 0 & \mbox{on}\quad z=0, \\
\bm{e}_z\cdot I_\delta\mathcal{P}I_\delta\nabla_X\tilde{\Phi}^\mathrm{res} = -r_B & \mbox{on}\quad z=-1,
\end{cases}
\]
where $\tilde{R} = R\circ\Theta = \sum_{i=0}^{2N^*+2}(z+1)^{p_i}H^{p_i}r_j$ and $\bm{e}_z=(0,\ldots,0,1)^\mathrm{T}$. 
By applying the standard theory of elliptic partial differential equations to the above problem, 
for $0\leq k\leq m-1$ we have 
\begin{align*}
\|J^k I_\delta\nabla_X\tilde{\Phi}^\mathrm{res}\|_{L^2(\Omega_0)}
&\lesssim \delta (\|J^k\tilde{R}\|_{L^2(\Omega_0)}+\|r_B\|_{H^k}) \\
&\lesssim \delta (\|(r_0,r_1,\ldots,r_{2N^*+2})\|_{H^k}+\|r_B\|_{H^k}).
\end{align*}
Moreover, in view of $\tilde{\Psi}=\sum_{i=0}^{2N^*+2}(z+1)^{p_i}H^{p_i}\psi_j$ and by Lemma~\ref{L.L-invertible}, we have 
\begin{align*}
\|J^{-k} I_\delta\nabla_X\tilde{\Psi}\|_{L^2(\Omega_0)}
&\lesssim \|\nabla\bm{\psi}\|_{H^{-k}}+\delta^{-1}\|\bm{\psi}'\|_{H^{-k}} \\
&\lesssim \|\nabla\psi\|_{H^{-k}}
\end{align*}
for $|k|\leq m-1$. 
Summarizing the above estimates we have $|I_1| \lesssim \delta^{2j+1}\|\nabla\phi\|_{H^{k+2j+1}}\|\nabla\psi\|_{H^{-k}}$ 
for $0\leq k\leq m-2(j+1)$ and $0\leq j\leq 2N+1$.

As for the term $I_2$, the evaluation is exactly the same as in~\cite{Iguchi2018-2}. 
In fact, the identities 
\begin{align*}
I_2 &= \sum_{i,j=0}^{2N^*+2}(L_{ij}\tilde{\phi}_j, \psi_i)_{L^2} - \sum_{j=0}^{N^*}(L_{0j}\phi_j, \psi)_{L^2} \\
&= \sum_{j=0}^{N^*}\sum_{i=N^*+1}^{2N^*+2}((L_{ij}-H^{p_i}L_{0j})\varphi_j,\psi_i)_{L^2}
 - \sum_{i,j=N^*+1}^{2N^*+2}((L_{ij}-H^{p_i}L_{0j})\tilde{\phi}_j,\psi_i)_{L^2}
\end{align*}
were shown in~\cite[Equation (7.7)]{Iguchi2018-2}, 
where $\bm{\varphi}:=(\varphi_0,\varphi_1,\ldots,\varphi_{N^*})$ was defined by 
$\varphi_i:=\phi_i-\tilde{\phi}_i$ for $i=0,1,\ldots,N^*$. 
Now, we decompose $j=j_1+j_2$ such that $1\leq j_1\leq N+1$ and $0\leq j_2\leq N$. 
Then, by~\cite[Lemmas 5.2, 5.4, 6.2 and 6.7]{Iguchi2018-2} we see that 
\begin{align*}
|I_2| &\lesssim \{
 \|\bm{\varphi}\|_{H^{k+2j_1+1}}+\|(\tilde{\phi}_{N^*+1},\ldots,\tilde{\phi}_{2N^*+2})\|_{H^{k+2j_1+1}} \\
&\quad\;
 +\delta^{-2}(\|\bm{\varphi}\|_{H^{k+2j_1-1}}+\|(\tilde{\phi}_{N^*+1},\ldots,\tilde{\phi}_{2N^*+2})\|_{H^{k+2j_1-1}}) \}
 \|(\psi_{N^*+1},\ldots,\psi_{2N^*+2})\|_{H^{-(k+2j_1-1)}} \\
&\lesssim
 \delta^{2(j_1+j_2)}\|\nabla\phi\|_{H^{k+2(j_1+j_2)}}\|\nabla\psi\|_{H^{-k}}
\end{align*}
if $\max\{|k|,|k+2j_1-2|,|k+2j_1+1|,|k+2(j_1+j_2)|\} \leq m-1$ and $\max\{|k|,|k+1|,|k+2j_1-1|\}\leq m$. 
These conditions are satisfied under the restriction $-m+1\leq k\leq m-2(j+1)$.

To summarize, we obtain as desired
$|((\Lambda-\Lambda^{(N)})\phi,\psi)_{L^2}| \lesssim \delta^{2j}\|\nabla \phi\|_{H^{k+2j+1}} \|\nabla\psi\|_{H^{-k}}$ 
for $0\leq k\leq m-2(j+1)$ and $1\leq j\leq 2N+1$. 
The proof is complete. 
\end{proof}

This lemma and the scaling relations~\eqref{relations-DN-B} imply immediately the following lemma.

\begin{lemma}\label{L.L-estimate-D2N-3}
Let $c, M$ be positive constants and $m,j$ integers such that $m>\frac{n}{2}+1$, $m\geq 2(j+1)$, and $1\leq j\leq 2N+1$. 
We assume {\rm (H1)} or {\rm (H2)}. 
There exists a positive constant $C$ such that for any positive parameters 
$\underline{h}_1, \underline{h}_2, \delta$ satisfying $\underline{h}_1\delta, \underline{h}_2\delta \leq 1$, 
if $\zeta\in H^m$, $b\in W^{m+1,\infty}$, $H_1 = 1 - \underline{h}_1^{-1}\zeta$, 
and $H_2 = 1 + \underline{h}_2^{-1}\zeta - \underline{h}_2^{-1}b$ satisfy~\eqref{eq-Hyp2}, 
then for any $\phi_1,\phi_2 \in \mathring{H}^{k+2(j+1)}$ with $0\leq k\leq m-2(j+1)$ we have 
\[
\begin{cases}
 \|(-\Delta)^{-\frac12}(\underline{h}_1\Lambda_1^{(N)}(\zeta,\delta,\underline{h}_1)\phi_1
  - \Lambda_1(\zeta,\delta,\underline{h}_1)\phi_1)\|_{H^k}
  \leq C \underline{h}_1(\underline{h}_1\delta)^{2j}\|\nabla \phi_1\|_{H^{k+2j+1}}, \\
 \|(-\Delta)^{-\frac12}(\underline{h}_2\Lambda_2^{(N)}(\zeta,b,\delta,\underline{h}_2)\phi_2
  - \Lambda_2(\zeta,b,\delta,\underline{h}_2)\phi_2)\|_{H^k}
  \leq C \underline{h}_2(\underline{h}_2\delta)^{2j}\|\nabla \phi_2\|_{H^{k+2j+1}}.
\end{cases}
\]
\end{lemma}

We recall also the estimate for the Dirichlet-to-Neumann map $\Lambda(\zeta,b,\delta)$ itself. 
The following lemma is now standard. 
For sharper estimates, we refer to T. Iguchi~\cite{Iguchi2009} and D. Lannes~\cite{Lannes2013-2}.

\begin{lemma}\label{L.estimate-DN}
Let $c, M$ be positive constants $m$ an integer such that $m>\frac{n}{2}+2$. 
There exists a positive constant $C$ such that if $\zeta\in H^m$, $b\in W^{m,\infty}$, and $H=1+\zeta-b$ satisfy~\eqref{eq-Hyp0}, 
then for any $\phi\in\mathring{H}^{k+1}$ with $|k|\leq m-1$ and any $\delta\in(0,1]$ we have 
$\|\Lambda(\zeta,b,\delta)\phi\|_{H^{k-1}} \leq C \|\nabla \phi\|_{H^k}$. 
\end{lemma}

This lemma and the scaling relations~\eqref{relations-DN-B} imply immediately the following lemma.

\begin{lemma}\label{L.estimate-DN2}
Let $c, M$ be positive constants and $m$ an integer such that $m>\frac{n}{2}+2$. 
There exists a positive constant $C$ such that for any positive parameters 
$\underline{h}_1, \underline{h}_2, \delta$ satisfying $\underline{h}_1\delta, \underline{h}_2\delta \leq 1$, 
if $\zeta\in H^m$, $b\in W^{m,\infty}$, $H_1 = 1 - \underline{h}_1^{-1}\zeta$, 
and $H_2 = 1 + \underline{h}_2^{-1}\zeta - \underline{h}_2^{-1}b$ satisfy~\eqref{eq-Hyp20}, 
then for any $\phi_1,\phi_2 \in \mathring{H}^{k+1}$ with $|k|\leq m-1$ we have 
\[
\begin{cases}
 \|\Lambda_1(\zeta,\delta,\underline{h}_1)\phi_1\|_{H^{k-1}} \leq C\underline{h}_1\|\nabla\phi_1\|_{H^k}, \\
 \|\Lambda_2(\zeta,b,\delta,\underline{h}_2)\phi_2\|_{H^{k-1}} \leq C\underline{h}_2\|\nabla\phi_2\|_{H^k}. 
\end{cases} 
\]
\end{lemma}
%
%
\subsection{Consistency of the Kakinuma model revisited}
As we mentioned in Remark~\ref{R.app-sol}, the approximate solution to the Kakinuma model~\eqref{Kakinuma-dimensionless}
made from the solution $(\zeta,\phi_1,\phi_2)$ to the full model can be constructed as a solution to~\eqref{Conditions-v2}, that is, 
\begin{equation}\label{Conditions-v2bis}
\begin{cases}
\mathcal{L}_{1,i}(H_1,\delta,\underline{h}_1)\tilde{\bm{\phi}}_1=0 \quad\mbox{for}\quad i=1,2,\ldots,N, \\
\mathcal{L}_{2,i}(H_2,b,\delta,\underline{h}_2)\tilde{\bm{\phi}}_2=0 \quad\mbox{for}\quad i=1,2,\ldots,N^*, \\
\underline{h}_1\mathcal{L}_{1,0}(H_1,\delta,\underline{h}_1) \tilde{\bm{\phi}}_1
+ \underline{h}_2\mathcal{L}_{2,0}(H_2,b,\delta,\underline{h}_2) \tilde{\bm{\phi}}_2 = 0, \\
\underline{\rho}_2\bm{l}_2(H_2) \cdot \tilde{\bm{\phi}}_2
- \underline{\rho}_1\bm{l}_1(H_1) \cdot \tilde{\bm{\phi}}_1 = \underline{\rho}_2\phi_2-\underline{\rho}_2\phi_1,
\end{cases}
\end{equation}
in place of~\eqref{Conditions}, that is, 
\begin{equation}\label{Conditionsbis}
 \begin{cases}
  \bm{l}_1(H_1)\cdot\bm{\phi}_1=\phi_1, \quad \mathcal{L}_{1,i}(H_1,\delta,\underline{h}_1)\bm{\phi}_1=0
   \quad\mbox{for}\quad i=1,2,\ldots,N, \\
  \bm{l}_2(H_2)\cdot\bm{\phi}_2=\phi_2, \quad \mathcal{L}_{2,i}(H_2,b,\delta,\underline{h}_2)\bm{\phi}_2=0
   \quad\mbox{for}\quad i=1,2,\ldots,N^*.
 \end{cases}
\end{equation}
To show this fact, we need to guarantee that the difference between these two solutions is of order 
$O((\underline{h}_1\delta)^{4N+2}+(\underline{h}_2\delta)^{4N+2})$. 
The following lemma gives such an estimate.

\begin{lemma}\label{L.estimate-phi-phi}
Let $c, M$ be positive constants and $m$ an integer such that $m>\frac{n}{2}+1$ and ${m\geq4(N+1)}$. 
We assume {\rm (H1)} or {\rm (H2)}. 
There exists a positive constant $C$ such that for any positive parameters 
$\underline{\rho}_1, \underline{\rho}_2, \underline{h}_1, \underline{h}_2, \delta$ 
satisfying $\underline{h}_1\delta, \underline{h}_2\delta \leq 1$, 
if $\zeta\in H^m$, $b\in W^{m+1,\infty}$, $H_1 = 1 - \underline{h}_1^{-1}\zeta$, 
and $H_2 = 1 + \underline{h}_2^{-1}\zeta - \underline{h}_2^{-1}b$ satisfy~\eqref{eq-Hyp2}, 
then for any $\phi_1,\phi_2 \in \mathring{H}^{k+4(N+1)}$ with $0\leq k\leq m-4(N+1)$ satisfying the compatibility condition 
$\Lambda_1(\zeta,\delta,\underline{h}_1)\phi_1+\Lambda_2(\zeta,b,\delta,\underline{h}_2)\phi_2=0$ 
the solution $(\bm{\phi}_1,\bm{\phi}_2)$ to~\eqref{Conditionsbis} and the solution 
$(\tilde{\bm{\phi}}_1,\tilde{\bm{\phi}}_2)$ to~\eqref{Conditions-v2bis} satisfy 
\begin{align*}
\sum_{\ell=1,2}\underline{\rho}_\ell \underline{h}_\ell(
 \|\nabla(\tilde{\bm{\phi}}_\ell-\bm{\phi}_\ell)\|_{H^k}^2
 + (\underline{h}_\ell\delta)^{-2}\|\tilde{\bm{\phi}}_\ell'-\bm{\phi}_\ell'\|_{H^k}^2
 + (\underline{h}_\ell\delta)^{-4}\|\tilde{\bm{\phi}}_\ell'-\bm{\phi}_\ell'\|_{H^{k-1}}^2 ) \\
\leq C\sum_{\ell=1,2}\underline{\rho}_\ell \underline{h}_\ell(\underline{h}_\ell\delta)^{2(4N+2)}\|\nabla\phi_\ell\|_{H^{k+4N+3}}^2.
\end{align*}
\end{lemma}

\begin{proof}
For simplicity, we write $\mathcal{L}_{1,i}=\mathcal{L}_{1,i}(H_1,\delta,\underline{h}_1)$, $\bm{l}_1=\bm{l}_1(H_1)$, and so on. 
We recall that $\Lambda_1^{(N)}\colon \phi_1 \mapsto \mathcal{L}_{1,0}\bm{\phi}_1$ and 
$\Lambda_2^{(N)}\colon \phi_2 \mapsto \mathcal{L}_{2,0}\bm{\phi}_2$. 
Notice that $\tilde{\bm{\phi}}_\ell-\bm{\phi}_\ell$ for $\ell=1,2$ satisfy 
\[
\begin{cases}
 \mathcal{L}_{1,i}(\tilde{\bm{\phi}}_1-\bm{\phi}_1)=0 \quad\mbox{for}\quad i=1,2,\ldots,N, \\
 \mathcal{L}_{2,i}(\tilde{\bm{\phi}}_2-\bm{\phi}_2)=0 \quad\mbox{for}\quad i=1,2,\ldots,N^*, \\
 \underline{h}_1\mathcal{L}_{1,0} (\tilde{\bm{\phi}}_1-\bm{\phi}_1)
  + \underline{h}_2\mathcal{L}_{2,0} (\tilde{\bm{\phi}}_2-\bm{\phi}_2)
  = (\Lambda_1-\underline{h}_1\Lambda_1^{(N)})\phi_1 + (\Lambda_2-\underline{h}_2\Lambda_2^{(N)})\phi_2, \\
 \underline{\rho}_2\bm{l}_2 \cdot (\tilde{\bm{\phi}}_2 - \bm{\phi}_2)
  - \underline{\rho}_1\bm{l}_1 \cdot (\tilde{\bm{\phi}}_1 - \bm{\phi}_1) = 0.
\end{cases}
\]
Since the right-hand side of the third equation can be written as $\nabla\cdot\bm{f}_3$ with 
\[
\bm{f}_3=-\nabla(-\Delta)^{-1}\bigl( (\Lambda_1-\underline{h}_1\Lambda_1^{(N)})\phi_1
 - (\Lambda_2-\underline{h}_2\Lambda_2^{(N)})\phi_2\bigr), 
\]
by Lemmas~\ref{L.elliptic} and~\ref{L.L-estimate-D2N-3} we obtain 
\begin{align*}
&\sum_{\ell=1,2}\underline{\rho}_\ell \underline{h}_\ell(
 \|\nabla(\tilde{\bm{\phi}}_\ell-\bm{\phi}_\ell)\|_{H^k}^2
 + (\underline{h}_\ell\delta)^{-2}\|\tilde{\bm{\phi}}_\ell'-\bm{\phi}_\ell'\|_{H^k}^2 ) \\
&\makebox[10em]{}\lesssim \min\biggl\{\frac{\underline{\rho}_1}{\underline{h}_1},\frac{\underline{\rho}_2}{\underline{h}_2} \biggr\}
  \|\bm{f}_3\|_{H^k}^2 \\
&\makebox[10em]{}\lesssim \sum_{l=1,2}\frac{\underline{\rho}_\ell}{\underline{h}_\ell}
 \|(-\Delta)^{-\frac12} (\Lambda_\ell-\underline{h}_\ell\Lambda_\ell^{(N)})\phi_\ell\|_{H^k}^2 \\
&\makebox[10em]{}\lesssim \sum_{l=1,2} \underline{\rho}_\ell \underline{h}_\ell 
 (\underline{h}_\ell\delta)^{2(4N+2)}\|\nabla\phi_\ell\|_{H^{k+4N+3}}^2. 
\end{align*}
Moreover, it follows from Lemma~\ref{L.EE-low} that 
\[
(\underline{h}_\ell\delta)^{-2}\|\tilde{\bm{\phi}}_\ell'-\bm{\phi}_\ell'\|_{H^{k-1}} 
 \lesssim \|\nabla(\tilde{\bm{\phi}}_\ell-\bm{\phi}_\ell)\|_{H^k}
 + (\underline{h}_\ell\delta)^{-1}\|\tilde{\bm{\phi}}_\ell'-\bm{\phi}_\ell'\|_{H^k}
\]
for $\ell=1,2$. 
This completes the proof. 
\end{proof}

The following proposition gives another version of Theorem~\ref{theorem-consistency2} for the consistency 
of the Kakinuma model.

\begin{proposition}\label{prop-consistency}
Let $c, M$ be positive constants and $m$ an integer such that $m \geq 4N+4$ and $m>\frac{n}{2}+2$. 
We assume {\rm (H1)} or {\rm (H2)}. 
There exists a positive constant $C$ such that for any positive parameters 
$\underline{\rho}_1, \underline{\rho}_2, \underline{h}_1, \underline{h}_2, \delta$ satisfying 
$\underline{h}_1\delta, \underline{h}_2\delta \leq 1$, and for any 
solution $(\zeta,\phi_1,\phi_2)$ to the full model for interfacial gravity waves~\eqref{full-model-evolution} 
on a time interval $[0,T]$ satisfying~\eqref{cond.consisitency}, 
if we define $H_1$ and $H_2$ as in~\eqref{thicknesses} and $(\tilde{\bm{\phi}}_1,\tilde{\bm{\phi}}_2)$ 
as a solution to~\eqref{Conditions-v2bis}, 
then $(\zeta,\tilde{\bm{\phi}}_1,\tilde{\bm{\phi}}_2)$ satisfy approximately the Kakinuma model as 
\begin{equation}\label{app-Kakinuma}
\begin{cases}
 \displaystyle
 {\bm l}_1(H_1)\underline{h}_1^{-1}\partial_t\zeta + L_1(H_1,\delta,\underline{h}_1)\tilde{\bm{\phi}}_1
  = \bm{\mathfrak{r}}_1, \\
 \displaystyle
 {\bm l}_2(H_2)\underline{h}_2^{-1}\partial_t\zeta - L_2(H_2,b,\delta,\underline{h}_2)\tilde{\bm{\phi}}_2
  = \bm{\mathfrak{r}}_2, \\
 \underline{\rho}_1\bigl\{ {\bm l}_1(H_1) \cdot \partial_t\tilde{\bm{\phi}}_1 
   + \frac12\bigl( |\tilde{\bm{u}}_1|^2 + (\underline{h}_1\delta)^{-2} \tilde{w}_1^2 \bigr) \bigr\} \\
 \quad
 - \underline{\rho}_2\bigl\{ {\bm l}_2(H_2) \cdot \partial_t\tilde{\bm{\phi}}_2 
  + \frac12\bigl( |\tilde{\bm{u}}_2|^2 +  (\underline{h}_2\delta)^{-2} \tilde{w}_2^2 \bigr) \bigr\} 
 - \zeta = \mathfrak{r}_0,
\end{cases}
\end{equation}
where $\tilde{\bm{u}}_1, \tilde{\bm{u}}_2, \tilde{w}_1, \tilde{w}_2$ are defined by~\eqref{def-uw} with 
$(\bm{\phi}_1,\bm{\phi}_2)$ replaced by $(\tilde{\bm{\phi}}_1,\tilde{\bm{\phi}}_2)$, 
and the errors $(\bm{\mathfrak{r}}_1,\bm{\mathfrak{r}}_2,\mathfrak{r}_0)$ satisfy 
\begin{equation}\label{error-estimate2}
\begin{cases}
 \displaystyle
 \sum_{\ell=1,2}\underline{\rho}_\ell \underline{h}_\ell \|\bm{\mathfrak{r}}_\ell(t)\|_{H^{m-(4N+5)}}^2
  \leq C\sum_{\ell=1,2}\underline{\rho}_\ell \underline{h}_\ell
   (\underline{h}_\ell\delta)^{2(4N+2)}\|\nabla\phi_\ell(t)\|_{H^{m-1}}^2, \\
 \displaystyle
 \|\mathfrak{r}_0(t)\|_{H^{m-4(N+1)}} \leq 
  C
  \bigl( (\underline{h}_1\delta)^{4N+2}+(\underline{h}_2\delta)^{4N+2}\bigr)
  (\underline{h}_1^{-1}+\underline{h}_2^{-1})
  \sum_{\ell=1,2}\underline{\rho}_\ell \underline{h}_\ell  \|\nabla\phi_\ell(t)\|_{H^{m-1}}^2,
\end{cases}
\end{equation}
for $t\in[0,T]$. 
\end{proposition}

\begin{proof}
Let $\bm{\phi}_1$ and $\bm{\phi}_2$ be the unique solutions to~\eqref{Conditionsbis}, and 
$(\tilde{\bm{\mathfrak{r}}}_1,\tilde{\bm{\mathfrak{r}}}_2,\tilde{\mathfrak{r}}_0)$ the errors in 
Theorem~\ref{theorem-consistency2}. 
Then, the errors $(\bm{\mathfrak{r}}_1,\bm{\mathfrak{r}}_2,\mathfrak{r}_0)$ in the proposition can be written as 
\[
\begin{cases}
 \bm{\mathfrak{r}}_1 = \tilde{\bm{\mathfrak{r}}}_1 - L_1(H_1,\delta,\underline{h}_1)(\tilde{\bm{\phi}}_1-\bm{\phi}_1), \\
 \bm{\mathfrak{r}}_2 = \tilde{\bm{\mathfrak{r}}}_2 + L_2(H_2,b,\delta,\underline{h}_2)(\tilde{\bm{\phi}}_2-\bm{\phi}_2), \\
 \mathfrak{r}_0 = \tilde{\mathfrak{r}}_0 + \underline{\rho}_1\{
  \underline{h}_1^{-1}(\partial_t\zeta)(\tilde{w}_1-w_1)
   - \frac12\bigl( (\tilde{\bm{u}}_1+\bm{u}_1)\cdot(\tilde{\bm{u}}_1-\bm{u}_1)
    + (\underline{h}_1\delta)^{-2}(\tilde{w}_1+w_1)(\tilde{w}_1-w_1) \bigr)\} \\
 \qquad - \underline{\rho}_2\{
  \underline{h}_2^{-1}(\partial_t\zeta)(\tilde{w}_2-w_2)
   - \frac12\bigl( (\tilde{\bm{u}}_2+\bm{u}_2)\cdot(\tilde{\bm{u}}_2-\bm{u}_2)
    + (\underline{h}_2\delta)^{-2}(\tilde{w}_2+w_2)(\tilde{w}_2-w_2) \bigr)\}.
\end{cases}
\]
Therefore, we have 
\[
\|\bm{\mathfrak{r}}_\ell-\tilde{\bm{\mathfrak{r}}}_\ell\|_{H^k}
\lesssim \|\nabla(\tilde{\bm{\phi}}_\ell-\bm{\phi}_\ell)\|_{H^{k+1}} + \|\tilde{\bm{\phi}}_\ell'-\bm{\phi}_\ell'\|_{H^{k+1}}
 + (\underline{h}_\ell\delta)^{-2}\|\tilde{\bm{\phi}}_\ell'-\bm{\phi}_\ell'\|_{H^k}
\]
for $-m\leq k\leq m-1$ and $\ell=1,2$. 
Applying this estimate with $k=m-(4N+5)$ and the estimate in Lemma~\ref{L.estimate-phi-phi} with $k=m-4(N+1)$ 
and using the result in Theorem~\ref{theorem-consistency2}, we obtain the first estimate in~\eqref{error-estimate2}. 
Since $m-2>\frac{n}{2}$, we have 
\begin{align*}
\|\mathfrak{r}_0-\tilde{\mathfrak{r}}_0\|_{H^k}
&\lesssim \sum_{\ell=1,2}\underline{\rho}_\ell\{ 
 (\|\tilde{\bm{u}}_\ell\|_{H^{m-2}}+\|\bm{u}_\ell\|_{H^{m-2}}) \|\tilde{\bm{u}}_\ell-\bm{u}_\ell\|_{H^k} \\
&\quad\;
 + \bigl(\underline{h}_\ell^{-1}\|\partial_t\zeta\|_{H^{m-2}} + 
  (\underline{h}_\ell\delta)^{-2}(\|\tilde{w}_\ell\|_{H^{m-2}}+\|w_\ell\|_{H^{m-2}}) \bigr)\|\tilde{w}_\ell-w_\ell\|_{H^k} \}
\end{align*}
for $|k|\leq m-2$.
Here, it follows from Lemmas~\ref{L.Lambda},~\ref{L.elliptic}, and~\ref{L.estimate-phi-phi} that 
\begin{align*}
\sum_{\ell=1,2} \underline{\rho}_\ell \underline{h}_\ell
 (\|\bm{u}_\ell\|_{H^{m-1}}^2 + (\underline{h}_\ell\delta)^{-2}\|w_\ell\|_{H^{m-1}}^2 )
&\lesssim \sum_{\ell=1,2} \underline{\rho}_\ell \underline{h}_\ell
 (\|\nabla\bm{\phi}_\ell\|_{H^{m-1}}^2 + (\underline{h}_\ell\delta)^{-2}\|\bm{\phi}_\ell'\|_{H^{m-1}}^2) \\
&\lesssim \sum_{\ell=1,2} \underline{\rho}_\ell \underline{h}_\ell \|\nabla\phi_\ell\|_{H^{m-1}}^2,
\end{align*}
\begin{align*}
\sum_{\ell=1,2} \underline{\rho}_\ell \underline{h}_\ell
 (\|\tilde{\bm{u}}_\ell\|_{H^{m-1}}^2 + (\underline{h}_\ell\delta)^{-2}\|\tilde{w}_\ell\|_{H^{m-1}}^2 )
&\lesssim \sum_{\ell=1,2} \underline{\rho}_\ell \underline{h}_\ell
 (\|\nabla\tilde{\bm{\phi}}_\ell\|_{H^{m-1}}^2 + (\underline{h}_\ell\delta)^{-2}\|\tilde{\bm{\phi}}_\ell'\|_{H^{m-1}}^2) \\
&\lesssim \min\biggl\{\frac{\underline{h}_1}{\underline{\rho}_1},\frac{\underline{h}_2}{\underline{\rho}_2} \biggr\}
  \|\nabla(\underline{\rho}_2\phi_2-\underline{\rho}_1\phi_1)\|_{H^{m-1}}^2 \\
&\lesssim \sum_{\ell=1,2} \underline{\rho}_\ell \underline{h}_\ell \|\nabla\phi_\ell\|_{H^{m-1}}^2,
\end{align*}
and
\begin{align*}
&\sum_{\ell=1,2} \underline{\rho}_\ell \underline{h}_\ell
 (\|\tilde{\bm{u}}_\ell-\bm{u}_\ell\|_{H^k}^2 + (\underline{h}_\ell\delta)^{-2}\|\tilde{w}_\ell-w_\ell\|_{H^k}^2 ) \\
&\makebox[5em]{}
\lesssim \sum_{\ell=1,2} \underline{\rho}_\ell \underline{h}_\ell
 (\|\nabla(\tilde{\bm{\phi}}_\ell-\bm{\phi}_\ell)\|_{H^k}^2
  + (\underline{h}_\ell\delta)^{-2}\|\tilde{\bm{\phi}}_\ell'-\bm{\phi}_\ell'\|_{H^k}^2) \\
&\makebox[5em]{}
\lesssim \sum_{\ell=1,2} \underline{\rho}_\ell \underline{h}_\ell (\underline{h}_\ell\delta)^{2(4N+2)}\|\nabla\phi_\ell\|_{H^{k+4N+3}}^2
\end{align*}
for $0\leq k\leq m-4(N+1)$. 
Moreover, it follows from Lemma~\ref{L.estimate-DN2} that 
$\|\partial_t\zeta\|_{H^{m-2}} = \|\Lambda_\ell\phi_\ell\|_{H^{m-2}}
 \lesssim \underline{h}_\ell \|\nabla\phi_\ell\|_{H^{m-1}}$ for $\ell=1,2$. 
Summarizing the above estimates and using the result in Theorem~\ref{theorem-consistency2}, 
we easily obtain the second estimate in~\eqref{error-estimate2}. 
The proof is complete. 
\end{proof}

\subsection{Completion of the proof of Theorem~\ref{theorem-justification}}
Now we are ready to prove Theorem~\ref{theorem-justification}. 
Let $(\zeta^{\mbox{\rm\tiny IW}},\phi_1^{\mbox{\rm\tiny IW}},\phi_2^{\mbox{\rm\tiny IW}})$ be the solution to the full model 
for interfacial gravity waves~\eqref{full-model-evolution} with uniform bound stated in the theorem, and define 
${\phi^{\mbox{\rm\tiny IW}}:=\underline{\rho}_2\phi_2^{\mbox{\rm\tiny IW}}-\underline{\rho}_1\phi_1^{\mbox{\rm\tiny IW}}}$, 
which is a canonical variable of the full model. 
We first ensure a uniform bound on the time derivative of the canonical variables 
$(\zeta^{\mbox{\rm\tiny IW}},\phi^{\mbox{\rm\tiny IW}})$. 
It follows from the first and the second equations in~\eqref{full-model-evolution} that 
$\partial_t\zeta^{\mbox{\rm\tiny IW}}=-\Lambda_1^{\mbox{\rm\tiny IW}}\phi_1^{\mbox{\rm\tiny IW}}
=\Lambda_2^{\mbox{\rm\tiny IW}}\phi_2^{\mbox{\rm\tiny IW}}$, where 
$\Lambda_1^{\mbox{\rm\tiny IW}}=\Lambda_1(\zeta^{\mbox{\rm\tiny IW}},\delta,\underline{h}_1)$ and 
$\Lambda_2^{\mbox{\rm\tiny IW}}=\Lambda_2(\zeta^{\mbox{\rm\tiny IW}},b,\delta,\underline{h}_2)$. 
Similar notations will be used in the following without any comment. 
Therefore, by Lemma~\ref{L.estimate-DN2} we have 
\begin{align*}
\|\partial_t\zeta^{\mbox{\rm\tiny IW}}\|_{H^{m-1}}^2
&= \min\{ \|\Lambda_1^{\mbox{\rm\tiny IW}}\phi_1^{\mbox{\rm\tiny IW}}\|_{H^{m-1}}^2,
 \|\Lambda_2^{\mbox{\rm\tiny IW}}\phi_2^{\mbox{\rm\tiny IW}}\|_{H^{m-1}}^2 \} \\
&\lesssim \min\{ \underline{h}_1^2\|\nabla\phi_1^{\mbox{\rm\tiny IW}}\|_{H^m}^2,
 \underline{h}_2^2\|\nabla\phi_2^{\mbox{\rm\tiny IW}}\|_{H^m}^2 \} \\
&\lesssim \min\biggl\{ \frac{\underline{h}_1}{\underline{\rho}_1}, \frac{\underline{h}_2}{\underline{\rho}_2} \biggr\}
 \sum_{\ell=1,2} \underline{\rho}_\ell \underline{h}_\ell \|\nabla\phi_\ell^{\mbox{\rm\tiny IW}}\|_{H^m}^2 \\
&\leq 2\sum_{\ell=1,2} \underline{\rho}_\ell \underline{h}_\ell \|\nabla\phi_\ell^{\mbox{\rm\tiny IW}}\|_{H^m}^2,
\end{align*}
where we used~\eqref{parameter-relation}. 
It follows from the third equation in~\eqref{full-model-evolution} that 
\begin{align*}
\partial_t\phi^{\mbox{\rm\tiny IW}}
&= \underline{\rho}_2\partial_t\phi_2^{\mbox{\rm\tiny IW}} - \underline{\rho}_1\partial_t\phi_1^{\mbox{\rm\tiny IW}} \\
&= \frac12\underline{\rho}_1\biggl( |\nabla\phi_1^{\mbox{\rm\tiny IW}}|^2
 - \delta^2 \frac{(\Lambda_1^{\mbox{\rm\tiny IW}}\phi_1^{\mbox{\rm\tiny IW}}
  - \nabla\zeta^{\mbox{\rm\tiny IW}} \cdot \nabla\phi_1^{\mbox{\rm\tiny IW}} )^2}{
  1+\delta^2|\nabla\zeta^{\mbox{\rm\tiny IW}}|^2} \biggr) \\
&\quad\;
 - \frac12\underline{\rho}_2\biggl( |\nabla\phi_2^{\mbox{\rm\tiny IW}}|^2
  - \delta^2 \frac{(\Lambda_2^{\mbox{\rm\tiny IW}}\phi_2^{\mbox{\rm\tiny IW}}
   + \nabla\zeta^{\mbox{\rm\tiny IW}} \cdot \nabla\phi_2^{\mbox{\rm\tiny IW}} )^2}{
  1+\delta^2|\nabla\zeta^{\mbox{\rm\tiny IW}}|^2} \biggr)
 - \zeta^{\mbox{\rm\tiny IW}}.
\end{align*}
Here, we note that in view of the conditions $\underline{h}_1\delta, \underline{h}_2\delta \leq 1$ and 
$\underline{h}_1^{-1}, \underline{h}_2^{-1}\lesssim 1$ we have $\delta\lesssim 1$. 
Therefore, by Lemma~\ref{L.estimate-DN2} we have 
\begin{align*}
\|\partial_t\phi^{\mbox{\rm\tiny IW}}\|_{H^{m-1}}
&\lesssim \|\zeta^{\mbox{\rm\tiny IW}}\|_{H^{m-1}}
 + \sum_{\ell=1,2}\underline{\rho}_\ell\{
  \|\nabla\phi_\ell^{\mbox{\rm\tiny IW}}\|_{H^{m-1}}^2
 + \delta^2(\underline{h}_\ell^2\|\nabla\phi_\ell^{\mbox{\rm\tiny IW}}\|_{H^m}^2
  + \|\nabla\phi_\ell^{\mbox{\rm\tiny IW}}\|_{H^{m-1}}^2) \} \\
&\lesssim \|\zeta^{\mbox{\rm\tiny IW}}\|_{H^{m-1}}
 + \sum_{\ell=1,2} \underline{\rho}_\ell \underline{h}_\ell \|\nabla\phi_\ell^{\mbox{\rm\tiny IW}}\|_{H^m}^2.
\end{align*}
Hence, we obtain $\|\partial_t\zeta^{\mbox{\rm\tiny IW}}\|_{H^{m-1}}+\|\partial_t\phi^{\mbox{\rm\tiny IW}}\|_{H^{m-1}} \lesssim 1$.

Let $(\tilde{\bm{\phi}}_1^{\mbox{\rm\tiny IW}},\tilde{\bm{\phi}}_2^{\mbox{\rm\tiny IW}})$ be the solution 
to~\eqref{Conditions-v2bis} with $(\zeta,\phi)=(\zeta^{\mbox{\rm\tiny IW}},\phi^{\mbox{\rm\tiny IW}})$. 
Then, Proposition~\ref{prop-consistency} states that 
$(\zeta^{\mbox{\rm\tiny IW}},\tilde{\bm{\phi}}_1^{\mbox{\rm\tiny IW}},\tilde{\bm{\phi}}_2^{\mbox{\rm\tiny IW}})$ satisfy 
approximately the Kakinuma model as~\eqref{app-Kakinuma} and the errors $(\bm{\mathfrak{r}}_1,\bm{\mathfrak{r}}_2,\mathfrak{r}_0)$ 
satisfy~\eqref{error-estimate2}. 
Moreover, it follows from Lemma~\ref{L.elliptic} that 
\begin{align*}
\sum_{\ell=1,2} \underline{\rho}_\ell \underline{h}_\ell
 ( \|\nabla\tilde{\bm{\phi}}_\ell^{\mbox{\rm\tiny IW}}\|_{H^m}^2
  + (\underline{h}_\ell\delta)^{-2}\|\tilde{\bm{\phi}}_\ell^{\mbox{\rm\tiny IW}\, \prime}\|_{H^m}^2 )
&\lesssim \min\biggl\{\frac{\underline{h}_1}{\underline{\rho}_1},\frac{\underline{h}_2}{\underline{\rho}_2} \biggr\}
 \|\nabla\phi^{\mbox{\rm\tiny IW}}\|_{H^m}^2 \\
&\lesssim \sum_{\ell=1,2} \underline{\rho}_\ell \underline{h}_\ell \|\nabla\phi_\ell^{\mbox{\rm\tiny IW}}\|_{H^m}^2 
\lesssim 1,
\end{align*}
which yields 
\[
\sum_{\ell=1,2} \underline{\rho}_\ell \underline{h}_\ell
 ( \|\tilde{\bm{u}}_\ell^{\mbox{\rm\tiny IW}}\|_{H^m}^2
  + (\underline{h}_\ell\delta)^{-2}\|\tilde{w}_\ell^{\mbox{\rm\tiny IW}}\|_{H^m}^2 
  + (\underline{h}_\ell\delta)^{-4}\|\tilde{\bm{\phi}}_\ell^{\mbox{\rm\tiny IW}\, \prime}\|_{H^{m-1}}^2 ) \lesssim 1,
\]
where $\tilde{\bm{u}}_1^{\mbox{\rm\tiny IW}}, \tilde{\bm{u}}_2^{\mbox{\rm\tiny IW}}, \tilde{w}_1^{\mbox{\rm\tiny IW}}, 
\tilde{w}_2^{\mbox{\rm\tiny IW}}$ are defined by~\eqref{def-uw} with $(\bm{\phi}_1,\bm{\phi}_2)$ replaced by 
$(\tilde{\bm{\phi}}_1^{\mbox{\rm\tiny IW}},\tilde{\bm{\phi}}_2^{\mbox{\rm\tiny IW}})$, and we used Lemma~\ref{L.EE-low}. 
We proceed to evaluate $(\partial_t\tilde{\bm{\phi}}_1^{\mbox{\rm\tiny IW}},\partial_t\tilde{\bm{\phi}}_2^{\mbox{\rm\tiny IW}})$. 
To this end, we derive equations for these time derivatives by differentiating~\eqref{Conditions-v2bis} with respect to $t$. 
The procedure is almost the same as in the proof of Lemma~\ref{L.time-derivatives-and-elliptic}. 
The only difference is the last equation in~\eqref{equationdtphi}, especially, the expression of $f_4$. 
In this case, $f_4$ has the form 
\[
f_4 = \partial_t\phi^{\mbox{\rm\tiny IW}}
 + \underline{\rho}_1\tilde{w}_1^{\mbox{\rm\tiny IW}}\underline{h}_1^{-1}\partial_t\zeta^{\mbox{\rm\tiny IW}}
 - \underline{\rho}_2\tilde{w}_2^{\mbox{\rm\tiny IW}}\underline{h}_2^{-1}\partial_t\zeta^{\mbox{\rm\tiny IW}},
\]
so that $\|f_4\|_{H^{m-1}} \lesssim 1$. 
Therefore, we obtain 
\[
\sum_{\ell=1,2} \underline{\rho}_\ell \underline{h}_\ell
 ( \|\nabla\partial_t\tilde{\bm{\phi}}_\ell^{\mbox{\rm\tiny IW}}\|_{H^{m-2}}^2
  + (\underline{h}_\ell\delta)^{-2}\|\partial_t\tilde{\bm{\phi}}_\ell^{\mbox{\rm\tiny IW}\, \prime}\|_{H^{m-2}}^2 ) \lesssim 1. 
\]

Let $(\zeta^{\mbox{\rm\tiny K}},\bm{\phi}_1^{\mbox{\rm\tiny K}},\bm{\phi}_2^{\mbox{\rm\tiny K}})$ be the solution 
to the initial value problem for the Kakinuma model~\eqref{Kakinuma-dimensionless}--\eqref{Kaki:IC} stated in the theorem, whose unique existence is guaranteed by 
Theorem~\ref{theorem-uniform} and Proposition~\ref{preparation-ini}. 
Note also that the solution satisfies the uniform bound~\eqref{uniform-sol} together with the stability and 
non-cavitation conditions~\eqref{uniform-below}. 
It follows from Lemma~\ref{L.EE-low} that $\underline{\rho}_\ell\underline{h}_\ell(\underline{h}_\ell\delta)^{-4}
 \|\bm{\phi}_\ell^{\mbox{\rm\tiny K} \, \prime}\|_{H^{m-1}}^2 \lesssim 1$ for $\ell=1,2$. 
Moreover, the time derivatives 
$(\partial_t\zeta^{\mbox{\rm\tiny K}},\partial_t\bm{\phi}_1^{\mbox{\rm\tiny K}},\partial_t\bm{\phi}_2^{\mbox{\rm\tiny K}})$ 
satisfy~\eqref{estimate-time-derivatives} and $(\bm{u}_\ell^{\mbox{\rm\tiny K}},w_\ell^{\mbox{\rm\tiny K}})$ $(\ell=1,2)$,
which are defined by~\eqref{def-uw} with $(\bm{\phi}_1,\bm{\phi}_2)$ replaced by 
$({\bm{\phi}}_1^{\mbox{\rm\tiny K}},{\bm{\phi}}_2^{\mbox{\rm\tiny K}})$, satisfy~\eqref{estimates-uw}. 
Putting
\[
\zeta^\mathrm{res} := \zeta^{\mbox{\rm\tiny K}}-\zeta^{\mbox{\rm\tiny IW}}, \qquad
\bm{\phi}_\ell^\mathrm{res} := \bm{\phi}_\ell^{\mbox{\rm\tiny K}} - \tilde{\bm{\phi}}_\ell^{\mbox{\rm\tiny IW}} \quad
(\ell=1,2),
\]
we will show that $(\zeta^\mathrm{res},\bm{\phi}_1^\mathrm{res},\bm{\phi}_2^\mathrm{res})$ can be estimated 
by the errors $(\bm{\mathfrak{r}}_1,\bm{\mathfrak{r}}_2,\mathfrak{r}_0)$. 
To this end, we are going to evaluate 
\[
E_k^\mathrm{res}(t)
:= \|\zeta^\mathrm{res}(t)\|_{H^k}^2 + \sum_{\ell=1,2} \underline{\rho}_\ell \underline{h}_\ell
 ( \|\nabla\bm{\phi}_\ell^\mathrm{res}(t)\|_{H^k}^2
  + (\underline{h}_\ell\delta)^{-2}\|\bm{\phi}_\ell^{\mathrm{res}\, \prime}(t)\|_{H^k}^2 )
\]
for an appropriate integer $k$ by making use of energy estimates similar to the ones obtained in Sections~\ref{S.elliptic} 
and~\ref{S.hyperbolic} for the proof of the well-posedness of the initial value problem for the Kakinuma model~\eqref{Kakinuma-dimensionless}--\eqref{Kaki:IC}.
Here, we note that $E_k^\mathrm{res}(0)=0$.

As in the case of the energy estimate for the Kakinuma model, we first need to evaluate times derivatives 
$(\partial_t\zeta^\mathrm{res},\partial_t\bm{\phi}_1^\mathrm{res},\partial_t\bm{\phi}_2^\mathrm{res})$ 
in terms of $E_k^\mathrm{res}$. 
By taking difference between the first components of the first two equations in~\eqref{Kakinuma-dimensionless-compact} and 
\eqref{app-Kakinuma}, $\partial_t\zeta^\mathrm{res}$ can be written in two ways as 
\begin{align*}
\partial_t\zeta^\mathrm{res}
&= -\underline{h}_1\{ \mathcal{L}_{1,0}^{\mbox{\rm\tiny K}}\bm{\phi}_1^\mathrm{res}
 + (\mathcal{L}_{1,0}^{\mbox{\rm\tiny K}}-\mathcal{L}_{1,0}^{\mbox{\rm\tiny IW}})\tilde{\bm{\phi}}_1^{\mbox{\rm\tiny IW}}
 + \mathfrak{r}_{1,0} \} \\
&= \underline{h}_2\{ \mathcal{L}_{2,0}^{\mbox{\rm\tiny K}}\bm{\phi}_2^\mathrm{res}
 + (\mathcal{L}_{2,0}^{\mbox{\rm\tiny K}}-\mathcal{L}_{2,0}^{\mbox{\rm\tiny IW}})\tilde{\bm{\phi}}_2^{\mbox{\rm\tiny IW}}
 + \mathfrak{r}_{2,0} \},
\end{align*}
where $\mathcal{L}_{1,0}^{\mbox{\rm\tiny K}}=\mathcal{L}_{1,0}(H_1^{\mbox{\rm\tiny K}},\delta,\underline{h}_1)$, 
$H_1^{\mbox{\rm\tiny K}}=1-\underline{h}_1^{-1}\zeta^{\mbox{\rm\tiny K}}$, and similar simplifications are used, 
and $\mathfrak{r}_{\ell,0}$ is the $0$th component of the error $\bm{\mathfrak{r}}_\ell$ for $\ell=1,2$. 
Therefore, we have 
\begin{align*}
\|\partial_t\zeta^\mathrm{res}\|_{H^{k-1}}
&\lesssim \underline{h}_\ell\{ \|\nabla\bm{\phi}_\ell^\mathrm{res}\|_{H^k} 
 + \|\bm{\phi}_{\ell}^\mathrm{res\, \prime}\|_{H^k} \\
&\quad
 + \|\zeta^\mathrm{res}\|_{H^k} ( \|\nabla\tilde{\bm{\phi}}_\ell^{\mbox{\rm\tiny IW}}\|_{H^m} 
 + \|\tilde{\bm{\phi}}_{\ell}^\mathrm{\mbox{\rm\tiny IW}\, \prime}\|_{H^m} )
 + \|\mathfrak{r}_{\ell,0}\|_{H^{k-1}} \}
\end{align*}
for $\ell=1,2$ and $|k|\leq m$. 
Hence, by the technique used in the proof of Lemma~\ref{L.time-derivatives-and-elliptic} we obtain 
\begin{align*}
\|\partial_t\zeta^\mathrm{res}\|_{H^{k-1}}^2
&\lesssim \sum_{\ell=1,2} \underline{\rho}_\ell \underline{h}_\ell\{
 \|\nabla\bm{\phi}_\ell^\mathrm{res}\|_{H^k}^2 
 + \|\bm{\phi}_{\ell}^\mathrm{res\, \prime}\|_{H^k}^2 \\
&\qquad
 + \|\zeta^\mathrm{res}\|_{H^k}^2 ( \|\nabla\tilde{\bm{\phi}}_\ell^{\mbox{\rm\tiny IW}}\|_{H^m}^2 
 + \|\tilde{\bm{\phi}}_{\ell}^\mathrm{\mbox{\rm\tiny IW}\, \prime}\|_{H^m}^2 ) + \|\mathfrak{r}_{\ell,0}\|_{H^{k-1}}^2 \} \\
&\lesssim E_k^\mathrm{res} + \sum_{\ell=1,2} \underline{\rho}_\ell \underline{h}_\ell \|\bm{\mathfrak{r}}_\ell\|_{H^{k-1}}^2
\end{align*}
for $|k|\leq m$. 
We proceed to evaluate $(\partial_t\bm{\phi}_1^\mathrm{res},\partial_t\bm{\phi}_2^\mathrm{res})$. 
We recall that $(\partial_t\bm{\phi}_1^{\mbox{\rm\tiny K}},\partial_t\bm{\phi}_2^{\mbox{\rm\tiny K}})$ satisfy~\eqref{equationdtphi} with 
$(\zeta,\bm{\phi}_1,\bm{\phi}_2)=(\zeta^{\mbox{\rm\tiny K}},\bm{\phi}_1^{\mbox{\rm\tiny K}},\bm{\phi}_2^{\mbox{\rm\tiny K}})$ 
and note that, differentiating the first three equations of~\eqref{Conditions-v2bis} with respect to $t$ 
and using the last equation in~\eqref{app-Kakinuma}, 
$(\partial_t\tilde{\bm{\phi}}_1^{\mbox{\rm\tiny IW}},\partial_t\tilde{\bm{\phi}}_2^{\mbox{\rm\tiny IW}})$ also satisfy 
\eqref{equationdtphi} with $(\zeta,\bm{\phi}_1,\bm{\phi}_2)
=(\zeta^{\mbox{\rm\tiny IW}},\tilde{\bm{\phi}}_1^{\mbox{\rm\tiny IW}},\tilde{\bm{\phi}}_2^{\mbox{\rm\tiny IW}})$ 
and $f_4$ added with the error term $-\mathfrak{r}_0$. 
By taking the difference between these equations, we have therefore 
\[
\begin{cases}
 \mathcal{L}_{1,i}^{\mbox{\rm\tiny IW}} \partial_t\bm{\phi}_1^\mathrm{res} = f_{1,i}^\mathrm{res} \quad\mbox{for}\quad i=1,2,\ldots,N, \\
 \mathcal{L}_{2,i}^{\mbox{\rm\tiny IW}} \partial_t\bm{\phi}_2^\mathrm{res} = f_{2,i}^\mathrm{res} \quad\mbox{for}\quad i=1,2,\ldots,N^*, \\
 \underline{h}_1\mathcal{L}_{1,0}^{\mbox{\rm\tiny IW}} \partial_t\bm{\phi}_1^\mathrm{res}
  + \underline{h}_2\mathcal{L}_{2,0}^{\mbox{\rm\tiny IW}} \partial_t\bm{\phi}_2^\mathrm{res} = \nabla\cdot\bm{f}_3^\mathrm{res}, \\
 - \underline{\rho}_1\bm{l}_1^{\mbox{\rm\tiny IW}} \cdot \partial_t\bm{\phi}_1^\mathrm{res}
  + \underline{\rho}_2\bm{l}_2^{\mbox{\rm\tiny IW}} \cdot \partial_t\bm{\phi}_2^\mathrm{res} = f_4^\mathrm{res},
\end{cases}
\]
where 
\[
\begin{cases}
 f_{1,i}^\mathrm{res}
  = f_{1,i}^{\mbox{\rm\tiny K}}-\tilde{f}_{1,i}^{\mbox{\rm\tiny IW}}
   + (\mathcal{L}_{1,i}^{\mbox{\rm\tiny IW}}-\mathcal{L}_{1,i}^{\mbox{\rm\tiny K}})\partial_t\bm{\phi}_1^{\mbox{\rm\tiny K}}
   \quad\mbox{for}\quad i=1,2,\ldots,N, \\
 f_{2,i}^\mathrm{res}
  = f_{2,i}^{\mbox{\rm\tiny K}}-\tilde{f}_{2,i}^{\mbox{\rm\tiny IW}}
   + (\mathcal{L}_{2,i}^{\mbox{\rm\tiny IW}}-\mathcal{L}_{2,i}^{\mbox{\rm\tiny K}})\partial_t\bm{\phi}_2^{\mbox{\rm\tiny K}}
   \quad\mbox{for}\quad i=1,2,\ldots,N^*, \\
 \bm{f}_3^\mathrm{res}
  = \bm{f}_3^{\mbox{\rm\tiny K}}-\tilde{\bm{f}}_3^{\mbox{\rm\tiny IW}}
   + \underline{h}_1((\bm{a}_{1,0}^{\mbox{\rm\tiny K}}-\bm{a}_{1,0}^{\mbox{\rm\tiny IW}})\otimes\nabla)^\mathrm{T}
    \partial_t \bm{\phi}_1^{\mbox{\rm\tiny K}} \\
 \qquad\quad{}
   + \underline{h}_2 \{((\bm{a}_{2,0}^{\mbox{\rm\tiny K}}-\bm{a}_{2,0}^{\mbox{\rm\tiny IW}})\otimes\nabla)^\mathrm{T}
    \partial_t \bm{\phi}_2^{\mbox{\rm\tiny K}} 
   - ((\bm{b}_{2,0}^{\mbox{\rm\tiny K}}-\bm{b}_{2,0}^{\mbox{\rm\tiny IW}})\cdot \partial_t \bm{\phi}_2^{\mbox{\rm\tiny K}})
    \underline{h}_2^{-1}\nabla b \}, \\
 f_4^\mathrm{res}
  = f_4^{\mbox{\rm\tiny K}}-\tilde{f}_4^{\mbox{\rm\tiny IW}}+\mathfrak{r}_0
   - \underline{\rho}_1(\bm{l}_1^{\mbox{\rm\tiny IW}}-\bm{l}_1^{\mbox{\rm\tiny K}})\cdot\partial_t\bm{\phi}_1^{\mbox{\rm\tiny K}}
   + \underline{\rho}_2(\bm{l}_2^{\mbox{\rm\tiny IW}}-\bm{l}_2^{\mbox{\rm\tiny K}})\cdot\partial_t\bm{\phi}_2^{\mbox{\rm\tiny K}}.
\end{cases}
\]
Here, $f_{1,i}^{\mbox{\rm\tiny K}}$, $f_{2,i}^{\mbox{\rm\tiny K}}$, $\bm{f}_3^{\mbox{\rm\tiny K}}$, $f_4^{\mbox{\rm\tiny K}}$ 
(respectively $\tilde{f}_{1,i}^{\mbox{\rm\tiny IW}}$, $\tilde{f}_{2,i}^{\mbox{\rm\tiny IW}}$, $\tilde{\bm{f}}_3^{\mbox{\rm\tiny IW}}$, 
 $\tilde{f}_4^{\mbox{\rm\tiny IW}}$) are those in~\eqref{commutator-dt} with $(\zeta,\bm{\phi}_1,\bm{\phi}_2) 
 =(\zeta^{\mbox{\rm\tiny K}},\bm{\phi}_1^{\mbox{\rm\tiny K}},\bm{\phi}_2^{\mbox{\rm\tiny K}})$ 
(respectively $(\zeta,\bm{\phi}_1,\bm{\phi}_2)
 =(\zeta^{\mbox{\rm\tiny IW}},\tilde{\bm{\phi}}_1^{\mbox{\rm\tiny IW}},\tilde{\bm{\phi}}_2^{\mbox{\rm\tiny IW}})$), 
$\bm{a}_{\ell,0}^{\mbox{\rm\tiny K}}=\bm{a}_{\ell,0}(H_\ell^{\mbox{\rm\tiny K}})$ and 
$\bm{b}_{2,0}^{\mbox{\rm\tiny K}}=\bm{b}_{2,0}(H_2^{\mbox{\rm\tiny K}})$, 
where $\bm{a}_{\ell,0}(H_\ell)$ and $\bm{b}_{2,0}(H_2)$ are the $0$th columns of the matrixes 
$A_\ell(H_\ell)$ and $B_2(H_2)$ defined by~\eqref{def-A} and~\eqref{def-B2}, respectively, and so on.
Note the relations $\mathcal{L}_{1,0}\bm{\phi}_1=-\nabla\cdot((\bm{a}_{1,0}\otimes\nabla)^\mathrm{T}\bm{\phi}_1)$ 
and $\mathcal{L}_{2,0}\bm{\phi}_2=-\nabla\cdot((\bm{a}_{2,0}\otimes\nabla)^\mathrm{T}\bm{\phi}_2
-(\bm{b}_{2,0}\cdot\bm{\phi}_2)\underline{h}_2^{-1}\nabla b)$. 
Therefore, by Lemma~\ref{L.elliptic} we have, for $1\leq k\leq m+1$, 
\begin{align*}
& \sum_{\ell=1,2}\underline{\rho}_\ell \underline{h}_\ell \bigl(
 \|\nabla\partial_t\bm{\phi}_\ell^\mathrm{res}\|_{H^{k-1}}^2
  + (\underline{h}_\ell\delta)^{-2}\|\partial_t\bm{\phi}_\ell^{\mathrm{res}\,\prime}\|_{H^{k-1}}^2 \bigr) \\
&\lesssim  \sum_{\ell=1,2}\underline{\rho}_\ell \underline{h}_\ell
  (\underline{h}_\ell\delta)^2 \|\bm{f}_\ell^{\mathrm{res} \, \prime}\|_{H^{k-1}}^2 
 + \min\biggl\{\frac{\underline{\rho}_1}{\underline{h}_1},\frac{\underline{\rho}_2}{\underline{h}_2} \biggr\}
  \|\bm{f}_3^\mathrm{res}\|_{H^{k-1}}^2
 + \min\biggl\{\frac{\underline{h}_1}{\underline{\rho}_1},\frac{\underline{h}_2}{\underline{\rho}_2} \biggr\}
  \|f_4^\mathrm{res}\|_{H^k}^2.
\end{align*}
We will evaluate each term in the right-hand side. 
For $1\leq k\leq m-1$, we see that
\begin{align*}
\|\bm{f}_\ell^{\mathrm{res} \, \prime}\|_{H^{k-1}}
&\lesssim  \underline{h}_\ell^{-1}\|\zeta^\mathrm{res}\|_{H^{k}} ( \|\nabla\bm{\phi}_\ell^{\mbox{\rm\tiny K}}\|_{H^m}
 + (\underline{h}_\ell\delta)^{-2}\|\bm{\phi}_\ell^{\mbox{\rm\tiny K} \, \prime}\|_{H^m} )
  \underline{h}_\ell^{-1}\|\partial_t\zeta^{\mbox{\rm\tiny K}}\|_{H^{m-1}} \\
&\quad\;
 + ( \|\nabla\bm{\phi}_\ell^\mathrm{res}\|_{H^k}
  + (\underline{h}_\ell\delta)^{-2}\|\bm{\phi}_\ell^{\mathrm{res} \, \prime}\|_{H^k} )
  \underline{h}_\ell^{-1}\|\partial_t\zeta^{\mbox{\rm\tiny K}}\|_{H^{m-1}} \\
&\quad\; 
 + ( \|\nabla\tilde{\bm{\phi}}_\ell^{\mbox{\rm\tiny IW}}\|_{H^m}
  + (\underline{h}_\ell\delta)^{-2}\|\tilde{\bm{\phi}}_\ell^{\mbox{\rm\tiny IW} \, \prime}\|_{H^m} )
  \underline{h}_\ell^{-1}\|\partial_t\zeta^\mathrm{res}\|_{H^{k-1}}\\
&\quad\; 
 + \underline{h}_\ell^{-1}\|\zeta^\mathrm{res}\|_{H^{k}}( \|\nabla\partial_t{\bm{\phi}}_\ell^{\mbox{\rm\tiny K}}\|_{H^{m-1}}
  + (\underline{h}_\ell\delta)^{-2}\|\partial_t{\bm{\phi}}_\ell^{\mbox{\rm\tiny K} \, \prime}\|_{H^{m-1}} ) 
\end{align*}
for $\ell=1,2$, 
\begin{align*}
\|\bm{f}_3^\mathrm{res}\|_{H^{k-1}}
&\lesssim \sum_{\ell=1,2}\{
 \|\bm{u}_\ell^{\mbox{\rm\tiny K}}-\tilde{\bm{u}}_\ell^{\mbox{\rm\tiny IW}}\|_{H^k}
  \|\partial_t\zeta^{\mbox{\rm\tiny K}}\|_{H^{m-1}}
 + \|\tilde{\bm{u}}_\ell^{\mbox{\rm\tiny IW}}\|_{H^m}\|\partial_t\zeta^\mathrm{res}\|_{H^{k-1}}\\
&\qquad
 + \|\zeta^\mathrm{res}\|_{H^k}( \|\nabla\partial_t \bm{\phi}_\ell^{\mbox{\rm\tiny K}}\|_{H^{m-1}}
  + \|\partial_t\bm{\phi}_\ell^{\mbox{\rm\tiny K} \, \prime}\|_{H^{m-1}} ) \},
\end{align*}
and 
\begin{align*}
\|f_4^\mathrm{res}\|_{H^k}
&\lesssim \sum_{\ell=1,2}\underline{\rho}_\ell \big\{
 (\|\bm{u}_\ell^{\mbox{\rm\tiny K}}\|_{H^m}+\|\tilde{\bm{u}}_\ell^{\mbox{\rm\tiny IW}}\|_{H^m})
 \|\bm{u}_\ell^{\mbox{\rm\tiny K}}-\tilde{\bm{u}}_\ell^{\mbox{\rm\tiny IW}}\|_{H^k} \\
&\qquad
 + (\underline{h}_\ell\delta)^{-2}(\|w_\ell^{\mbox{\rm\tiny K}}\|_{H^m}+\|\tilde{w}_\ell^{\mbox{\rm\tiny IW}}\|_{H^m})
 \|w_\ell^{\mbox{\rm\tiny K}}-\tilde{w}_\ell^{\mbox{\rm\tiny IW}}\|_{H^k} \\
&\qquad
 + \underline{h}_\ell^{-1} \|\zeta^\mathrm{res}\|_{H^k} \|\partial_t\bm{\phi}_\ell^{\mbox{\rm\tiny K} \, \prime}\|_{H^{m-1}}  \big\} 
 + \|\zeta^\mathrm{res}\|_{H^k} +\|\mathfrak{r}_0\|_{H^k}.
\end{align*}
Moreover, for any $0\leq k\leq m$ we have also 
\begin{equation}\label{estimate-uw-res}
\sum_{\ell=1,2}\underline{\rho}_\ell \underline{h}_\ell \bigl(
 \|\bm{u}_\ell^{\mbox{\rm\tiny K}}-\tilde{\bm{u}}_\ell^{\mbox{\rm\tiny IW}}\|_{H^k}^2
 + (\underline{h}_\ell\delta)^{-2}\|w_\ell^{\mbox{\rm\tiny K}}-\tilde{w}_\ell^{\mbox{\rm\tiny IW}}\|_{H^k}^2) 
\lesssim E_k^\mathrm{res}. 
\end{equation}
Summarizing the above estimates and using $\underline{h}_1^{-1}, \underline{h}_2^{-1}\lesssim 1$ we obtain, 
for $1\leq k\leq m-1$, 
\begin{align}\label{estimate-dtphi-res}
 \|\partial_t\zeta^\mathrm{res}\|_{H^{k-1}}^2
 + \sum_{\ell=1,2}\underline{\rho}_\ell \underline{h}_\ell \bigl(
 \|\nabla\partial_t\bm{\phi}_\ell^\mathrm{res}\|_{H^{k-1}}^2
  + (\underline{h}_\ell\delta)^{-2}\|\partial_t\bm{\phi}_\ell^{\mathrm{res}\,\prime}\|_{H^{k-1}}^2 \bigr) \\
\lesssim E_k^\mathrm{res} + \sum_{\ell=1,2}\underline{\rho}_\ell\underline{h}_\ell\|\bm{\mathfrak{r}}_\ell\|_{H^{k-1}}^2
 + \|\mathfrak{r}_0\|_{H^k}^2. \nonumber
\end{align}
We need also to evaluate $\underline{\rho}_\ell\underline{h}_\ell(\underline{h}_\ell\delta)^{-4}
\|\bm{\phi}_\ell^{\mathrm{res} \, \prime}\|_{H^{k-1}}^2$ for $\ell=1,2$ in terms of $E_k^\mathrm{res}$. 
In view of 
\[
\begin{cases}
 \mathcal{L}_{1,i}^{\mbox{\rm\tiny IW}} \bm{\phi}_1^\mathrm{res}
  =  \mathcal{L}_{1,i}^{\mbox{\rm\tiny IW}} \bm{\phi}_1^{\mbox{\rm\tiny K}}
  = (\mathcal{L}_{1,i}^{\mbox{\rm\tiny IW}}-\mathcal{L}_{1,i}^{\mbox{\rm\tiny K}})\bm{\phi}_1^{\mbox{\rm\tiny K}}
  =: h_{1,i}^\mathrm{res}
   \quad\mbox{for}\quad i=1,2,\ldots,N, \\
 \mathcal{L}_{2,i}^{\mbox{\rm\tiny IW}} \bm{\phi}_2^\mathrm{res}
  = \mathcal{L}_{2,i}^{\mbox{\rm\tiny IW}} \bm{\phi}_2^{\mbox{\rm\tiny K}}
  = (\mathcal{L}_{2,i}^{\mbox{\rm\tiny IW}}-\mathcal{L}_{2,i}^{\mbox{\rm\tiny K}})\bm{\phi}_2^{\mbox{\rm\tiny K}}
  =: h_{2,i}^\mathrm{res}
   \quad\mbox{for}\quad i=1,2,\ldots,N^*, 
\end{cases}
\]
Lemma~\ref{L.EE-low} yields 
$(\underline{h}_\ell\delta)^{-2} \|\bm{\phi}_\ell^{\mathrm{res} \, \prime}\|_{H^{k-1}}
\lesssim \|\nabla\bm{\phi}_\ell^\mathrm{res}\|_{H^k} + \|\bm{\phi}_\ell^{\mathrm{res} \, \prime} \|_{H^{k}}
 + \|\bm{h}_\ell^{\mathrm{res} \, \prime}\|_{H^{k-1}}$ and we have 
$\|\bm{h}_\ell^{\mathrm{res} \, \prime}\|_{H^{k-1}} \lesssim 
 (\|\nabla\bm{\phi}_\ell^{\mbox{\rm\tiny K}}\|_{H^m} + \|\bm{\phi}_\ell^{\mbox{\rm\tiny K} \, \prime}\|_{H^m}
  + (\underline{h}_\ell\delta)^{-2}\|\bm{\phi}_\ell^{\mbox{\rm\tiny K} \, \prime}\|_{H^{m-1}})
  \|\zeta^\mathrm{res}\|_{H^k}$ for $1\leq k\leq m$. 
Therefore, for $1\leq k\leq m$ we obtain
\begin{equation}\label{estimate-low-res}
\sum_{\ell=1,2}\underline{\rho}_\ell \underline{h}_\ell
 (\underline{h}_\ell\delta)^{-4}\|\bm{\phi}_\ell^{\mathrm{res} \, \prime}\|_{H^{k-1}}^2
 \lesssim E_k^\mathrm{res}.
\end{equation}

Now, by deriving equations for spatial derivatives of $(\zeta^\mathrm{res},\bm{\phi}_1^\mathrm{res},\bm{\phi}_2^\mathrm{res})$ 
and applying the energy estimate obtained in Subsection~\ref{Analysis-LEs} we will evaluate $E_{k}^\mathrm{res}$. 
Let $\beta$ be a multi-index such that $1\leq|\beta|\leq k$. 
Applying $\partial^\beta$ to the Kakinuma model~\eqref{Kakinuma-dimensionless-compact} for 
$(\zeta^{\mbox{\rm\tiny K}},\bm{\phi}_1^{\mbox{\rm\tiny K}},\bm{\phi}_2^{\mbox{\rm\tiny K}})$ 
and to~\eqref{app-Kakinuma} for 
$(\zeta^{\mbox{\rm\tiny IW}},\tilde{\bm{\phi}}_1^{\mbox{\rm\tiny IW}},\tilde{\bm{\phi}}_2^{\mbox{\rm\tiny IW}})$ 
and taking the difference between the resulting equations, we obtain 
\[
\begin{cases}
 \displaystyle
 {\bm l}_1^{\mbox{\rm\tiny K}}(\partial_t+\bm{u}_1^{\mbox{\rm\tiny K}}\cdot\nabla)\partial^\beta\zeta^\mathrm{res}
  + \underline{h}_1 L_1^{\mbox{\rm\tiny K},\mathrm{pr}}\partial^\beta\bm{\phi}_1^\mathrm{res} = \bm{f}_{1,\beta}^\mathrm{res}, \\
 \displaystyle
 {\bm l}_2^{\mbox{\rm\tiny K}}(\partial_t+\bm{u}_2^{\mbox{\rm\tiny K}}\cdot\nabla)\partial^\beta\zeta^\mathrm{res}
  - \underline{h}_2 L_2^{\mbox{\rm\tiny K},\mathrm{pr}}\partial^\beta\bm{\phi}_2^\mathrm{res} = \bm{f}_{2,\beta}^\mathrm{res}, \\
 \underline{\rho}_1{\bm l}_1^{\mbox{\rm\tiny K}}\cdot
  ( \partial_t+\bm{u}_1^{\mbox{\rm\tiny K}}\cdot\nabla )\partial^\beta\bm{\phi}_1^\mathrm{res} 
  - \underline{\rho}_2{\bm l}_2^{\mbox{\rm\tiny K}}\cdot
  ( \partial_t+\bm{u}_1^{\mbox{\rm\tiny K}}\cdot\nabla )\partial^\beta\bm{\phi}_2^\mathrm{res} 
  - a^{\mbox{\rm\tiny K}}\partial^\beta\zeta^\mathrm{res} = f_{0,\beta}^\mathrm{res},
\end{cases}
\]
where 
\[
\begin{cases}
 \bm{f}_{1,\beta}^\mathrm{res} := \bm{f}_{1,\beta}^{\mbox{\rm\tiny K}} - \tilde{\bm{f}}_{1,\beta}^{\mbox{\rm\tiny IW}}
  - \underline{h}_1\partial^\beta\bm{\mathfrak{r}}_1
  + \underline{h}_1( L_1^{\mbox{\rm\tiny IW},\mathrm{pr}} - L_1^{\mbox{\rm\tiny K},\mathrm{pr}} )
   \partial^\beta\tilde{\bm{\phi}}_1^{\mbox{\rm\tiny IW}} \\
 \qquad\quad{}
  +\bigl(\bm{l}_1^{\mbox{\rm\tiny IW}}(\partial_t+\tilde{\bm{u}}_1^{\mbox{\rm\tiny IW}}\cdot\nabla)
   - \bm{l}_1^{\mbox{\rm\tiny K}}(\partial_t+\bm{u}_1^{\mbox{\rm\tiny K}}\cdot\nabla) \bigr)
   \partial^\beta\zeta^{\mbox{\rm\tiny IW}}, \\
 \bm{f}_{2,\beta}^\mathrm{res} := \bm{f}_{2,\beta}^{\mbox{\rm\tiny K}} - \tilde{\bm{f}}_{2,\beta}^{\mbox{\rm\tiny IW}}
  - \underline{h}_2\partial^\beta\bm{\mathfrak{r}}_2
  - \underline{h}_2( L_2^{\mbox{\rm\tiny IW},\mathrm{pr}} - L_2^{\mbox{\rm\tiny K},\mathrm{pr}} )
   \partial^\beta\tilde{\bm{\phi}}_2^{\mbox{\rm\tiny IW}} \\
 \qquad\quad{}
  +\bigl(\bm{l}_2^{\mbox{\rm\tiny IW}}(\partial_t+\tilde{\bm{u}}_2^{\mbox{\rm\tiny IW}}\cdot\nabla)
   - \bm{l}_2^{\mbox{\rm\tiny K}}(\partial_t+\bm{u}_2^{\mbox{\rm\tiny K}}\cdot\nabla) \bigr)
   \partial^\beta\zeta^{\mbox{\rm\tiny IW}}, \\
 f_{0,\beta}^\mathrm{res} := f_{0,\beta}^{\mbox{\rm\tiny K}} - \tilde{f}_{0,\beta}^{\mbox{\rm\tiny IW}}
  - \partial^\beta \mathfrak{r}_0 - (\tilde{a}^{\mbox{\rm\tiny IW}}
  - a^{\mbox{\rm\tiny K}})\partial^\beta \zeta^{\mbox{\rm\tiny IW}} \\
 \qquad\quad{}
 + \underline{\rho}_1\bigl( {\bm l}_1^{\mbox{\rm\tiny IW}} ( \partial_t+\tilde{\bm{u}}_1^{\mbox{\rm\tiny IW}}\cdot\nabla )
  - {\bm l}_1^{\mbox{\rm\tiny K}} ( \partial_t+\bm{u}_1^{\mbox{\rm\tiny K}}\cdot\nabla ) \bigr)\cdot
  \partial^\beta \tilde{\bm{\phi}}_1^{\mbox{\rm\tiny IW}} \\
 \qquad\quad{}
 - \underline{\rho}_2\bigl( {\bm l}_2^{\mbox{\rm\tiny IW}} ( \partial_t+\tilde{\bm{u}}_2^{\mbox{\rm\tiny IW}}\cdot\nabla )
  - {\bm l}_2^{\mbox{\rm\tiny K}} ( \partial_t+\bm{u}_2^{\mbox{\rm\tiny K}}\cdot\nabla ) \bigr)\cdot
  \partial^\beta \tilde{\bm{\phi}}_2^{\mbox{\rm\tiny IW}}.
\end{cases}
\]
Here, $\bm{f}_{1,\beta}^{\mbox{\rm\tiny K}}$, $\bm{f}_{2,\beta}^{\mbox{\rm\tiny K}}$, and $f_{0,\beta}^{\mbox{\rm\tiny K}}$ 
are those in~\eqref{f1beta}--\eqref{f0beta} with $(\zeta,\bm{\phi}_1,\bm{\phi}_2)
=(\zeta^{\mbox{\rm\tiny K}},\bm{\phi}_1^{\mbox{\rm\tiny K}},\bm{\phi}_2^{\mbox{\rm\tiny K}})$, and so on. 
As we saw, all the assumptions in Proposition~\ref{L.energy-estimate} are satisfied, so that 
we have 
\[
\mathscr{E}(\partial^\beta \bm{U}^\mathrm{res}(t)) \lesssim 
 \int_0^t\mathscr{F}_\beta^\mathrm{res}(\tau)\mathrm{d}\tau,
\]
where $\bm{U}^\mathrm{res}:=(\zeta^\mathrm{res},\bm{\phi}_1^\mathrm{res},\bm{\phi}_2^\mathrm{res})^\mathrm{T}$, 
$\mathscr{E}$ is defined in~\eqref{def-E}, and 
\begin{align*}
\mathscr{F}_\beta^\mathrm{res}
&=  \|f_{0,\beta}^\mathrm{res}\|_{H^1}( \|\partial_t \zeta^\mathrm{res}\|_{H^{k-1}}
 + \|\zeta^\mathrm{res}\|_{H^k} ) \\
&\quad\;
 + \sum_{\ell=1,2}\underline{\rho}_\ell( \|\bm{f}_{\ell,\beta}^\mathrm{res}\|_{L^2} + \|\zeta^\mathrm{res}\|_{H^k} )
  ( \|\nabla\partial_t \bm{\phi}_\ell^\mathrm{res}\|_{H^{k-1}} + \|\nabla \bm{\phi}_\ell^\mathrm{res}\|_{H^k} ).
\end{align*}
In view of $\|(\zeta^{\mbox{\rm\tiny IW}},\zeta^{\mbox{\rm\tiny K}})\|_{H^m} \lesssim 1$, 
straightforward calculations yield 
\begin{align*}
\|\bm{f}_{\ell,\beta}^\mathrm{res}\|_{L^2}
&\lesssim
 (\|\partial_t\zeta^{\mbox{\rm\tiny IW}}\|_{H^{m-1}}+\|\tilde{\bm{u}}_\ell^{\mbox{\rm\tiny IW}}\|_{H^m})
  \|\zeta^\mathrm{res}\|_{H^k} \\
&\quad\;
 + \underline{h}_\ell(\|\nabla\tilde{\bm{\phi}}_\ell^{\mbox{\rm\tiny IW}}\|_{H^m}
  + \|\tilde{\bm{\phi}}_\ell^{\mbox{\rm\tiny IW} \, \prime}\|_{H^m}
  + (\underline{h}_\ell\delta)^{-2}\|\tilde{\bm{\phi}}_\ell^{\mbox{\rm\tiny IW} \, \prime}\|_{H^{m-1}})
  \|\zeta^\mathrm{res}\|_{H^k} \\
&\quad\;
  + \underline{h}_\ell(\|\nabla\tilde{\bm{\phi}}_\ell^\mathrm{res}\|_{H^k}
  + \|\tilde{\bm{\phi}}_\ell^{\mathrm{res} \, \prime}\|_{H^k}
  + (\underline{h}_\ell\delta)^{-2}\|\tilde{\bm{\phi}}_\ell^{\mathrm{res} \, \prime}\|_{H^{k-1}}) \\
&\quad\;
 + \|\partial_t\zeta^\mathrm{res}\|_{H^{k-1}}
 + \|\tilde{\bm{u}}_\ell^{\mbox{\rm\tiny IW}}-\bm{u}_\ell^{\mbox{\rm\tiny K}}\|_{H^k}
 + \underline{h}_\ell\|\bm{\mathfrak{r}}_\ell\|_{H^k}
\end{align*}
for $\ell=1,2$ and $\frac{n}{2}<k\leq m-1$. 
As for $f_{0,\beta}^\mathrm{res}$, we note the relation 
\begin{align*}
& \bigl\{ 
 \bigl( [\partial^\beta, \bm{l}_2^{\mbox{\rm\tiny K}}]
  - \bm{l}_2'(H_2^{\mbox{\rm\tiny K}})(\underline{h}_2^{-1}\partial^\beta\zeta^{\mbox{\rm\tiny K}}) \bigr) 
 - \bigl( [\partial^\beta, \bm{l}_2^{\mbox{\rm\tiny IW}}]
  - \bm{l}_2'(H_2^{\mbox{\rm\tiny IW}})(\underline{h}_2^{-1}\partial^\beta\zeta^{\mbox{\rm\tiny IW}}) \bigr) \bigr\}^\mathrm{T}
  \partial_t\tilde{\bm{\phi}}_2^{\mbox{\rm\tiny IW}} \\
&\makebox[1em]{}= \int_0^1
 \bigl\{ [\partial^\beta, \bm{l}_2'(sH^{\mbox{\rm\tiny IW}}+(1-s)H^{\mbox{\rm\tiny K}})] \\
&\makebox[4.5em]{}
  - \bm{l}_2''(sH^{\mbox{\rm\tiny IW}}+(1-s)H^{\mbox{\rm\tiny K}})
   \underline{h}_2^{-1}\partial^\beta( s\zeta^{\mbox{\rm\tiny IW}}+(1-s)\zeta^{\mbox{\rm\tiny K}} )
   \bigr\}^\mathrm{T}(\underline{h}_2^{-1}\zeta^\mathrm{res})\partial_t\tilde{\bm{\phi}}_2^{\mbox{\rm\tiny IW}} \\
&\makebox[1em]{}\phantom{= \int_0^1}
 + \bm{l}_2'(sH^{\mbox{\rm\tiny IW}}+(1-s)H^{\mbox{\rm\tiny K}})
  \bigl\{ [\partial^\beta,\underline{h}_2^{-1}\zeta^\mathrm{res}]
   - (\partial^\beta(\underline{h}_2^{-1}\zeta^\mathrm{res})) \bigr\}^\mathrm{T}\partial_t\tilde{\bm{\phi}}_2^{\mbox{\rm\tiny IW}}
  \mathrm{d}s.
\end{align*}
Therefore, straightforward calculations yield 
\begin{align*}
\|f_{0,\beta}^\mathrm{res}\|_{H^1}
&\lesssim \sum_{l=1,2}\underline{\rho}_\ell\bigl\{
 (\|\nabla\partial_t\tilde{\bm{\phi}}_\ell^{\mbox{\rm\tiny IW}}\|_{H^{m-2}}
  + \|\partial_t\tilde{\bm{\phi}}_\ell^{\mbox{\rm\tiny IW} \, \prime}\|_{H^{m-2}})\|\zeta^\mathrm{res}\|_{H^k} \\
&\quad\;
  + (\|\tilde{\bm{u}}_\ell^{\mbox{\rm\tiny IW}}\|_{H^m}+\|\bm{u}_\ell^{\mbox{\rm\tiny K}}\|_{H^m})
   (\|\nabla\tilde{\bm{\phi}}_\ell^{\mbox{\rm\tiny IW}}\|_{H^m}
  + \|\tilde{\bm{\phi}}_\ell^{\mbox{\rm\tiny IW} \, \prime}\|_{H^m})\|\zeta^\mathrm{res}\|_{H^k} \\
&\quad\;
 + (\underline{h}_\ell\delta)^{-2}\|w_\ell^{\mbox{\rm\tiny K}}\|_{H^m}
  \|\tilde{\bm{\phi}}_\ell^{\mbox{\rm\tiny IW} \, \prime}\|_{H^m}\|\zeta^\mathrm{res}\|_{H^k} 
 + \|\nabla\partial_t\tilde{\bm{\phi}}_\ell^\mathrm{res}\|_{H^{k-1}}
  + \|\partial_t\tilde{\bm{\phi}}_\ell^{\mathrm{res} \, \prime}\|_{H^{k-1}} \\
&\quad\;
 + \|\bm{u}_\ell^{\mbox{\rm\tiny K}}\|_{H^m} (\|\nabla\tilde{\bm{\phi}}_\ell^\mathrm{res}\|_{H^k}
  + \|\tilde{\bm{\phi}}_\ell^{\mathrm{res} \, \prime}\|_{H^k} )
 + (\underline{h}_\ell\delta)^{-2}\|w_\ell^{\mbox{\rm\tiny K}}\|_{H^m}
  \|\tilde{\bm{\phi}}_\ell^{\mathrm{res} \, \prime}\|_{H^k} \\
&\quad\;
 + (\|\bm{u}_\ell^{\mbox{\rm\tiny K}}\|_{H^m} + \|\tilde{\bm{u}}_\ell^{\mbox{\rm\tiny IW}}\|_{H^m}
  + \|\nabla\tilde{\bm{\phi}}_\ell^{\mbox{\rm\tiny IW}}\|_{H^m}
  + \|\tilde{\bm{\phi}}_\ell^{\mbox{\rm\tiny IW} \, \prime}\|_{H^m} )
   \|\tilde{\bm{u}}_\ell^{\mbox{\rm\tiny IW}}-\bm{u}_\ell^{\mbox{\rm\tiny K}}\|_{H^k} \\
&\quad\;
 + (\underline{h}_\ell\delta)^{-2}( \|w_\ell^{\mbox{\rm\tiny K}}\|_{H^m}
  + \|\tilde{w}_\ell^{\mbox{\rm\tiny IW}}\|_{H^m} + \|\tilde{\bm{\phi}}_\ell^{\mbox{\rm\tiny IW} \, \prime}\|_{H^m} )
  \|\tilde{w}_\ell^{\mbox{\rm\tiny IW}}-w_\ell^{\mbox{\rm\tiny K}}\|_{H^k} \bigr\} 
 + \|\mathfrak{r}_0\|_{H^{k+1}}
\end{align*}
for $\frac{n}{2}<k\leq m-2$. 
In view of the above estimates and~\eqref{estimate-uw-res}--\eqref{estimate-low-res} we obtain
$\mathscr{F}_\beta^\mathrm{res} \lesssim E_k^\mathrm{res} + \mathfrak{R}_k$ with 
$\mathfrak{R}_k := \|\mathfrak{r}_0\|_{H^{k+1}}^2 + \sum_{l=1,2}\underline{\rho}_\ell\underline{h}_\ell\|\bm{\mathfrak{r}}_\ell\|_{H^k}^2$. 
We note that the multi-index $\beta$ is assumed to satisfy $1\leq|\beta|\leq k$. 
As for the case $\beta=0$, we have $\frac{\mathrm{d}}{\mathrm{d}t}E_0^\mathrm{res}(t) \lesssim E_k^\mathrm{res}(t)$, 
hence ${E_0^\mathrm{res}(t) \lesssim \int_0^tE_k^\mathrm{res}(\tau)\mathrm{d}\tau}$. 
Summarizing the above estimates we obtain 
$E_k^\mathrm{res}(t) \lesssim \int_0^t(E_k^\mathrm{res}(\tau)+\mathfrak{R}_k(\tau))\mathrm{d}\tau$ for $\frac{n}{2}<k\leq m-2$. 
Putting ${k=m-4(N+1)}$ and applying Gronwall's inequality and~\eqref{error-estimate2} in Proposition~\ref{prop-consistency} 
we obtain $E_{m-4(N+1)}^\mathrm{res}(t) \lesssim (\underline{h}_1\delta)^{4N+2}+(\underline{h}_2\delta)^{4N+2}$ 
for $0\leq t\leq \min\{T,T^{\mbox{\tiny\rm IW}}\}$.

It remains to evaluate $\phi_\ell^{\mbox{\rm\tiny IW}}-\phi_\ell^{\mbox{\rm\tiny K}}$ for $\ell=1,2$. 
Let $(\bm{\phi}_1^{\mbox{\rm\tiny IW}},\bm{\phi}_2^{\mbox{\rm\tiny IW}})$ be the solution to~\eqref{Conditions} 
with $(\zeta,\phi_1,\phi_2)=(\zeta^{\mbox{\rm\tiny IW}},\phi_1^{\mbox{\rm\tiny IW}},\phi_2^{\mbox{\rm\tiny IW}})$. 
Then, we have 
$
\phi_\ell^{\mbox{\rm\tiny K}}-\phi_\ell^{\mbox{\rm\tiny IW}}
= \bm{l}_\ell^{\mbox{\rm\tiny K}}\cdot\bm{\phi}_\ell^\mathrm{res}
 + (\bm{l}_\ell^{\mbox{\rm\tiny K}}-\bm{l}_\ell^{\mbox{\rm\tiny IW}})\cdot\tilde{\bm{\phi}}_\ell^{\mbox{\rm\tiny IW}}
 + \bm{l}_\ell^{\mbox{\rm\tiny IW}}\cdot(\tilde{\bm{\phi}}_\ell^{\mbox{\rm\tiny IW}}-\bm{\phi}_\ell^{\mbox{\rm\tiny IW}}),
$
so that for any $0\leq k\leq m-1$ 
\begin{align*}
\|\nabla\phi_\ell^{\mbox{\rm\tiny K}}-\nabla\phi_\ell^{\mbox{\rm\tiny IW}}\|_{H^k}
&\lesssim \|\nabla\bm{\phi}_\ell^\mathrm{res}\|_{H^k}+\|\bm{\phi}_\ell^{\mathrm{res} \, \prime}\|_{H^k}
 + \underline{h}_\ell^{-1}\|\zeta^\mathrm{res}\|_{H^{k+1}}\|\tilde{\bm{\phi}}_\ell^{\mbox{\rm\tiny IW} \, \prime}\|_{H^m} \\
&\quad\;
 + \|\nabla(\tilde{\bm{\phi}}_\ell^{\mbox{\rm\tiny IW}}-\bm{\phi}_\ell^{\mbox{\rm\tiny IW}})\|_{H^k}
 + \|\tilde{\bm{\phi}}_\ell^{\mbox{\rm\tiny IW} \, \prime}-\bm{\phi}_\ell^{\mbox{\rm\tiny IW} \, \prime}\|_{H^k}. 
\end{align*}
Therefore, the previous result together with Lemma~\ref{L.estimate-phi-phi} implies
\[
\sum_{\ell=1,2}\underline{\rho}_\ell\underline{h}_\ell
 \|\nabla\phi_\ell^{\mbox{\rm\tiny K}}-\nabla\phi_\ell^{\mbox{\rm\tiny IW}}\|_{H^{m-(4N+5)}}^2
\lesssim (\underline{h}_1\delta)^{4N+2}+(\underline{h}_2\delta)^{4N+2}.
\]
This completes the proof of Theorem~\ref{theorem-justification}.

\section{Approximation of Hamiltonians; proof of Theorem~\ref{theorem-Hamiltonian}}\label{S.Hamiltonian}
As was shown in the companion paper~\cite[Theorem 8.4]{DucheneIguchi2020}, the Kakinuma model~\eqref{Kakinuma-dimensionless} 
enjoys a Hamiltonian structure analogous to the one exhibited on the full model for interfacial gravity waves
by T. B. Benjamin and T. J. Bridges in~\cite{BenjaminBridges1997}. 
In this section, we will prove Theorem~\ref{theorem-Hamiltonian}, which states that the Hamiltonian 
$\mathscr{H}^{\mbox{\rm\tiny K}}(\zeta,\phi)$ of the Kakinuma model defined in~\eqref{Hamiltonian-Kakinuma}
approximates
the Hamiltonian $\mathscr{H}^{\mbox{\rm\tiny IW}}(\zeta,\phi)$ of the full model defined in~\eqref{Hamiltonian-full-model}
with an error of order $O((\underline{h}_1\delta)^{4N+2}+(\underline{h}_2\delta)^{4N+2})$.

\subsection{Preliminary elliptic estimates}
We consider the following transmission problem 
\begin{equation}\label{BVP-bilayer}
\begin{cases}
 \nabla_X\cdot I_\delta^2\nabla_X\Phi_\ell = 0 & \mbox{in}\quad \Omega_\ell \qquad (\ell=1,2), \\
 \bm{n}\cdot I_\delta^2\nabla_X\Phi_\ell = 0 & \mbox{on}\quad \Sigma_\ell \qquad (\ell=1,2), \\
 \bm{n}\cdot I_\delta^2\nabla_X\Phi_2-\bm{n}\cdot I_\delta^2\nabla_X\Phi_1 = r_S & \mbox{on}\quad \Gamma, \\
 \underline{\rho}_2\Phi_2-\underline{\rho_1}\Phi_1 = \phi & \mbox{on}\quad \Gamma,
\end{cases}
\end{equation}
where the rigid-lid $\Sigma_1$ of the upper layer $\Omega_1$, the bottom $\Sigma_2$ of the lower layer $\Omega_2$, 
and the interface $\Gamma$ are defined by $z=\underline{h}_1$, $z=-\underline{h}_2+b(\bm{x})$, and $z=\zeta(\bm{x})$, 
respectively, $I_\delta:=\operatorname{diag}(1,\ldots,1,\delta^{-1})$, $\nabla_X:=(\nabla,\partial_z)^\mathrm{T}=(\partial_1,\ldots,\partial_n,\partial_z)$, and 
$\bm{n}$ is an upward normal vector, specifically, $\bm{n}=\bm{e}_z$ on $\Sigma_1$, $\bm{n}=(-\nabla b, 1)^\mathrm{T}$ on 
$\Sigma_2$, and $\bm{n}=(-\zeta,1)^\mathrm{T}$ on $\Gamma$.

\begin{lemma}\label{L.elliptic.bilayer}
Let $c,M$ be positive constants. 
There exists a positive constant $C$ such that
for any positive parameters 
$\underline{\rho}_1, \underline{\rho}_2, \underline{h}_1, \underline{h}_2, \delta$ satisfying 
$\underline{h}_1\delta, \underline{h}_2\delta \leq 1$, 
if $\zeta,b \in W^{1,\infty}$, $H_1 = 1 - \underline{h}_1^{-1}\zeta$, and 
$H_2 = 1 + \underline{h}_2^{-1}\zeta - \underline{h}_2^{-1}b$ satisfy  
\[
\begin{cases}
 \underline{h}_1^{-1}\|\zeta\|_{W^{1,\infty}} + \underline{h}_2^{-1}\|\zeta\|_{W^{1,\infty}}
  + \underline{h}_2^{-1}\|b\|_{W^{1,\infty}} \leq M, \\
 H_1(\bm{x}) \geq c, \quad H_2(\bm{x}) \geq c \quad\mbox{for}\quad \bm{x}\in\mathbf{R}^n,
\end{cases}
\]
then for any $(r_S,\phi)$ satisfying $\nabla\phi\in H^{-\frac12}$ and $(-\Delta)^{-\frac12}r_S\in H^\frac12$ 
there exists a solution $(\Phi_1,\Phi_2)$ to the transmission problem~\eqref{BVP-bilayer}. 
The solution is unique up to an additive constant of the form $(\underline{\rho}_2\mathcal{C},\underline{\rho}_1\mathcal{C})$ 
and satisfies 
\begin{align}\label{EEforTMP1}
&\sum_{\ell=1,2}\underline{\rho}_\ell \|I_\delta\nabla_X\Phi_\ell\|_{L^2(\Omega_\ell)}^2 \\
&\leq C\bigl( 
 \|((\underline{\rho}_1\Lambda_{2,0}+\underline{\rho}_2\Lambda_{1,0})^{-1}\Lambda_{1,0}\Lambda_{2,0})^\frac12\phi\|_{L^2}^2 
 + \underline{\rho}_1\underline{\rho}_2\|(\underline{\rho}_1\Lambda_{2,0}+\underline{\rho}_2\Lambda_{1,0})^{-\frac12}r_S\|_{L^2}^2 \bigr),
 \nonumber
\end{align}
where $\Lambda_{1,0}=\Lambda_1(0,\delta,\underline{h}_1)$ and $\Lambda_{2,0}=\Lambda_2(0,0,\delta,\underline{h}_2)$ 
are Dirichlet-to-Neumann maps in the case $\zeta(\bm{x})\equiv b(\bm{x}) \equiv0$. 
Particularly, if we further impose $\phi\in\mathring{H}^1$, $(-\Delta)^{-\frac12}r_S\in H^1$, the natural restrictions~\eqref{parameters}, 
and  $\underline{h}_\mathrm{min} \leq \underline{h}_1,\underline{h}_2$ with a positive constant $\underline{h}_\mathrm{min}$, 
then we have 
\begin{equation}\label{EEforTMP2}
\sum_{\ell=1,2}\underline{\rho}_\ell \|I_\delta\nabla_X\Phi_\ell\|_{L^2(\Omega_\ell)}^2 
\leq C\|\nabla\phi\|_{L^2}^2 + C\min_{\ell=1,2}\left\{
 \frac{\underline{\rho}_\ell}{\underline{h}_\ell}\|((-\Delta)^{-\frac12} + \underline{h}_\ell\delta)r_S\|_{L^2}^2\right\},
\end{equation}
where the constant $C$ depends also on $\underline{h}_\mathrm{min}$. 
\end{lemma}

\begin{proof}
The existence and the uniqueness of the solution is standard, so that we focus on deriving the uniform estimate of the solution. 
To this end, it is convenient to transform the water regions $\Omega_1$ and $\Omega_2$ into simple domains 
$\Omega_{1,0}=\mathbf{R}^n\times(0,\underline{h}_1)$ and $\Omega_{2,0}=\mathbf{R}^n\times (-\underline{h}_2,0)$ by using diffeomorphisms 
$\Theta_\ell(\bm{x},z)=(\bm{x},\theta_\ell(\bm{x},z)) \colon \Omega_{\ell,0}\to\Omega_\ell$ $(\ell=1,2)$, respectively, 
where $\theta_1(\bm{x},z)=(1-\underline{h}_1^{-1}\zeta(\bm{x}))z+\zeta(\bm{x})$ and 
$\theta_2(\bm{x},z)=(1+\underline{h}_2^{-1}(\zeta(\bm{x})-b(\bm{x})))z+\zeta(\bm{x})$. 
Put $\tilde{\Phi}_\ell=\Phi_\ell\circ\Theta_\ell$ $(\ell=1,2)$. 
Then, the transmission problem~\eqref{BVP-bilayer} is transformed into 
\[
\begin{cases}
 \nabla_X\cdot I_\delta\mathcal{P}_\ell I_\delta\nabla_X\tilde{\Phi}_\ell = 0 & \mbox{in}\quad \Omega_{\ell,0} \qquad (\ell=1,2), \\
 \bm{e}_z\cdot I_\delta\mathcal{P}_\ell I_\delta\nabla_X\tilde{\Phi}_\ell = 0 & \mbox{on}\quad \Sigma_{\ell,0} \qquad (\ell=1,2), \\
 \bm{e}_z\cdot I_\delta\mathcal{P}_2 I_\delta\nabla_X\tilde{\Phi}_2
  -\bm{e}_z\cdot I_\delta\mathcal{P}_1 I_\delta\nabla_X\tilde{\Phi}_1 = r_S & \mbox{on}\quad \Gamma_0, \\
 \underline{\rho}_2\tilde{\Phi}_2-\underline{\rho_1}\tilde{\Phi}_1 = \phi & \mbox{on}\quad \Gamma_0,
\end{cases}
\]
where $\Sigma_{1,0}$, $\Sigma_{2,0}$, and $\Gamma_0$ are represented as $z=\underline{h}_1$, $z=-\underline{h}_2$, 
and $z=0$, respectively, and 
\[
\mathcal{P}_\ell
:= \det\biggl( \frac{\partial\Theta_\ell}{\partial X} \biggr)I_\delta^{-1}\biggl( \frac{\partial\Theta_\ell}{\partial X} \biggr)^{-1}
 I_\delta^2\biggl( \biggl( \frac{\partial\Theta_\ell}{\partial X} \biggr)^{-1} \biggr)^\mathrm{T}I_\delta^{-1}
 \qquad (\ell=1,2).
\]
We note that $\|I_\delta\nabla_X\Phi_\ell\|_{L^2(\Omega_\ell)}
 \simeq \|I_\delta\nabla_X\tilde{\Phi}_\ell\|_{L^2(\Omega_{\ell,0})}$ $(\ell=1,2)$. 
Let $(\Psi_1,\Psi_2)$ be a solution to the transmission problem 
\[
\begin{cases}
 \nabla_X\cdot I_\delta^2\nabla_X\Psi_\ell = 0 & \mbox{in}\quad \Omega_{\ell,0} \qquad (\ell=1,2), \\
 \bm{e}_z\cdot I_\delta^2\nabla_X\Psi_\ell = 0 & \mbox{on}\quad \Sigma_{\ell,0} \qquad (\ell=1,2), \\
 \bm{e}_z\cdot I_\delta^2\nabla_X\Psi_2-\bm{e}_z\cdot I_\delta^2\nabla_X\Psi_1 = r_S & \mbox{on}\quad \Gamma_0, \\
 \underline{\rho}_2\Psi_2-\underline{\rho_1}\Psi_1 = \phi & \mbox{on}\quad \Gamma_0,
\end{cases}
\]
and put $\Phi_\ell^\mathrm{res}=\tilde{\Phi}_\ell-\Psi_\ell$ $(\ell=1,2)$. 
Then, we can decompose 
\[
|I_\delta\nabla_X\Phi_\ell^\mathrm{res}|^2
 - I_\delta\nabla_X\Phi_\ell^\mathrm{res}\cdot(I-\mathcal{P}_\ell)I_\delta\nabla_X\tilde{\Phi}_\ell
= \nabla_X\Phi_\ell^\mathrm{res}\cdot\{ (I_\delta\mathcal{P}_\ell I_\delta\nabla_X\tilde{\Phi}_\ell
 - I_\delta^2\nabla_X\Psi_\ell) \}
\]
for $\ell=1,2$ and $\underline{\rho}_1\Phi_1^\mathrm{res}= \underline{\rho}_2\Phi_2^\mathrm{res}$ on $z=0$. 
Therefore, denoting the unit outward normal vector to $\partial\Omega_{\ell,0}$ by $N_\ell$ $(\ell=1,2)$ we have 
\begin{align*}
& \sum_{\ell=1,2}\underline{\rho}_\ell\int_{\Omega_{\ell,0}} \bigl(
 |I_\delta\nabla_X\Phi_\ell^\mathrm{res}|^2
 - I_\delta\nabla_X\Phi_\ell^\mathrm{res}\cdot(I-\mathcal{P}_\ell)I_\delta\nabla_X\tilde{\Phi}_\ell \bigr)\mathrm{d}X \\
&\quad= \sum_{\ell=1,2}\int_{\partial\Omega_{\ell,0}}
 \underline{\rho}_\ell\Phi_\ell^\mathrm{res}(N_\ell\cdot I_\delta\mathcal{P}_\ell I_\delta\nabla_X\tilde{\Phi}_\ell
  - N_\ell\cdot I_\delta^2\nabla_X\Psi_\ell) \mathrm{d}S \\
&\quad= \sum_{\ell=1,2}\int_{\mathbf{R}^n}
 \underline{\rho}_1\bigl[ \Phi_1^\mathrm{res}\{ (\bm{e}_z\cdot I_\delta\mathcal{P}_2 I_\delta\nabla_X\tilde{\Phi}_2
   - \bm{e}_z\cdot I_\delta^2\nabla_X\Psi_2) \\
&\makebox[7em]{}
  - (\bm{e}_z\cdot I_\delta\mathcal{P}_1 I_\delta\nabla_X\tilde{\Phi}_1
   - \bm{e}_z\cdot I_\delta^2\nabla_X\Psi_1) \} \bigr]\bigr|_{z=0} \mathrm{d}\bm{x} \\
&\quad=0,
\end{align*}
so that we obtain 
\[
\sum_{\ell=1,2}\underline{\rho}_\ell\int_{\Omega_{\ell,0}}
 |I_\delta\nabla_X\Phi_\ell^\mathrm{res}|^2 \mathrm{d}X
= \sum_{\ell=1,2}\underline{\rho}_\ell\int_{\Omega_{\ell,0}}
 I_\delta\nabla_X\Phi_\ell^\mathrm{res}\cdot(I-\mathcal{P}_\ell)I_\delta\nabla_X\tilde{\Phi}_\ell \mathrm{d}X.
\]
Similarly, in view of the decomposition
\begin{align*}
 I_\delta\nabla_X\Phi_\ell^\mathrm{res} \cdot \mathcal{P}_\ell I_\delta\nabla_X\Phi_\ell^\mathrm{res}
 - I_\delta\nabla_X\Phi_\ell^\mathrm{res}\cdot(I-\mathcal{P}_\ell)I_\delta\nabla_X\Psi_\ell \\
= \nabla_X\Phi_\ell^\mathrm{res}\cdot\{ (I_\delta\mathcal{P}_\ell I_\delta\nabla_X\tilde{\Phi}_\ell
 - I_\delta^2\nabla_X\Psi_\ell) \}
\end{align*}
for $\ell=1,2$, we obtain 
\[
\sum_{\ell=1,2}\underline{\rho}_\ell\int_{\Omega_{\ell,0}}
 I_\delta\nabla_X\Phi_\ell^\mathrm{res} \cdot \mathcal{P}_\ell I_\delta\nabla_X\Phi_\ell^\mathrm{res}
= \sum_{\ell=1,2}\underline{\rho}_\ell\int_{\Omega_{\ell,0}}
 I_\delta\nabla_X\Phi_\ell^\mathrm{res}\cdot(I-\mathcal{P}_\ell)I_\delta\nabla_X\Psi_\ell \mathrm{d}X.
\]
It follows from these two identities that 
\[
\sum_{\ell=1,2}\underline{\rho}_\ell\|I_\delta\nabla_X\Phi_\ell^\mathrm{res}\|_{L^2(\Omega_{\ell,0})}^2
\lesssim \min\biggl\{
 \sum_{\ell=1,2}\underline{\rho}_\ell\|I_\delta\nabla_X\tilde{\Phi}_\ell\|_{L^2(\Omega_{\ell,0})}^2,
 \sum_{\ell=1,2}\underline{\rho}_\ell\|I_\delta\nabla_X\Psi_\ell\|_{L^2(\Omega_{\ell,0})}^2
 \biggr\},
\]
which yields the equivalence 
\[
\sum_{\ell=1,2}\underline{\rho}_\ell\|I_\delta\nabla_X\tilde{\Phi}_\ell\|_{L^2(\Omega_{\ell,0})}^2
\simeq \sum_{\ell=1,2}\underline{\rho}_\ell\|I_\delta\nabla_X\Psi_\ell\|_{L^2(\Omega_{\ell,0})}^2.
\]
Therefore, it is sufficient to evaluate the right-hand side of the above equation. 
In other words, the evaluation is reduced to the simple case $\zeta(\bm{x})\equiv b(\bm{x})\equiv 0$.

Putting $\psi_\ell = \Psi_\ell|_{z=0}$ $(\ell=1,2)$, we see that 
\[
\sum_{\ell=1,2}\underline{\rho}_\ell\|I_\delta\nabla_X\Psi_\ell\|_{L^2(\Omega_{\ell,0})}^2
= \sum_{\ell=1,2}\underline{\rho}_\ell(\Lambda_{\ell,0}\psi_\ell,\psi_\ell)_{L^2}
\]
and that 
\[
\begin{cases}
 \Lambda_{1,0}\psi_1+\Lambda_{2,0}\psi_2 = r_S, \\
 \underline{\rho}_2\psi_2-\underline{\rho}_1\psi_1 = \phi.
\end{cases}
\]
Particularly, we have 
\[
\begin{pmatrix} \psi_1 \\ \psi_2 \end{pmatrix}
= (\underline{\rho}_1\Lambda_{2,0}+\underline{\rho}_2\Lambda_{1,0})^{-1}
 \begin{pmatrix}
  -\Lambda_{2,0}\phi+\underline{\rho}_2r_S \\
   \Lambda_{1,0}\phi+\underline{\rho}_1r_S
 \end{pmatrix}.
\]
Therefore, 
\[
\sum_{\ell=1,2}\underline{\rho}_\ell\|I_\delta\nabla_X\Psi_\ell\|_{L^2(\Omega_{\ell,0})}^2
=
\begin{cases}
 \|((\underline{\rho}_1\Lambda_{2,0}+\underline{\rho}_2\Lambda_{1,0})^{-1}\Lambda_{1,0}\Lambda_{2,0})^\frac12\phi\|_{L^2}^2 
  &\mbox{if } r_S=0, \\
 \underline{\rho}_1\underline{\rho}_2\|(\underline{\rho}_1\Lambda_{2,0}+\underline{\rho}_2\Lambda_{1,0})^{-\frac12}r_S\|_{L^2}^2 
  &\mbox{if } \phi=0. \\
\end{cases}
\]
Hence, by the linearity of the problem we obtain~\eqref{EEforTMP1}.

Finally, in order to show~\eqref{EEforTMP2} it is sufficient to evaluate the symbols of the Fourier multipliers 
$(\underline{\rho}_1\Lambda_{2,0}+\underline{\rho}_2\Lambda_{1,0})^{-1}\Lambda_{1,0}\Lambda_{2,0}$ 
and $\underline{\rho}_1\underline{\rho}_2(\underline{\rho}_1\Lambda_{2,0}+\underline{\rho}_2\Lambda_{1,0})^{-1}$. 
We recall that the symbol of the Dirichlet-to-Neumann map $\Lambda_{\ell,0}$ is given by 
$\sigma(\Lambda_{\ell,0})=\delta^{-1}|\bm{\xi}|\tanh(\underline{h}_\ell\delta|\bm{\xi}|)$ for $\ell=1,2$. 
In view of $0\leq\tanh\xi\leq\xi$ for $\xi\geq0$, we have 
\begin{align*}
\sigma((\underline{\rho}_1\Lambda_{2,0}+\underline{\rho}_2\Lambda_{1,0})^{-1}\Lambda_{1,0}\Lambda_{2,0})
&\leq \min\biggl\{ \frac{\sigma({\Lambda}_{1,0})}{\underline{\rho}_1}, \frac{\sigma({\Lambda}_{2,0})}{\underline{\rho}_2} \biggr\} \\
&\leq \min\biggl\{ \frac{\underline{h}_1}{\underline{\rho}_1}, \frac{\underline{h}_2}{\underline{\rho}_2} \biggr\}|\bm{\xi}|^2 \\
&\leq 2|\bm{\xi}|^2,
\end{align*}
where we used~\eqref{parameter-relation}. 
In view of $\tanh\xi \simeq (1+\xi)^{-1}\xi$ for $\xi\geq0$ and the relation~\eqref{parameters}, we have 
\begin{align*}
\sigma(\underline{\rho}_1\underline{\rho}_2(\underline{\rho}_1\Lambda_{2,0}+\underline{\rho}_2\Lambda_{1,0})^{-1})
&\simeq \frac{\underline{\rho}_1\underline{\rho}_2}{\underline{h}_1\underline{h}_2}
 \frac{(1+\underline{h}_1\delta|\bm{\xi}|)(1+\underline{h}_2\delta|\bm{\xi}|)}{(1+\delta|\bm{\xi}|)|\bm{\xi}|^2} \\
&\lesssim \min\biggl\{ 
 \frac{\underline{\rho}_1}{\underline{h}_1}\underline{\rho}_2\frac{1+\underline{h}_1\delta|\bm{\xi}|}{|\bm{\xi}|^2}, 
 \frac{\underline{\rho}_2}{\underline{h}_2}\underline{\rho}_1\frac{1+\underline{h}_2\delta|\bm{\xi}|}{|\bm{\xi}|^2} \biggr\} \\
&\lesssim \min\biggl\{ 
 \frac{\underline{\rho}_1}{\underline{h}_1}(|\bm{\xi}|^{-1}+\underline{h}_1\delta)^2, 
 \frac{\underline{\rho}_2}{\underline{h}_2}(|\bm{\xi}|^{-1}+\underline{h}_2\delta)^2 \biggr\},
\end{align*}
where we used $1\lesssim \underline{h}_1,\underline{h}_2$. 
These estimates imply~\eqref{EEforTMP2}. 
The proof is complete. 
\end{proof}

\subsection{Completion of the proof of Theorem~\ref{theorem-Hamiltonian}}
Now we are ready to prove Theorem~\ref{theorem-Hamiltonian}. 
We recall the definitions~\eqref{def-l1l2} of $\bm{l}_1(H_1)$, $\bm{l}_2(H_2)$ and~\eqref{def-calL} of 
the operators $\mathcal{L}_{1,i}(H_1,\delta,\underline{h}_1)$ and $\mathcal{L}_{2,i}(H_2,b,\delta,\underline{h}_2)$. 
These depend on $N$, so that we denote them by $\bm{l}_1^{(N)}(H_1)$, $\bm{l}_2^{(N)}(H_2)$ and 
$\mathcal{L}_{1,i}^{(N)}(H_1,\delta,\underline{h}_1)$ and $\mathcal{L}_{2,i}^{(N)}(H_2,b,\delta,\underline{h}_2)$, 
respectively, in the following argument. 
For given $(\zeta,\phi)$, let $\Phi$ be the solution to the transmission problem~\eqref{BVP-bilayer} with $r_S=0$ and let 
$(\bm{\phi}_1,\bm{\phi}_2)$ and $(\tilde{\bm{\phi}}_1,\tilde{\bm{\phi}}_2)$ be the solutions to the problems 
\[
\begin{cases}
 \mathcal{L}_{1,i}^{(N)}(H_1,\delta,\underline{h}_1)\bm{\phi}_1=0 \quad\mbox{for}\quad i=1,2,\ldots,N, \\
 \mathcal{L}_{2,i}^{(N)}(H_2,b,\delta,\underline{h}_2)\bm{\phi}_2=0 \quad\mbox{for}\quad i=1,2,\ldots,N^*, \\
 \underline{h}_1\mathcal{L}_{1,0}^{(N)}(H_1,\delta,\underline{h}_1) \bm{\phi}_1
  + \underline{h}_2\mathcal{L}_{2,0}^{(N)}(H_2,b,\delta,\underline{h}_2) \bm{\phi}_2 = 0, \\
 \underline{\rho}_2\bm{l}_2^{(N)}(H_2) \cdot \bm{\phi}_2
  - \underline{\rho}_1\bm{l}_1^{(N)}(H_1) \cdot \bm{\phi}_1 = \phi
\end{cases}
\]
and 
\[
\begin{cases}
 \mathcal{L}_{1,i}^{(2N+2)}(H_1,\delta,\underline{h}_1)\tilde{\bm{\phi}}_1=0 \quad\mbox{for}\quad i=1,2,\ldots,2N+2, \\
 \mathcal{L}_{2,i}^{(2N+2)}(H_2,b,\delta,\underline{h}_2)\tilde{\bm{\phi}}_2=0 \quad\mbox{for}\quad i=1,2,\ldots,2N^*+2, \\
 \underline{h}_1\mathcal{L}_{1,0}^{(2N+2)}(H_1,\delta,\underline{h}_1) \tilde{\bm{\phi}}_1
  + \underline{h}_2\mathcal{L}_{2,0}^{(2N+2)}(H_2,b,\delta,\underline{h}_2) \tilde{\bm{\phi}}_2 = 0, \\
 \underline{\rho}_2\bm{l}_2^{(2N+2)}(H_2) \cdot \tilde{\bm{\phi}}_2
  - \underline{\rho}_1\bm{l}_1^{(2N+2)}(H_1) \cdot \tilde{\bm{\phi}}_1 = \phi,
\end{cases}
\]
respectively, and define $(\Phi_1^\mathrm{app},\Phi_2^\mathrm{app})$ and 
$(\tilde{\Phi}_1^\mathrm{app},\tilde{\Phi}_2^\mathrm{app})$ by~\eqref{Approximation-nondim} and 
\[
\begin{cases}
 \displaystyle
  \tilde{\Phi}_1^\mathrm{app}(\bm{x},z) = \sum_{i=0}^{2N+2} (1-\underline{h}_1^{-1}z)^{2i}\tilde{\phi}_{1,i}(\bm{x}), \\[2.5ex]
 \displaystyle
  \tilde{\Phi}_2^\mathrm{app}(\bm{x},z) = \sum_{i=0}^{2N^*+2} (1+\underline{h}_2^{-1}(z-b(\bm{x})))^{p_i}\tilde{\phi}_{2,i}(\bm{x}),
\end{cases}
\]
respectively. 
Then, by the definitions of the Hamiltonian functionals $\mathscr{H}^{\mbox{\rm\tiny IW}}(\zeta,\phi)$ and 
$\mathscr{H}^{\mbox{\rm\tiny K}}(\zeta,\phi)$ given in Section~\ref{S.def-hamiltonian}, we have 
\begin{align*}
2(\mathscr{H}^{\mbox{\rm\tiny IW}}(\zeta,\phi)-\mathscr{H}^{\mbox{\rm\tiny K}}(\zeta,\phi))
&= \sum_{\ell=1,2}\underline{\rho}_\ell\int_{\Omega_\ell}
 (|I_\delta\nabla_X\Phi_\ell|^2 - |I_\delta\nabla_X\Phi_\ell^\mathrm{app}|^2)\mathrm{d}X \\
&= \sum_{\ell=1,2}\underline{\rho}_\ell\int_{\Omega_\ell}
 (|I_\delta\nabla_X\Phi_\ell|^2 - |I_\delta\nabla_X\tilde{\Phi}_\ell^\mathrm{app}|^2)\mathrm{d}X \\
&\quad\;
 + \sum_{\ell=1,2}\underline{\rho}_\ell\int_{\Omega_\ell}
 (|I_\delta\nabla_X\tilde{\Phi}_\ell^\mathrm{app}|^2 - |I_\delta\nabla_X\Phi_\ell^\mathrm{app}|^2)\mathrm{d}X \\
&=: I_1+I_2.
\end{align*}
We will evaluate $I_1$ and $I_2$, separately.

In order to evaluate $I_1$, we put $\Phi_\ell^\mathrm{res}=\Phi_\ell-\tilde{\Phi}_\ell^\mathrm{app}$ $(\ell=1,2)$, so that 
\begin{align*}
|I_1|
&= \biggl| \sum_{\ell=1,2}\underline{\rho}_\ell\int_{\Omega_\ell}
 I_\delta\nabla_X\Phi_\ell^\mathrm{res}\cdot I_\delta\nabla_X(\Phi_\ell+\tilde{\Phi}_\ell^\mathrm{app})\mathrm{d}X \biggr| \\
&\leq \sum_{\ell=1,2}\underline{\rho}_\ell \|I_\delta\nabla_X\Phi_\ell^\mathrm{res}\|_{L^2(\Omega_\ell)}
 ( \|I_\delta\nabla_X\Phi_\ell\|_{L^2(\Omega_\ell)} + \|I_\delta\nabla_X\tilde{\Phi}_\ell^\mathrm{app}\|_{L^2(\Omega_\ell)} ).
\end{align*}
It follows from Lemma~\ref{L.elliptic.bilayer} that 
$\sum_{\ell=1,2}\underline{\rho}_\ell\|I_\delta\nabla_X\Phi_\ell\|_{L^2(\Omega_\ell)}^2 \lesssim \|\nabla\phi\|_{L^2}^2$. 
We see also that 
\begin{align*}
\sum_{\ell=1,2}\underline{\rho}_\ell \|I_\delta\nabla_X\tilde{\Phi}_\ell^\mathrm{app}\|_{L^2(\Omega_\ell)}^2
&= \sum_{\ell=1,2}\underline{\rho}_\ell\underline{h}_\ell(L_\ell^{(2N+2)}\tilde{\bm{\phi}}_\ell, \tilde{\bm{\phi}}_\ell)_{L^2} \\
&\lesssim \sum_{\ell=1,2}\underline{\rho}_\ell\underline{h}_\ell
 ( \|\nabla\tilde{\bm{\phi}}_\ell\|_{L^2}^2+(\underline{h}_\ell\delta)^{-2}\|\tilde{\bm{\phi}}'\|_{L^2}^2 ) \\
&\lesssim \|\nabla\phi\|_{L^2}^2,
\end{align*}
where we used Lemma~\ref{L.elliptic} and~\eqref{parameter-relation}. 
In order to evaluate $\|I_\delta\nabla_X\Phi_\ell^\mathrm{res}\|_{L^2(\Omega_\ell)}$, we first notice that 
$(\Phi_1^\mathrm{res},\Phi_2^\mathrm{res})$ satisfy 
\[
\begin{cases}
 \nabla_X\cdot I_\delta^2\nabla_X\Phi_\ell^\mathrm{res} = R_\ell & \mbox{in}\quad \Omega_\ell \qquad (\ell=1,2), \\
 \bm{n}\cdot I_\delta^2\nabla_X\Phi_1^\mathrm{res} = 0 & \mbox{on}\quad \Sigma_1, \\
 \bm{n}\cdot I_\delta^2\nabla_X\Phi_2^\mathrm{res} = \underline{h}_2 r_B & \mbox{on}\quad \Sigma_2, \\
 \underline{\rho}_2\Phi_2^\mathrm{res} - \underline{\rho}_1\Phi_1^\mathrm{res} = 0 &\mbox{on}\quad \Gamma, \\
 \Lambda_1[\Phi_1^\mathrm{res}|_{z=\zeta}] + \Lambda_2[\Phi_2^\mathrm{res}|_{z=\zeta}] = r_S,
\end{cases}
\]
where 
\[
\begin{cases}
 R_\ell = -\nabla_X\cdot I_\delta^2\nabla_X\tilde{\Phi}_\ell^\mathrm{app} \qquad (\ell=1,2), \\
 r_B = -\underline{h}_2^{-1}(-\nabla b, 1)^\mathrm{T}\cdot I_\delta^2(\nabla_X\tilde{\Phi}_\ell^\mathrm{app})|_{z=-\underline{h}_2+b}, \\
 \displaystyle
 r_S = \sum_{\ell=1,2}(\underline{h}_\ell\Lambda_\ell^{(2N+2)}-\Lambda_\ell)[\tilde{\Phi}_\ell^\mathrm{app}|_{z=\zeta}].\\
\end{cases}
\]
Here, we note that $R_\ell$ $(\ell=1,2)$ can be written the form 
\[
\begin{cases}
 \displaystyle
  R_1(\bm{x},z) = \sum_{i=0}^{2N+2} (1-\underline{h}_1^{-1}z)^{2i}r_{1,i}(\bm{x}), \\[2.5ex]
 \displaystyle
  R_2(\bm{x},z) = \sum_{i=0}^{2N^*+2} (1+\underline{h}_2^{-1}(z-b(\bm{x})))^{p_i}r_{2,i}(\bm{x}). 
\end{cases}
\]
Estimates for the residuals $(r_{1,0},r_{1,1},\ldots,r_{1,2N+2})$, $(r_{2,0},r_{2,0},\ldots,r_{2,2N^*+2})$, and $r_B$ 
were given in~\cite[Lemmas 6.4 and 6.9]{Iguchi2018-2} and their proofs. 
In fact, we have 
\begin{align*}
\|(r_{1,0},r_{1,1},\ldots,r_{1,2N+2})\|_{L^2}
&\lesssim \|\tilde{\phi}_{1,2N+2}\|_{H^2} \\
&\lesssim (\underline{h}_1\delta)^{4N+2}\|\nabla\tilde{\bm{\phi}}_1\|_{H^{4N+3}}
\end{align*}
and 
\begin{align*}
\|(r_{2,0},r_{2,1},\ldots,r_{2,2N^*+2})\|_{L^2} + \|r_B\|_{L^2}
&\lesssim \|(\tilde{\phi}_{2,2N^*+1},\tilde{\phi}_{2,2N^*+2})\|_{H^2} \\
&\lesssim (\underline{h}_2\delta)^{4N+2}(\|\nabla\tilde{\bm{\phi}}_2\|_{H^{4N+3}}+\|\tilde{\bm{\phi}}_2'\|_{H^{4N+3}}).
\end{align*}
We decompose $\Phi_\ell^\mathrm{res} = \Phi_\ell^\mathrm{res,1} + \Phi_\ell^\mathrm{res,2}$, where 
$(\Phi_1^\mathrm{res,1},\Phi_2 ^\mathrm{res,1})$ is a unique solution to the problem 
\[
\begin{cases}
 \nabla_X\cdot I_\delta^2\nabla_X\Phi_\ell^\mathrm{res,1} = R_\ell & \mbox{in}\quad \Omega_\ell \qquad (\ell=1,2), \\
 \bm{n}\cdot I_\delta^2\nabla_X\Phi_1^\mathrm{res,1} = 0 & \mbox{on}\quad \Sigma_1, \\
 \bm{n}\cdot I_\delta^2\nabla_X\Phi_2^\mathrm{res,1} = \underline{h}_2r_B & \mbox{on}\quad \Sigma_2, \\
 \Phi_\ell^\mathrm{res,1} =0 &\mbox{on}\quad \Gamma \qquad (\ell=1,2), 
\end{cases}
\]
so that $(\Phi_1^\mathrm{res,2},\Phi_2^\mathrm{res,2})$ satisfy 
\begin{equation}\label{def-phi-res-2}
\begin{cases}
 \nabla_X\cdot I_\delta^2\nabla_X\Phi_\ell^\mathrm{res,2} = 0 & \mbox{in}\quad \Omega_\ell \qquad (\ell=1,2), \\
 \bm{n}\cdot I_\delta^2\nabla_X\Phi_\ell^\mathrm{res,2} = 0 & \mbox{on}\quad \Sigma_\ell \qquad (\ell=1,2), \\
 \bm{n}\cdot I_\delta^2\nabla_X\Phi_2^\mathrm{res,2}
  - \bm{n}\cdot I_\delta^2\nabla_X\Phi_1^\mathrm{res,2} = r_S &\mbox{on}\quad \Gamma, \\
 \underline{\rho}_2\Phi_2^\mathrm{res,2} - \underline{\rho}_1\Phi_1^\mathrm{res,2} = 0 &\mbox{on}\quad \Gamma, 
\end{cases}
\end{equation}
where we used the relations $\Lambda_1[\Phi_1^\mathrm{res,2}|_{z=\zeta}] = - \bm{n}\cdot I_\delta^2\nabla_X\Phi_1^\mathrm{res,2}$ 
and $\Lambda_2[\Phi_2^\mathrm{res,2}|_{z=\zeta}] = \bm{n}\cdot I_\delta^2\nabla_X\Phi_2^\mathrm{res,2}$ on $\Gamma$. 
It is easy to see that 
\begin{align*}
\|I_\delta\nabla_X\Phi_1^\mathrm{res,1}\|_{L^2(\Omega_1)}^2
&\lesssim (\underline{h}_1\delta)^2\|R_1\|_{L^2(\Omega_1)}^2 \\
&\lesssim \underline{h}_1(\underline{h}_1\delta)^2\|(r_{1,0},r_{1,1},\ldots,r_{1,2N+2})\|_{L^2} \\
&\lesssim \underline{h}_1(\underline{h}_1\delta)^{2(4N+3)}\|\nabla\tilde{\bm{\phi}}_1\|_{H^{4N+3}}^2
\end{align*}
and that 
\begin{align*}
\|I_\delta\nabla_X\Phi_2^\mathrm{res,1}\|_{L^2(\Omega_2)}^2
&\lesssim \underline{h}_2(\underline{h}_2\delta)^2( \underline{h}_2^{-1}\|R_2\|_{L^2(\Omega_2)}^2 + \|r_B\|_{L^2}^2 ) \\
&\lesssim \underline{h}_2(\underline{h}_2\delta)^2( \|(r_{2,0},r_{2,1},\ldots,r_{2,2N^*+2})\|_{L^2} + \|r_B\|_{L^2}^2 ) \\
&\lesssim \underline{h}_2(\underline{h}_2\delta)^{2(4N+3)}
 ( \|\nabla\tilde{\bm{\phi}}_2\|_{H^{4N+3}}+\|\tilde{\bm{\phi}}_2'\|_{H^{4N+3}} ).
\end{align*}
Therefore, by Lemma~\ref{L.elliptic} together with~\eqref{parameter-relation} we have 
\[
\sum_{\ell=1,2}\underline{\rho}_\ell\|I_\delta\nabla_X\Phi_\ell^\mathrm{res,1}\|_{L^2(\Omega_\ell)}^2
\lesssim ((\underline{h}_1\delta)^{4N+3} + (\underline{h}_2\delta)^{4N+3})^2\|\nabla\phi\|_{H^{4N+3}}^2.
\]
On the other hand, it follows from Lemmas~\ref{L.elliptic.bilayer},~\ref{L.res1},~\ref{L.L-estimate-D2N-3}, 
and~\ref{L.elliptic} that 
\begin{align*}
\sum_{\ell=1,2}\underline{\rho}_\ell\|I_\delta\nabla_X\Phi_\ell^\mathrm{res,2}\|_{L^2(\Omega_\ell)}^2
&\lesssim \min_{\ell=1,2} \frac{\underline{\rho}_\ell}{\underline{h}_\ell}
 \|((-\Delta)^{-\frac12} + \underline{h}_\ell\delta)r_S\|_{L^2}^2 \\
&\lesssim \sum_{\ell=1,2} \frac{\underline{\rho}_\ell}{\underline{h}_\ell}
 \|((-\Delta)^{-\frac12} + \underline{h}_\ell\delta)(\underline{h}_\ell\Lambda_\ell^{(2N+2)}-\Lambda_\ell)
  [\tilde{\Phi}_\ell^\mathrm{app}|_{z=\zeta}]\|_{L^2}^2 \\
&\lesssim \sum_{\ell=1,2} \underline{\rho}_\ell\underline{h}_\ell
 (\underline{h}_\ell\delta)^{2(4N+2)}\|\nabla(\tilde{\Phi}_\ell^\mathrm{app}|_{z=\zeta})\|_{H^{4N+3}}^2 \\
&\lesssim \sum_{\ell=1,2} \underline{\rho}_\ell\underline{h}_\ell
 (\underline{h}_\ell\delta)^{2(4N+2)} (\|\nabla\tilde{\bm{\phi}}_\ell\|_{H^{4N+3}}^2 + \|\tilde{\bm{\phi}}_\ell'\|_{H^{4N+3}}^2 ) \\
&\lesssim ((\underline{h}_1\delta)^{4N+2} + (\underline{h}_2\delta)^{4N+2})^2\|\nabla\phi\|_{H^{4N+3}}^2.
\end{align*}
Summarizing the above estimates, we obtain 
$|I_1| \lesssim ((\underline{h}_1\delta)^{4N+2} + (\underline{h}_2\delta)^{4N+2})\|\nabla\phi\|_{H^{4N+3}}\|\nabla\phi\|_{L^2}$.

We proceed to evaluate $I_2$, which can be written as 
\begin{align*}
I_2 
&= \sum_{\ell=1,2}\underline{\rho}_\ell\underline{h}_\ell (L_\ell^{(2N+2)}\tilde{\bm{\phi}}_\ell,\tilde{\bm{\phi}}_\ell)_{L^2}
 - \sum_{\ell=1,2}\underline{\rho}_\ell\underline{h}_\ell (L_\ell^{(N)}\bm{\phi}_\ell,\bm{\phi}_\ell)_{L^2} \\
&=: I_{2,1}+I_{2,2}.
\end{align*}
In view of~\eqref{expression-L_k}, we see that 
\begin{align*}
I_{2,1}
&= \underline{\rho}_1\underline{h}_1(\mathcal{L}_{1,0}^{(2N+2)}\tilde{\bm{\phi}}_1,\bm{l}_1^{(2N+2)}\cdot\tilde{\bm{\phi}}_1)_{L^2}
 + \underline{\rho}_2\underline{h}_2(\mathcal{L}_{2,0}^{(2N+2)}\tilde{\bm{\phi}}_2,\bm{l}_2^{(2N+2)}\cdot\tilde{\bm{\phi}}_2)_{L^2} \\
&= (\underline{h}_2\mathcal{L}_{2,0}^{(2N+2)}\tilde{\bm{\phi}}_2,
 \underline{\rho}_2\bm{l}_2^{(2N+2)}\cdot\tilde{\bm{\phi}}_2 - \underline{\rho}_1\bm{l}_1^{(2N+2)}\cdot\tilde{\bm{\phi}}_1)_{L^2} \\
&= (\underline{h}_2\mathcal{L}_{2,0}^{(2N+2)}\tilde{\bm{\phi}}_2,\phi)_{L^2} \\
&= (\underline{h}_2\mathcal{L}_{2,0}^{(2N+2)}\tilde{\bm{\phi}}_2,
 \underline{\rho}_2\bm{l}_2^{(N)}\cdot\bm{\phi}_2 - \underline{\rho}_1\bm{l}_1^{(N)}\cdot\bm{\phi}_1)_{L^2} \\
&= \underline{\rho}_1\underline{h}_1(\mathcal{L}_{1,0}^{(2N+2)}\tilde{\bm{\phi}}_1,\bm{l}_1^{(N)}\cdot\bm{\phi}_1)_{L^2}
 + \underline{\rho}_2\underline{h}_2(\mathcal{L}_{2,0}^{(2N+2)}\tilde{\bm{\phi}}_2,\bm{l}_2^{(N)}\cdot\bm{\phi}_2)_{L^2} \\
&= \underline{\rho}_1\underline{h}_1 \sum_{i=0}^N\sum_{j=0}^{2N+2} (L_{1,ij}\tilde{\phi}_{1,j},\phi_{1,i})_{L^2}
 + \underline{\rho}_2\underline{h}_2 \sum_{i=0}^{N^*}\sum_{j=0}^{2N^*+2} (L_{2,ij}\tilde{\phi}_{2,j},\phi_{2,i})_{L^2} \\
&= \underline{\rho}_1\underline{h}_1 \sum_{i=0}^N\sum_{j=0}^{2N+2} (L_{1,ji}\phi_{1,i},\tilde{\phi}_{1,j})_{L^2}
 + \underline{\rho}_2\underline{h}_2 \sum_{i=0}^{N^*}\sum_{j=0}^{2N^*+2} (L_{2,ji}\phi_{2,i},\tilde{\phi}_{2,j})_{L^2},
\end{align*}
where we used $L_{\ell,ij}^*=L_{\ell,ji}$. 
Similarly, we see also that 
\begin{align*}
I_{2,2}
&= \underline{\rho}_1\underline{h}_1(\mathcal{L}_{1,0}^{(N)}\bm{\phi}_1,\bm{l}_1^{(N)}\cdot\bm{\phi}_1)_{L^2}
 + \underline{\rho}_2\underline{h}_2(\mathcal{L}_{2,0}^{(N)}\bm{\phi}_2,\bm{l}_2^{(N)}\cdot\bm{\phi}_2)_{L^2} \\
&= (\underline{h}_2\mathcal{L}_{2,0}^{(N)}\bm{\phi}_2,
 \underline{\rho}_2\bm{l}_2^{(N)}\cdot\bm{\phi}_2 - \underline{\rho}_1\bm{l}_1^{(N)}\cdot\bm{\phi}_1)_{L^2} \\
&= (\underline{h}_2\mathcal{L}_{2,0}^{(N)}\bm{\phi}_2,\phi)_{L^2} \\
&= (\underline{h}_2\mathcal{L}_{2,0}^{(N)}\bm{\phi}_2,
 \underline{\rho}_2\bm{l}_2^{(2N+2)}\cdot\tilde{\bm{\phi}}_2 - \underline{\rho}_1\bm{l}_1^{(2N+2)}\cdot\tilde{\bm{\phi}}_1)_{L^2} \\
&= \underline{\rho}_1\underline{h}_1(\mathcal{L}_{1,0}^{(N)}\bm{\phi}_1,\bm{l}_1^{(2N+2)}\cdot\tilde{\bm{\phi}}_1)_{L^2}
 + \underline{\rho}_2\underline{h}_2(\mathcal{L}_{2,0}^{(N)}\bm{\phi}_2,\bm{l}_2^{(2N+2)}\cdot\tilde{\bm{\phi}}_2)_{L^2} \\
&= \underline{\rho}_1\underline{h}_1\sum_{j=0}^{2N+2} (H_1^{2j}\mathcal{L}_{1,0}^{(N)}\bm{\phi}_1,\tilde{\phi}_{1,j})_{L^2}
 + \underline{\rho}_2\underline{h}_2\sum_{j=0}^{2N^*+2} (H_2^{p_j}\mathcal{L}_{2,0}^{(N)}\bm{\phi}_2,\tilde{\phi}_{2,j})_{L^2}. 
\end{align*}
Here, it follows from~\eqref{expression-L_k} that $H_1^{2j}\mathcal{L}_{1,0}^{(N)}\bm{\phi}_1 = \sum_{i=0}^NL_{1,ji}\phi_{1,i}$ 
and $H_2^{p_j}\mathcal{L}_{2,0}^{(N)}\bm{\phi}_1 = \sum_{i=0}^{N^*}L_{2,ji}\phi_{2,i}$ hold only for $j=0,1,\ldots,N$ and for 
$j=0,1,\ldots,N^*$, respectively. 
Therefore, we have 
\begin{align*}
I_{2,2}
&= \underline{\rho}_1\underline{h}_1 \sum_{i=0}^N\sum_{j=0}^{N} (L_{1,ji}\phi_{1,i},\tilde{\phi}_{1,j})_{L^2}
 + \underline{\rho}_2\underline{h}_2 \sum_{i=0}^{N^*}\sum_{j=0}^{N^*} (L_{2,ji}\phi_{2,i},\tilde{\phi}_{2,j})_{L^2} \\
&\quad\;
 + \underline{\rho}_1\underline{h}_1 \sum_{i=0}^N\sum_{j=N+1}^{2N+2} (H_1^{2j}L_{1,0i}\phi_{1,i},\tilde{\phi}_{1,j})_{L^2}
 + \underline{\rho}_2\underline{h}_2 \sum_{i=0}^{N^*}\sum_{j=N^*+1}^{2N^*+2} (H_2^{p_j}L_{2,0i}\phi_{2,i},\tilde{\phi}_{2,j})_{L^2},
\end{align*}
so that 
\begin{align*}
I_2
&= \underline{\rho}_1\underline{h}_1 \sum_{i=0}^N\sum_{j=N+1}^{2N+2}
  ((L_{1,ji}-H_1^{2j}L_{1,0i})\phi_{1,i},\tilde{\phi}_{1,j})_{L^2} \\
&\quad\;
 + \underline{\rho}_2\underline{h}_2 \sum_{i=0}^{N^*}\sum_{j=N^*+1}^{2N^*+2}
  ((L_{2,ji}-H_2^{p_j}L_{2,0i})\phi_{2,i},\tilde{\phi}_{2,j})_{L^2} \\
&= \underline{\rho}_1\underline{h}_1 \sum_{i=0}^N\sum_{j=N+1}^{2N+2}
  ((L_{1,ji}-H_1^{2j}L_{1,0i})(\phi_{1,i}-\tilde{\phi}_{1,i}),\tilde{\phi}_{1,j})_{L^2} \\
&\quad\;
 + \underline{\rho}_2\underline{h}_2 \sum_{i=0}^{N^*}\sum_{j=N^*+1}^{2N^*+2}
  ((L_{2,ji}-H_2^{p_j}L_{2,0i})(\phi_{2,i}-\tilde{\phi}_{2,i}),\tilde{\phi}_{2,j})_{L^2} \\
&\quad\;
 - \underline{\rho}_1\underline{h}_1 \sum_{i=N+1}^{2N+2}\sum_{j=N+1}^{2N+2}
  ((L_{1,ji}-H_1^{2j}L_{1,0i})\tilde{\phi}_{1,i},\tilde{\phi}_{1,j})_{L^2} \\
&\quad\;
 - \underline{\rho}_2\underline{h}_2 \sum_{i=N^*+1}^{2N^*+2}\sum_{j=N^*+1}^{2N^*+2}
  ((L_{2,ji}-H_2^{p_j}L_{2,0i})\tilde{\phi}_{2,i},\tilde{\phi}_{2,j})_{L^2}.
\end{align*}
Hence, denoting by $\bm{\varphi}_1=(\varphi_{1,0},\varphi_{1,1},\ldots,\varphi_{1,N})^\mathrm{T}$ and 
$\bm{\varphi}_2=(\varphi_{2,0},\varphi_{2,1},\ldots,\varphi_{2,N^*})^\mathrm{T}$ with 
$\varphi_{\ell,i}=\phi_{\ell,i}-\tilde{\phi}_{\ell,i}$ we obtain 
\begin{align*}
|I_2|
&\lesssim \sum_{\ell=1,2}\underline{\rho}_\ell\underline{h}_\ell
 ( \|\nabla\bm{\varphi}_\ell\|_{L^2}^2 + (\underline{h}_\ell\delta)^{-2}\|\bm{\varphi}_\ell'\|_{L^2}^2 ) \\
&\quad\;
 + \underline{\rho}_1\underline{h}_1
  \|(\tilde{\phi}_{1,N+1},\tilde{\phi}_{1,N+2},\ldots,\tilde{\phi}_{1,2N+2})\|_{H^1}^2 \\
&\quad\;
 + \underline{\rho}_2\underline{h}_2
  \|(\tilde{\phi}_{2,N^*+1},\tilde{\phi}_{2,N^*+2},\ldots,\tilde{\phi}_{2,2N^*+2})\|_{H^1}^2 \\
&\quad\;
 + \underline{\rho}_1\underline{h}_1(\underline{h}_1\delta)^{-2}
  \|(\tilde{\phi}_{1,N+1},\tilde{\phi}_{1,N+2},\ldots,\tilde{\phi}_{1,2N+2})\|_{L^2}^2 \\
&\quad\;
 + \underline{\rho}_2\underline{h}_2(\underline{h}_2\delta)^{-2}
 \|(\tilde{\phi}_{2,N^*+1},\tilde{\phi}_{2,N^*+2},\ldots,\tilde{\phi}_{2,2N^*+2})\|_{L^2}^2.
\end{align*}
Here, we note that $(\bm{\varphi}_1,\bm{\varphi}_2)$ satisfy 
\[
\begin{cases}
 \mathcal{L}_{1,i}^{(N)}\bm{\varphi}_1 = r_{1,i} \quad\mbox{for}\quad i=0,1,\ldots,N, \\
 \mathcal{L}_{2,i}^{(N)}\bm{\varphi}_2 = r_{2,i} \quad\mbox{for}\quad i=0,1,\ldots,N^*, \\
 \underline{h}_1\mathcal{L}_{1,0}^{(N)}\bm{\varphi}_1 + \underline{h}_2\mathcal{L}_{2,0}^{(N)}\bm{\varphi}_2
  = \nabla\cdot(\underline{h}_1\bm{r}_{3,1}+\underline{h}_2\bm{r}_{3,2}), \\
 \underline{\rho}_2\bm{l}_2^{(N)}\cdot\bm{\varphi}_2 - \underline{\rho}_1\bm{l}_1^{(N)}\cdot\bm{\varphi}_1
  = \underline{\rho}_1r_{4,1}+\underline{\rho}_2r_{4,2},
\end{cases}
\]
where 
\[
\begin{cases}
 r_{1,i} = -\sum_{j=N+1}^{2N+2}(L_{1,ij}-H^{2i}L_{1,0j})\tilde{\phi}_{1,j} \quad\mbox{for}\quad i=0,1,\ldots,N, \\
 r_{2,i} = -\sum_{j=N^*+1}^{2N^*+2}(L_{2,ij}-H^{p_i}L_{2,0j})\tilde{\phi}_{2,j} \quad\mbox{for}\quad i=0,1,\ldots,N^*, \\
 \nabla\cdot\bm{r}_{3,1} = \sum_{j=N+1}^{2N+2}L_{1,0j}\tilde{\phi}_{1,j}, \quad
 \nabla\cdot\bm{r}_{3,2} = \sum_{j=N^*+1}^{2N^*+2}L_{2,0j}\tilde{\phi}_{2,j}, \\
 r_{4,1} = \sum_{j=N+1}^{2N+2}H_1^{2j}\tilde{\phi}_{1,j}, \quad
 r_{4,2} = -\sum_{j=N^*+1}^{2N^*+2}H_2^{p_j}\tilde{\phi}_{2,j}.
\end{cases}
\]
We put $\bm{r}_1'=(0,r_{1,1},\ldots,r_{1,N})^\mathrm{T}$ and $\bm{r}_2'=(0,r_{2,1},\ldots,r_{2,N})^\mathrm{T}$. 
Then, with a suitable decomposition 
$\bm{r}_\ell=\bm{r}_\ell^\mathrm{high}+(\underline{h}_\ell\delta)^{-2}\bm{r}_\ell^\mathrm{low}$ for $\ell=1,2$, 
and using the linearity of~\eqref{equationvarphi}, we see by Lemma~\ref{L.elliptic} that 
\begin{align*}
& \sum_{\ell=1,2}\underline{\rho}_\ell\underline{h}_\ell
 ( \|\nabla\bm{\varphi}_\ell\|_{L^2}^2 + (\underline{h}_\ell\delta)^{-2}\|\bm{\varphi}_\ell'\|_{L^2}^2 ) \\
&\lesssim \sum_{\ell=1,2}\underline{\rho}_\ell\underline{h}_\ell
 ( \|\bm{r}_\ell^\mathrm{high}\|_{H^{-1}}^2 + (\underline{h}_\ell\delta)^{-2}\|\bm{r}_\ell^\mathrm{low}\|_{L^2}^2
  + \|\bm{r}_{3,\ell}\|_{L^2}^2 + \|r_{4,\ell}\|_{H^1}^2 ) \\
&\lesssim \underline{\rho}_1\underline{h}_1
\|(\tilde{\phi}_{1,N+1},\tilde{\phi}_{1,N+2},\ldots,\tilde{\phi}_{1,2N+2})\|_{H^1}^2 \\
&\quad\;
 + \underline{\rho}_2\underline{h}_2
  \|(\tilde{\phi}_{2,N^*+1},\tilde{\phi}_{2,N^*+2},\ldots,\tilde{\phi}_{2,2N^*+2})\|_{H^1}^2 \\
&\quad\;
 + \underline{\rho}_1\underline{h}_1(\underline{h}_1\delta)^{-2}
  \|(\tilde{\phi}_{1,N+1},\tilde{\phi}_{1,N+2},\ldots,\tilde{\phi}_{1,2N+2})\|_{L^2}^2 \\
&\quad\;
 + \underline{\rho}_2\underline{h}_2(\underline{h}_2\delta)^{-2}
  \|(\tilde{\phi}_{2,N^*+1},\tilde{\phi}_{2,N^*+2},\ldots,\tilde{\phi}_{2,2N^*+2})\|_{L^2}^2.
\end{align*}
Moreover, it follows from~\cite[Lemmas 5.2 and 5.4]{Iguchi2018-2} that 
\begin{align*}
& \|(\tilde{\phi}_{1,N+1},\tilde{\phi}_{1,N+2},\ldots,\tilde{\phi}_{1,2N+2})\|_{H^k}
 \lesssim (\underline{h}_1\delta)^{2N+2-k}\|\nabla\tilde{\bm{\phi}}_1\|_{H^{2N+1}} \\
& \|(\tilde{\phi}_{2,N^*+1},\tilde{\phi}_{2,N^*+2},\ldots,\tilde{\phi}_{2,2N^*+2})\|_{H^k}
 \lesssim (\underline{h}_2\delta)^{2N+2-k}(
  \|\nabla\tilde{\bm{\phi}}_2\|_{H^{2N+1}} + \|\tilde{\bm{\phi}}_2'\|_{H^{2N+1}} ),
\end{align*}
for $k=0,2$, and hence also for $k=1$ by interpolation, so that 
\begin{align*}
|I_2|
&\lesssim \sum_{\ell=1,2}\underline{\rho}_\ell\underline{h}_\ell(\underline{h}_\ell\delta)^{4N+2}
 ( \|\nabla\tilde{\bm{\phi}}_\ell\|_{H^{2N+1}}^2 + \|\tilde{\bm{\phi}}_\ell'\|_{H^{2N+1}}^2 ) \\
&\lesssim ((\underline{h}_1\delta)^{4N+2}+(\underline{h}_2\delta)^{4N+2})\|\nabla\phi\|_{H^{2N+1}}^2 \\
&\lesssim ((\underline{h}_1\delta)^{4N+2}+(\underline{h}_2\delta)^{4N+2})\|\nabla\phi\|_{H^{4N+2}}\|\nabla\phi\|_{L^2}, 
\end{align*}
where we used Lemma~\ref{L.elliptic} with~\eqref{parameter-relation}, and interpolation. 
This completes the proof of Theorem~\ref{theorem-Hamiltonian}.


\bigskip
Vincent Duch\^ene \par
{\sc Institut de Recherche Math\'ematique de Rennes} \par
{\sc Univ Rennes}, CNRS, IRMAR -- UMR 6625 \par
{\sc F-35000 Rennes, France} \par
E-mail: \texttt{vincent.duchene@univ-rennes1.fr}

\bigskip
Tatsuo Iguchi \par
{\sc Department of Mathematics} \par
{\sc Faculty of Science and Technology, Keio University} \par
{\sc 3-14-1 Hiyoshi, Kohoku-ku, Yokohama, 223-8522, Japan} \par
E-mail: \texttt{iguchi@math.keio.ac.jp}

\end{document}